\documentclass[a4paper,twoside,11pt]{article}
\usepackage{tikz-cd}
\usepackage{mathtools}
\usepackage{amssymb}
\usepackage[utf8]{inputenc}
\usepackage[english]{babel}
\usepackage[numbers]{natbib}
\usepackage{amsthm}
\usepackage{braket}

\usepackage{mathrsfs}

\newtheorem*{thm*}{Theorem}

\usepackage{fullpage}
\usepackage{amsmath}
\usepackage{amsfonts}
\usepackage{amsthm}
\usepackage{glossaries-extra}

\usepackage{fullpage}
\usepackage{amsmath}
\usepackage{amsfonts}
\usepackage{amsthm}

\newtheorem{mainthm}{Theorem}

\newtheorem{theorem}{Theorem}[section]
\newtheorem{lemma}[theorem]{Lemma}
\newtheorem{cor}[theorem]{Corollary}
\newtheorem*{cor*}{Corollary}
\newtheorem{proposition}[theorem]{Proposition}
\newtheorem*{prop*}{Proposition}

\theoremstyle{remark}
\newtheorem{remark}{Remark}

\theoremstyle{remark}

\theoremstyle{definition}
\newtheorem*{quest*}{Question}

\theoremstyle{definition}
\newtheorem{quest}{Question}

\theoremstyle{remark}

\theoremstyle{definition}
\newtheorem{definition}[theorem]{Definition}

\theoremstyle{definition}

\theoremstyle{definition}
\newtheorem*{conj*}{Conjecture}

\theoremstyle{definition}

\theoremstyle{definition}
\newtheorem*{exmp*}{Example}

\newcommand{\R}{\mathbb{R}}

\newcommand{\Z}{\mathbb{Z}}
\newcommand{\N}{\mathbb{N}}
\newcommand{\C}{\mathbb{C}}

\newcommand{\id}[1]{id_{#1}}

\newcommand{\supp}{\text{supp }}

\newcommand{\dom}{\text{dom} \;}

\newcommand{\T}{\mathbb{T}}

\newcommand{\im}{\text{im} \;}

\newcommand{\Braid}[2]{B^{#1}(#2)}
\newcommand{\cBraid}[2]{B_0^{#1}(#2)}
\newcommand{\kBraid}[1]{\Braid{k}{#1}}
\newcommand{\ckBraid}[1]{\cBraid{k}{#1}}
\newcommand{\CapBraid}[2]{\widetilde{B}_0^{#1}(#2)}
\newcommand{\kCapBraid}[1]{\CapBraid{k}{#1}}
\usepackage{hyperref}
\hypersetup{
    colorlinks=true,
    linkcolor=blue,
    }

\newcommand{\HJ}{\mathscr{H\mspace{-8mu}J}}
\newcommand{\Per}[1]{\widetilde{Per}_0({#1})}
\newcommand{\cL}[1]{\widetilde{\mathcal{L}_0}(#1)}

\renewcommand{\i}{{\mathrm{i}}}

\title{Hamiltonian Floer theory on surfaces: \\ \large Linking, positively transverse foliations and spectral invariants}
\author{Dustin Connery-Grigg \thanks{The research leading to this work received funding from the Fonds de recherche du Québec – Nature et technologies, Grant \# 260640. Writing of the article was partially supported by a postdoctoral fellowship from the \textit{Fondation Sciences Math\'{e}matiques de Paris} (FSMP)}}
\date{\today}

\begin{document}

\maketitle


\begin{abstract}
We develop connections between the qualitative dynamics of Hamiltonian isotopies on a surface $\Sigma$ and their chain-level Floer theory using ideas drawn from Hofer-Wysocki-Zehnder's theory of finite energy foliations. We associate to every collection of capped $1$-periodic orbits which is `maximally unlinked relative the Morse range' a singular foliation on $S^1 \times \Sigma$ which is positively transverse to the vector field $\partial_t \oplus X^H$ and which is assembled in a straight-forward way from the relevant Floer moduli spaces. Additionally, we provide a purely topological characterization of those Floer chains which both represent the fundamental class in $CF_*(H,J)$, and which lie in the image of some chain-level PSS map. This leads to the definition of a novel family of spectral invariants which share many of the same formal properties as the Oh-Schwarz spectral invariants, and we compute the novel spectral invariant associated to the fundamental class in entirely dynamical terms. This significantly extends a project initiated by Humili\`{e}re-Le Roux-Seyfaddini in \cite{HRS16}.
\end{abstract}

\tableofcontents

\section{Introduction}
The goal of this paper is to develop a certain picture --- the possibility of which was initially suggested by the work of Humili\`{e}re, Le Roux and Seyfaddini in \cite{HRS16} --- for Hamiltonian dynamics on surfaces via a set of tools which are well-suited to the study of the relationship between their dynamics and the structure of their Floer complexes. More precisely, recall that an isotopy
\begin{align*}
\phi^H: [0,1] \times \Sigma &\rightarrow \Sigma
\end{align*}
of a symplectic surface $(\Sigma,\omega)$ is called \textit{Hamiltonian} if it is realized as the flow of the periodic family of vector fields $(X^H_t)_{t \in S^1}$, where each vector field satisfies
\begin{align*}
\omega(X^H_t,\cdot)&=-dH_t(\cdot)
\end{align*}
for some smooth family of functions $(H_t: \Sigma \rightarrow \R )_{t \in S^1}$. Modulo some genericity conditions, to each such isotopy and any family $(J_t)_{t \in S^1}$ of $\omega$-compatible almost complex structures, we may associate a complex $CF_*(H,J)$ which is generated by the set $\Per{H}$ of capped $1$-periodic orbits of $\phi^H$. This complex has a $\Z$-grading $\mu_{CZ}$ which counts the amount of `symplectic winding' that occurs locally about the orbit throughout the isotopy, and whose differential counts the number of `Floer cylinders' $u \in \widetilde{\mathcal{M}}(\hat{x},\hat{y};H,J)$ which satisfy a certain elliptic partial differential equation and connect $\hat{x}, \hat{y} \in \Per{H}$. $CF_*(H,J)$ also carries a real-valued filtration induced by the \textit{Hamiltonian action functional} $\mathcal{A}_H$ associated to $H$ (see Section \ref{Sec:Floer} for definitions and more details about the construction of the Hamiltonian Floer complex). 
\par
Together with $\mu_{CZ}$, any collection of capped orbits $\hat{X} \subseteq \Per{H}$ such that the underlying orbits $x,y$ of any two elements $\hat{x}, \hat{y} \in \hat{X}$ are distinct forms a geometric object which we might call an \textit{indexed capped braid} $(\hat{X}, \mu_{CZ})$. We can think of the topology of the indexed capped braids which make up $\Per{H}$ as encoding the topological structure of $1$-periodic orbits of $(\phi^H_t)_{t \in [0,1]}$ along with the coarse local structure of the isotopy near these orbits. The broad question broached by this work may be stated as follows
\begin{quote}
What relations does the topological structure of $\Per{H}$ impose on the (filtered) algebraic structure of $CF_*(H,J)$ and vice versa?
\end{quote}  
Clearly, this question is only particularly meaningful in low-dimensions where the orbits may twist about one another in homotopically interesting ways, but in this situation it proves surprisingly fruitful and sheds considerable light on the possible dynamics of generic Hamiltonian isotopies on surfaces. 
\par
To wit, we will say that a capped braid $\hat{X}$ is \textit{unlinked} if the capping disks of the strands may be chosen such that their graphs in $D^2 \times \Sigma$ are disjoint. $\hat{X}$ will be said to be \textit{positive} (resp. \textit{negative}) if the capping disks may be chosen such that their graphs are transverse and have only positive (resp. negative) intersections. An indexed capped braid $\hat{X} \subseteq \Per{H}$ is \textit{maximally unlinked relative the Morse range}, denoted $\hat{X} \in murm(H)$, if $\hat{X}$ is unlinked, every (capped) strand of $\hat{X}$ has index lying in the set $\lbrace -1, 0, 1 \rbrace$ and $\hat{X}$ is maximal among all collections of capped orbits in $\Per{H}$ having these two properties. Similarly, $\hat{X} \subseteq \Per{H}$ is \textit{maximally positive relative index $1$}, denoted $\hat{X} \in mp_{(1)}$, if $\hat{X}$ is positive, every (capped) strand of $\hat{X}$ has index equal to $1$, and $\hat{X}$ is maximal among all collections of capped orbits having these two properties. 
\par
As a consequence of the theory developed herein, we obtain the following structural result for generic Hamiltonians $H$ on an arbitrary closed symplectic surface, which provides a topological interpretation of behaviour of the Hamiltonian Floer complex $CF_*(H,J)$ in the `Morse range', when $* \in \lbrace -1,0,1 \rbrace$.
\begin{mainthm}\label{MainThm: Foliation thm}
Let $(\Sigma,\omega)$ be a closed symplectic surface. $H \in C^\infty(S^1 \times \Sigma)$ be a non-degenerate Hamiltonian, and let $J \in C^\infty(S^1; \mathcal{J}_\omega(\Sigma))$ be such that $(H,J)$ is Floer regular. For any capped braid $\hat{X} \in murm(H)$, we may construct an oriented singular foliation $\mathcal{F}^{\hat{X}}$ of $S^1 \times \Sigma$ with the following properties
\begin{enumerate}
\item The singular leaves of $\mathcal{F}^{\hat{X}}$ are precisely the graphs of the orbits in $\hat{X}$. That is, they are  parametrized by the maps $t \mapsto (t,x(t))$ for $\hat{x}=[x,w_x] \in \hat{X}$.
\item The regular leaves are precisely the annuli parametrized by maps
\begin{align*}
\check{u}: \R \times S^1 &\rightarrow S^1 \times \Sigma \\
(s,t) &\mapsto (t,u(s,t)).
\end{align*}
for some $u \in \widetilde{\mathcal{M}}(\hat{x},\hat{y};H,J)$, and some $\hat{x}, \hat{y} \in \hat{X}$.
\item The vector field $\check{X}^H(t,z)=\partial_t \oplus X^H_t(z)$ is \textit{positively transverse} to every regular leaf of $\mathcal{F}^{\hat{X}}$. That is, $\check{X}^H \wedge \check{u}_*(\partial_s \wedge \partial_t)$ is a positively oriented volume element on each regular leaf, where $\check{u}$ is as above.
\end{enumerate}
\end{mainthm} 
It turns out (see Corollary \ref{KlinkingLemma}) that if $\mathcal{M}(\hat{x},\hat{y};H,J) \neq \emptyset$ with $\mu_{CZ}(\hat{x}),\mu_{CZ}(\hat{y}) \in \lbrace -1,0,1 \rbrace$ then $\hat{x}$ and $\hat{y}$ are unlinked and so are both contained in some $\hat{X} \in murm(H)$. Thus by allowing $\hat{X}$ to vary in $murm(H)$, the foliations $\mathcal{F}^{\hat{X}}$ above capture all of the Floer cylinders running between capped orbits in the Morse range, and so may be thought of as providing a geometric picture of the structure of the Floer complex $CF_*(H,J)$ in this range.
\par
The foliations $\mathcal{F}^{\hat{X}}$ arise as the projections to $S^1 \times \Sigma$ of certain (non-singular) `$\R$-invariant finite energy foliations' $\widetilde{\mathcal{F}}^{\hat{X}}$ on $\R \times S^1 \times \Sigma$. Finite energy foliations were initially introduced in contact geometry by Hofer, Wysocki and Zehnder in \cite{HWZ03} (although see also \cite{Ho93} and \cite{HWZ98} for earlier appearances of special cases of these objects) where the authors established the existence of an $\R$-invariant finite energy foliation for any generic Reeb vector field associated to the standard contact structure on $S^3$, and used these to show that such Reeb flows necessarily possess either $2$ or infinitely many periodic orbits. Such foliations were later imported into the study of Hamiltonian dynamics by Bramham in his PhD thesis \cite{Br08} in which he constructed an $\R$-invariant finite energy foliation for any smooth area-preserving disk map that is an irrational rotation on the boundary. Bramham subsequently used such foliations as a crucial tool in his recent celebrated study of irrational pseudo-rotations of the disk in \cite{Br15} and \cite{Br15b} in the context of long-standing questions about low-dimensional conservative systems with vanishing entropy. Theorem \ref{MainThm: Foliation thm} may thus be seen as an extension of Bramham's existence result to generic Hamiltonian systems on arbitrary closed symplectic surfaces, with additional controls over the various collections of orbits which may --- and do --- appear as the singular leaves of the foliations $\mathcal{F}^{\hat{X}}$.
\par 
Colin, Dehornoy and Rechtman have also recently established in \cite{CDR23} --- for generic Reeb vector fields on contact $3$-manifolds --- the existence of certain singular foliations to which the dynamics are positively transverse and which they call \textit{broken book decompositions}, which they then use to study the Reeb dynamics to great effect. If one works in the setting of stable Hamiltonian structures, which provides a common generalization of Reeb flows and mapping tori of Hamiltonian isotopies, then the foliations $\mathcal{F}^{\hat{X}}$ are broken book decompositions in the sense of \cite{CDR23} with binding given by the graphs of the orbits in $\hat{X}$. In contrast with the foliations $\mathcal{F}^{\hat{X}}$ however, the foliations in \cite{CDR23} do not arise directly as projections of finite energy foliations from the symplectization; instead, collections of pseudoholomorphic curves in the symplectization are projected to surfaces in the contact manifold which may self-intersect and may fail to be embedded at certain points, and the authors then apply a desingularization procedure and some additional arguments in order to obtain finitely many disjoint surfaces which may then be extended to a broken book decomposition by flowing these surfaces along the Reeb flow. This has the effect of making the relationship between the initial holomorphic curves and the resulting foliation somewhat subtle.
\par
The structure of the foliations $\mathcal{F}^{\hat{X}}$ could \textit{a priori} be rather complicated, however, it turns out that this is not the case. Write $\mathcal{F}^{\hat{X}}_t$ for the singular foliation on $\Sigma$ given by intersecting $\mathcal{F}^{\hat{X}}$ with the fiber $\lbrace t \rbrace \times \Sigma$ over $t \in S^1$. Sliding the fiber $\lbrace 0 \rbrace \times \Sigma$ along the circles $t \mapsto (t,u_s(t))$ provides a loop of diffeomorphisms $(\psi^{\hat{X}}_t)_{t \in S^1}$ such that $\psi^{\hat{X}}_t$ sends $\mathcal{F}^{\hat{X}}_0$ to $\mathcal{F}^{\hat{X}}_t$, and we prove the following:   
\begin{mainthm}\label{MainThm2}
For every $t \in S^1$, $\mathcal{F}^{\hat{X}}_t$ is a singular foliation of Morse type (ie. there is a choice of Morse function and metric on $\Sigma$ for which $\mathcal{F}_t^{\hat{X}}$ is the singular foliation associated to the negative gradient flow). Moreover, the loop $(\psi^{\hat{X}}_t)_{t \in S^1}$ is a contractible loop of diffeomorphisms such that the orbits of $(\psi^{\hat{X}})^{-1} \circ \phi^H$ are positively transverse to the foliation $\mathcal{F}^{\hat{X}}_0$.
\end{mainthm}    
We may thereby reduce the study of the qualitative dynamics of the isotopy $\phi^H$ to the better understood situation of dynamics which are positively transverse to a Morse-type foliation.
\par
Theorems \ref{MainThm: Foliation thm} and \ref{MainThm2} also serve to clarify a certain meta-mathematical question that has arisen with respect to the study of Hamiltonian systems on surfaces. In \cite{LeC05} (see also \cite{LeC91}), Le Calvez developed an approach to studying \textit{homeomorphisms} of surfaces via a theory of \textit{homotopically transverse foliations}\footnote{NB. The addition of the adjective `homotopically' here does not appear in Le Calvez's work, nor in works using his theory, but is being made by the present author in order to distinguish this notion from the stronger notion of transversality referred to in Theorem \ref{MainThm2}}. This theory has provided powerful tools and perspective for the study of surface homeomorphisms (see for instance \cite{LeC06}, \cite{LT18}, \cite{KT14}, and \cite{LSV21}). In their ground-breaking study of spectral invariants for autonomous Hamiltonians on surfaces in \cite{HRS16}, Humili\`{e}re-Le Roux-Seyfaddini frame their work as being inspired by their desire to understand the link between Le Calvez's theory and Hamiltonian Floer theory, and as an attempt to recover the Oh-Schwarz spectral invariant associated to the fundamental class via the techniques of transverse foliations. Therein, the authors explicitly raise the question of understanding the relationship between Le Calvez's theory and Hamiltonian Floer theory, noting that on their domain of common overlap --- Hamiltonian systems on surfaces --- the two theories seem to be equivalent in the sense that ``...much of what can be done via one theory can also be achieved via the other''. One of the consequences of Theorems \ref{MainThm: Foliation thm} and \ref{MainThm2} is that the foliations $\mathcal{F}^{\hat{X}}_0$ are homotopically transverse foliations, and so we obtain a direct Floer-theoretic construction of certain of the foliations appearing in Le Calvez's theory. In fact, it turns out that the foliations $\mathcal{F}^{\hat{X}}_0$ recover precisely the foliations corresponding to a natural class of maximal isotopies which are called \textit{torsion-low isotopies} in Le Calvez's theory, and which were initially introduced by Yan in \cite{Ya18}. This provides us with a correspondence between the two theories for non-degenerate Hamiltonian systems.  For a more thorough discussion of these matters, see Section \ref{Sec: Calvez}. 
\par
The techniques developed in this paper also enable a second, logically distinct --- though thematically related --- line of enquiry, examining the relationship between the qualitative dynamics of a Hamiltonian isotopy and those aspects of the filtered Floer complex which are `probeable' by \textit{chain-level PSS maps}. Heuristically, one may think of such chain maps as chain-level Floer continuation maps from some very small autonomous Morse Hamiltonian. To be more precise, recall that there is a natural way to identify the quantum homology of a symplectic manifold with the Floer homology of a generic Floer pair $(H,J)$ given by the \textit{PSS isomorphism} originally introduced in \cite{PSS96} (see also \cite{Sc95}). This isomorphism is induced on homology by choosing some ancillary data $\mathcal{D}$ --- which includes a Morse-Smale pair $(f,g)$ and so provides a Morse-theoretic model for the quantum chain complex $QC_*(f,g) = C^{Morse}(f,g) \otimes \Lambda_{\omega}$ whose homology computes the quantum homology of $(M,\omega)$ --- and this data is then used to construct a chain morphism
\begin{align*}
\Phi^{PSS}_{\mathcal{D}}: QC_*(f,g) &\mapsto CF_{*-n}(H,J).
\end{align*}
The induced map on homology is an isomorphism and is independent of the ancillary data $\mathcal{D}$. 
\par
Our next result provides, in the case of surfaces, a purely topological characterization of those non-trivial Floer cycles in $CF_*(H,J)$ which lie in the image of some chain-level PSS map and which represent the fundamental class.\footnote{Our methods also permit a similar type of characterization to be given for cycles which lie in the image of some chain-level PSS map and which represent the point class, although this characterization is somewhat more involved and we will not present it in this article. Interestingly, however, there seem to be fundamental obstructions to using the same approach to characterize Floer cycles of middle degree lying in the image of some chain-level PSS map.} Note that for $\sigma= \sum a_{\hat{x}} \hat{x} \in CF_*(H,J)$, $\supp \hat{\sigma} := \lbrace \hat{x}: \; a_{\hat{x}} \neq 0 \rbrace$ may be thought of as a capped braid.
\begin{mainthm}\label{MainThm: Top Char}
Let $\sigma \in CF_1(H,J)$. $\sigma$ is a non-trivial cycle such that $\sigma \in \im \Phi^{PSS}_{\mathcal{D}}$ for some regular PSS data $\mathcal{D}$ if and only if $\supp \sigma$ is a maximal positive capped braid relative index $1$.
\end{mainthm}
This result is something of a novelty, since in general it is very difficult to identify candidates for cycles which represent the fundamental class in Floer homology.
\par
Theorem \ref{MainThm: Top Char} motivates us to examine --- for general symplectic manifolds --- the quantity $c_{im}(\alpha;H)$ obtained by examining the infimal action level required to represent some non-zero quantum homology class $\alpha$ in the filtered Floer complex of $H$ via some chain-level PSS map. It turns out that these quantities define spectral invariants which are intimately related to the geometry of Hamiltonian fibrations over disks and which we term the \textit{PSS-image spectral invariants}. Theorem \ref{MainThm: Top Char} thereby provides us with a purely topological formula for $c_{im}([\Sigma];H)$ on arbitrary surfaces.
\begin{cor*}
Let $H$ be non-degenerate.
\begin{align*}
c_{im}([\Sigma];H)&= \min_{\hat{X} \in mp_{(1)}(H)} \max_{\hat{x} \in \hat{X}} \mathcal{A}_H(\hat{x}).
\end{align*}
\end{cor*}
The above result is quite similar in spirit to the characterization of the usual Oh-Schwarz spectral invariant associated to the fundamental class (see Section \ref{Sec: PSS defn Sect} for the definition) established by Humili\`{e}re-Le Roux-Seyfaddini in \cite{HRS16} for autonomous Hamiltonians on aspherical surfaces and, in fact, we will see in Section \ref{Sec: On equiv of OS and im} that their work implies that the two quantities agree in this setting. 
\par
In virtue of their definition, the PSS-image spectral invariants share many of the desirable formal properties of the usual Oh-Schwarz spectral invariants (see Section \ref{Sec: PSS defn Sect} for the definition), and so one can perform many of the same arguments with them, but with the added advantage that we have a better understanding of the relationship between PSS-image spectral invariants and the dynamical behaviour of the system under study --- at least on surfaces. For instance, we may use the PSS-image spectral invariants to define a conjugation invariant norm on $Ham(M,\omega)$ by
\begin{align*}
\gamma_{im}(\phi) &:= \inf \; c_{im}([M];H) + c_{im}([M]; \bar{H}),
\end{align*}
where the infimum is taken over Hamiltonians $H$ such that $\phi^H_1=\phi$. In the case of surfaces, the fact that the PSS-image spectral invariants share enough of the formal properties satisfied by the Oh-Schwarz spectral invariants allows Seyfaddini's argument establishing the $C^0$-continuity of the Oh-Schwarz spectral norm on surfaces (see Theorem 3 in \cite{Se13}) to carry through for $\gamma_{im}$. As a consequence, we obtain:
\begin{mainthm}\label{Mainthm-Norm}
On surfaces, the conjugation invariant norm $\gamma_{im}$ is both $C^0$-continuous and Hofer-continuous. Moreover, if $\phi$ is non-degenerate and $\Sigma \neq S^2$, then
\begin{align*}
\gamma_{im}(\phi)&= \min_{\hat{X} \in mp_{(1)}(H)} \max_{\hat{x} \in \hat{X}} \mathcal{A}_H(\hat{x}) - \max_{\hat{X} \in mn_{(-1)}(H)} \min_{\hat{x} \in \hat{X}} \mathcal{A}_H(\hat{x}),
\end{align*}
for $H$ any Hamiltonian such that $\phi^H_1=\phi$.\footnote{There is of course a similar dynamical formula for $\gamma_{im}$ on the sphere which is slightly more involved, taking into account the non-triviality of $\pi_1(Ham(S^2))$. See Corollary \ref{Cor: Sphere norm}} (Here $mn_{(-1)}(H)$ denotes the set of all capped braids $\hat{X} \subseteq \Per{H}$ which are `maximally negative relative index $-1$'. See Section \ref{Sec: Computing spec invar on surfaces} for the definition). 
\end{mainthm} 
The main interest in the above theorem, apart from the intrinsic interest of having a dynamically defined norm on Hamiltonian diffeomorphisms which is invariant under symplectic conjugation, is that the structure of the sets $mp_{(1)}(H)$ and $mn_{(-1)}(H)$ can change dramatically as one interpolates between Hamiltonians (in either the $C^0$ or the Hofer topology). Be that as it may, the above theorem guarantees that the quantity obtained by the specified mini-max procedure over these collections remains continuous nonetheless. 
\par
Finally, as a last example of the sort of information that can be drawn from Theorem \ref{MainThm: Top Char}, we may exploit the computability of our new-found spectral invariants together with their close relationship to the Oh-Schwarz spectral invariants, in order to extract the following dynamical controls over the commutator lengths of homotopy classes $\tilde{\phi} \in \widetilde{Ham}(S^2,\omega)$ from Entov's work in \cite{En04}. 
\begin{mainthm}\label{Mainthm-Entov}
Assume that $H \in C^{\infty}(S^1 \times S^2)$ is non-degenerate and normalized so that $\int H_t \omega =0$ for all $t \in S^1$, then
\begin{align*}
\min \big\lbrace  \min_{\hat{X} \in mp_{(1)}(H)} \max_{\hat{x} \in \hat{X}} \mathcal{A}_{H}(\hat{x}), -\max_{\hat{X} \in mn_{(-1)}(H)} \min_{\hat{x} \in \hat{X}} \mathcal{A}_H(\hat{x}) \big\rbrace &< -k Area(S^2,\omega)
\end{align*}
only if the commutator length of $\tilde{\phi}^H$ in $\widetilde{Ham}(S^2)$ is strictly greater than $2k+1$.
\end{mainthm}

\subsection{Structure of this paper}
Sections \ref{Braids}, \ref{Local} and \ref{Ch: Construct chain maps} develop the main conceptual and technical tools in this work, which we then apply in Sections \ref{Ch: App 1} and \ref{Sec: Pos Transverse foliations} to deduce the results laid out in the previous section. These latter two sections are independent and may be read in any order, but both depend heavily on the first three sections.
\par
Section \ref{Braids} introduces the basic notion of \textit{capped braids} and related concepts, the most important of which is the \textit{homological linking number} of two capped braids. Section \ref{Local} presents the elements of Floer theory of which we will have need, along with the work of Hofer-Wysocki-Zehnder and Siefring on the relative asymptotic behaviour of pseudoholomorphic cylinders, and explains the relevance of this and the homological linking number to the study of the Floer theory of low-dimensional systems. Section \ref{Ch: Construct chain maps} is the last of the technical sections, and introduces our technique for designing Floer continuation maps such that certain moduli spaces are non-empty. In Section \ref{Ch: App 1}, we introduce the PSS-image spectral invariants and prove Theorem \ref{MainThm: Top Char} and its various consequences. Section \ref{Sec: Pos Transverse foliations} proves Theorems \ref{MainThm: Foliation thm} and \ref{MainThm2}, and discusses some dynamical consequences as well as the relationship to Le Calvez's theory of transverse foliations. 

\subsection{Acknowledgments}
The author would like to thank his PhD supervisor Fran\c{c}ois Lalonde for his enthusiastic support and encouragement throughout this project. He would also like to thank Egor Shelukhin for his helpful comments and feedback on an earlier version of this paper, Patrice Le Calvez for enlightening discussions about his theory on transverse foliations, as well as Jordan Payette for many stimulating conversations about the subject matter. Finally, particular thanks are due to Vincent Humili\`{e}re, Sobhan Seyfaddini and Fr\'{e}d\'{e}ric Le Roux for their remarks on an earlier version of this paper which identified a significant error in a purported computation of the Oh-Schwarz spectral invariant. Additional thanks are due to Vincent Humili\`ere for directing my attention to Yan's work on torsion-low isotopies in \cite{Ya18} and suggesting their relationship to the foliations $\mathcal{F}^{\hat{X}}_0$ constructed in Section \ref{Sec: Pos Transverse foliations}. The author warmly thanks the anonymous reviewers for their careful reading of the manuscript and their many insightful comments and suggestions.

\section{Capped braids and homological linking}\label{Braids}
In this section we introduce the basic notion of \textit{capped braids}, their appropriate notion of equivalence (\textit{$0$-homotopy}), and an important relative invariant of a pair of capped braids with the same number of strands which we call the \textit{homological linking number}. These notions will later serve to encode the topology of the capped $1$-periodic orbits of a non-degenerate Hamiltonian system, and give controls on the behaviour of Floer-type cylinders which may run between various collections of such capped orbits.
\par
The basic definitions required for our work with capped braids are presented in Section \ref{Sec: Capped Braids}. Section \ref{LocalLinkSec} explains a mild generalization of the classical winding number of loops in $\R^2$ to capped loops in a surface whose underlying loops are close in the loop space. This generalization becomes important later in the work when we consider the asymptotics of Floer-type cylinders which emerge or converge to the same orbit. Section \ref{HomLinkSec} introduces the homological linking number and establishes its basic properties including how it recovers the generalized winding number in the setting of Section \ref{HomLinkSec}.   
\par
For the duration of this paper, $\Sigma$ will always denote a closed, connected and smooth symplectic surface $(\Sigma,\omega)$, $\mathcal{L}(M)$ the space of smooth parametrized loops in the manifold $M$, $\mathcal{L}_0(M)$ its space of contractible loops and $\cL{M}$ its Novikov covering space (see \cite{McSa12}, Section 12.1). 
\par
Throughout this work, there is a certain amount of juggling of different perspectives on the same objects that will be necessary. In particular, though our main objects of study are isotopies on some surface $\Sigma$ --- and so the initial arena in which the action takes place is $2$ dimensional --- it will frequently be useful to work on the $3$-dimensional mapping torus
\begin{align*}
\check{\Sigma}&:= S^1 \times \Sigma
\end{align*}
as well as the $4$-dimensional space 
\begin{align*}
\widetilde{\Sigma}&:= \R \times S^1 \times \Sigma.
\end{align*}
In order to have our notation be suggestive of these perspectival shifts, maps taking values in $\check{\Sigma}$ will frequently be adorned with a $\check{\cdot}$, while maps taking values in $\widetilde{\Sigma}$ will be adorned with a $\tilde{\cdot}$. The most frequent use-cases for this notation will be the following. For $x \in \mathcal{L}(\Sigma)$, we write $\check{x}(t):=(t,x(t))$ for its graph in $\check{\Sigma}=S^1 \times \Sigma$. For $u :I \times S^1 \rightarrow \Sigma$, where $I \subseteq \R$, we write $\tilde{u}(s,t):=(s,t,u(s,t))$ for its graph in $\widetilde{\Sigma}=I \times S^1 \times \Sigma$ and $\check{u}(s,t):=(t,u(s,t))$ for the projection of this graph onto $\check{\Sigma}$.  
\subsection{Capped Braids}\label{Sec: Capped Braids}
\begin{definition}
For any $k \in \N$, we define the \textbf{$k$-configuration space} 
\begin{align*}
C_k(\Sigma)&:= \lbrace (z_1, \ldots, z_k) \in \Sigma^k: (i \neq j) \Rightarrow z_i \neq z_j \rbrace
\end{align*}
\end{definition}

\begin{definition}
An \textbf{(ordered) $k$-braid} is an element $X=(x_1,\ldots,x_k) \in \mathcal{L}(C_k(\Sigma))$.  Denote by $\kBraid{\Sigma}$ the space of ordered $k$-braids.
The loop  $x_i$ is called the \textbf{$i$-th strand of $X$}, for $i=1, \ldots, k$.
\end{definition}
\begin{definition}
An \textbf{unordered $k$-braid} is an element $[X] \in \mathcal{L}(C_k(\Sigma))/S_k$, where $S_k$ acts  by permutation of coordinates. Such unordered braids may be identified with certain finite subsets of $\mathcal{L}(\Sigma)$.
\end{definition}

\begin{remark}\label{rem:OrderedUnorderedDontMatter}
We raise the distinction between ordered and unordered braids here mainly to flag for the reader that we will make no real effort outside of this section to separate these two concepts. In particular, we will routinely treat ordered braids as finite subsets of $\mathcal{L}(\Sigma)$ and perform set-wise operations on them, when properly speaking we should be speaking of the unordered braids which they represent. We will moreover speak simply of `braids' relying on the context to make clear whether these braids are ordered or unordered. For the remainder of this section, we will make a clear distinction between ordered and unordered braids, mainly to convince the suspicious reader that nothing essential is lost in making this elision.
\end{remark}

\begin{definition}
The \textbf{graph} $\check{X}$ of an (ordered) $k$-braid is the set-valued map\footnote{Note that on this definition, the graph of an ordered braid forgets the information about the ordering. One might reasonably object to this and prefer a definition which keeps track of orderings, but we will have no need for that in this article.}
$\check{X}(t)=\sqcup_{i=1}^k \check{x}_i(t) \subseteq S^1 \times \Sigma$, $t \in S^1$. The graph of an unordered braid $[X]$ is the graph of some (hence every) representative $X$ of $[X]$.
\end{definition}

\begin{definition}
An ordered $k'$-braid $Y \in  \Braid{k'}{\Sigma}$ is an ordered \textbf{sub-braid} of $X \in \Braid{k}{\Sigma}$ if $Y \subseteq X$, as an ordered set. An unordered braid $[Y]$ is a sub-braid of $[X]$ if $[Y] \subseteq [X]$ as sets.  There is an obvious partial ordering on the collection of sub-braids of $X$ (resp. of $[X]$) given by inclusion.
\end{definition}

\begin{definition}
$X \in \kBraid{\Sigma}$ is \textbf{contractible} if each strand of $X$ is a contractible loop. We write $\ckBraid{\Sigma}$ for the space of contractible ordered $k$-braids. $[X] \in  \kBraid{\Sigma}/S_k$ is contractible if some (hence every) representative $X$ of $[X]$ is contractible.
\end{definition}

\begin{definition}
A continuous map $h: [0,1] \rightarrow \kBraid{\Sigma}$ with $h(0)=X$, $h(1)=Y$ is a \textbf{braid homotopy from $X$ to $Y$}. When such a map exists, we shall say that $X$ and $Y$ are \textbf{braid homotopic}, denoted $X \simeq Y$. The map $(s,t) \mapsto h_i(s,t)$ is called the \textbf{$i$-th strand of $h$}. To any braid homotopy, we associate its \textbf{graph} $\tilde{h}(s,t)=\sqcup_{i=1}^k \tilde{h}_i(s,t) \subseteq [0,1] \times S^1 \times \Sigma$, $(s,t) \in [0,1] \times S^1$.
\end{definition}

\begin{definition}
An ordered braid $X$ will be said to be \textbf{trivial} if all of its strands are constant maps. We will sometimes write $0 \in \kBraid{\Sigma}$ to stand for some fixed but arbitrary trivial braid, when the particular choices of the constant maps are unimportant. An unordered braid $[X]$ is trivial if some (hence every) ordered representative is trivial.
\end{definition}

\begin{definition}
$X \in \kBraid{\Sigma}$ is \textbf{unlinked} if $X \simeq 0$. An ordered braid is \textbf{linked} if it is not unlinked. An unordered braid is unlinked (resp. linked) if some, hence every, ordered representative is unlinked (resp. linked).
\end{definition}


\begin{definition}
A continuous map $h: [0,1] \rightarrow \mathcal{L}(\Sigma)^k$
with $h(0)=X \in \kBraid{\Sigma}$, $h(1)=Y \in \kBraid{\Sigma}$ will be called a \textbf{braid cobordism} if there exists some $\delta >0$ such that $h(s) \in \kBraid{\Sigma}, \; \forall s \in (0, \delta) \cup (1-\delta,1)$.
\end{definition}

\begin{remark}
\begin{enumerate}
\item Note that any two $k$-braids $X, Y \in \kBraid{\Sigma}$ such that the $i$-th strand of $x$ is homotopic to the $i$-th strand of $Y$ for $i=1,\ldots,k$ are connected by a braid cobordism. Consequently, we are primarily interested in properties of the cobordisms themselves, rather than the fact of their existence. 
\item We will frequently find ourselves concerned with maps $h: I \rightarrow \mathcal{L}(\Sigma)^k$, where $I= \R$ or $I=[a,b]$ for some $a,b \in \R$, and in the case that $I= \R$, it will always be the case that $h$ extends continuously to a map $\bar{\R} \rightarrow  \mathcal{L}(\Sigma)^k$ such that on some neighbourhood of $\pm \infty$, the graphs of the strands of $h$ do not intersect. In such a case, we will speak freely of `the' braid cobordism \textbf{induced} by $h$, which is simply any braid cobordism $h \circ \varphi$, where $\varphi: \bar{I} \rightarrow [0,1]$ is any orientation-preserving diffeomorphism.
\end{enumerate}
\end{remark}

\begin{definition}\label{Def: Positive Cobordism}
Let $h: [0,1] \rightarrow \mathcal{L}(\Sigma)^k$, $k \in \N$, be a braid cobordism. We will say that $h$ is a \textbf{positive} (resp. \textbf{negative}) cobordism if for $1 \leq i < j \leq k$, the graphs $\tilde{h}_i, \tilde{h}_j: [0,1] \times S^1 \rightarrow [0,1] \times S^1 \times \Sigma$ are transverse, and every intersection is positive (resp. negative).
\end{definition}

\begin{definition}
An (ordered) \textbf{capped $k$-braid} $\hat{X}$ is an equivalence class $[X,\vec{w}]$ where $X \in \ckBraid{\Sigma}$ and $\vec{w}=(w_1, \cdots, w_k)$ with $w_i: D^2 \rightarrow \Sigma$ a capping disk for the $i$-th strand of $X$, subject to the equivalence relation $[X,\vec{w}] \sim [X', \vec{w}']$ if and only if $X=X'$ and $[w_i] \# (-[w_i']) =0 \in \pi_2(\Sigma)$ for each $i=1, \cdots, k$. The space of ordered capped $k$-braids is denoted by $\kCapBraid{\Sigma}$. The capped loop $\hat{x}_i=[x_i,w_i] \in \cL{\Sigma}$ is called the \textbf{$i$-th strand} of $\hat{X}$. The notion of capped sub-braids $\hat{Y} \subseteq \hat{X}$ is defined in the obvious way. 
\end{definition}
\begin{remark}
The distinction between ordered capped braids and unordered capped braids obtains here as well, and we adopt parallel conventions as those discussed in the case of braids in Remark \ref{rem:OrderedUnorderedDontMatter}.
\end{remark}
$\pi_2(\Sigma)^k$ acts on $\kCapBraid{\Sigma}$ by the obvious `gluing of spheres':
\begin{align*}
(A_1,\ldots,A_k) \cdot ([x_1,w_1], \ldots, [x_k,w_k]) &= ([x_1,A_1 \# w_1], \ldots, [x_k,A_k \# w_k]),
\end{align*}
where here we abuse notation slightly by thinking of $A_i \in \pi_2(\Sigma, x_i(0))$ as being both a homotopy class of maps, as well as a particular choice of a representative from that class. This action does not descend to an action on $\kCapBraid{\Sigma}/S_k$. However, if we denote by $Fix_{S_k}(\pi_2(\Sigma)^k) \simeq \pi_2(\Sigma)$ the set of fixed points of the action of the symmetric group on $\pi_2(\Sigma)^k$ by permutation of coordinates, we obtain a well-defined induced action by $Fix_{S_k}(\pi_2(\Sigma)^k)$ on unordered braids given by $(A,\ldots,A) \cdot [\hat{X}] = [(A,\ldots,A) \cdot \hat{X}]$, for $A \in \pi_2(\Sigma)$.
\begin{definition}
A trivial braid $0 \in \kBraid{\Sigma}$ has a naturally associated capping $\hat{0} \in \kCapBraid{\Sigma}$ given by capping each strand of $0$ with the constant capping. We call any such braid a \textbf{trivial capped braid}. When the particular components of a trivial capped braid are unimportant, we denote some fixed but arbitrary capped braid by the symbol $\hat{0}$. An unordered capped braid is said to be \textbf{trivial} if some (hence every) ordered representative is trivial.   
\end{definition}

\begin{definition}
For $A=(A_1, \cdots, A_k) \in \pi_2(\Sigma)^k$, an ordered braid cobordism $h$ from $X$ to $Y$ will be called an \textbf{$A$-cobordism from $[X,\vec{w}]$ to $[Y, \vec{v}]$} if $[w_i] \# [h_i] \# (-[ v_i]) = A_i$, for all $i=1, \cdots, k$. Whenever such a map exists, $[X, \vec{w}]$ and $[Y, \vec{v}]$ will be said to be \textbf{$A$-cobordant}. This notion descends to unordered capped braids provided that $A \in Fix_{S_k}(\pi_2(\Sigma)^k)$.
\end{definition}

\begin{definition}
An $A$-cobordism $h$ from $\hat{X}=[X,\vec{w}]$ to $\hat{Y}=[Y,\vec{v}]$ is called an $A$-homotopy if $h$ is in addition a braid homotopy from $X$ to $Y$. In such a situation, we will say that $\hat{X}$ and $\hat{Y}$ are \textbf{$A$-homotopic}, and we will denote the relation by $\hat{X} \simeq_A \hat{Y}$. This notion descends to unordered capped braids, provided that $A \in Fix_{S_k}(\pi_2(\Sigma))$.
\end{definition}

\begin{remark}
In this work, outside of this section, we will exclusively be concerned $A$-homotopies and $A$-cobordisms with the case where $A=\vec{0}=(0,\ldots,0) \in \pi_2(\Sigma)^k$. In this case we will speak of $0$-homotopies and write $\hat{X} \simeq_0 \hat{Y}$ when $\hat{X}$ and $\hat{Y}$ are $0$-homotopic.
\end{remark}

\begin{definition}\label{def:NaturalCapping}
If $u : [0,1] \rightarrow \mathcal{L}_0(\Sigma)$ is a homotopy from $x$ to $y$, then for any choice of cappings $\hat{x}=[x,w_x]$ and $\hat{y}=[y,w_y]$, $u$ is an $A$-cobordism from $\hat{x}$ to $\hat{y}$, for $A= [w_x] \# [u] \# (-[w_y]) \in \pi_2(\Sigma)$.
Moreover, for any $s \in [0,1]$, there are two natural choices of cappings for the loop $u_s=u(s) \in \mathcal{L}_0(\Sigma)$. Namely, if we write $\alpha^s(\tau,t):= u(s \cdot \tau,t)$ and $\beta(\tau,t):=u(1-(1-s) \cdot \tau,t)$ for $s \in [0,1]$, then we may associate to $u_s$ either of the cappings $[u_s,w_x \# \alpha^s]$ or $[u_s,w_y \# \beta^s]$, and these two cappings are obviously related by $A \cdot [u_s,w_y \# \beta^s]=[u_s,w_x \# \alpha^s]$. Consequently, if $u$ is a $0$-homotopy between $\hat{x}$ and $\hat{y}$, these two cappings agree and we may associate a unique capping
\begin{align*}
\hat{u}_s:=[u_s, w_x \# \alpha^s]=[u_s, w_y \# \beta^s]
\end{align*}
to each $u_s$ in this case. We will call such a capping \textbf{the natural capping} of $u_s$ whenever $u$ is such a $0$-homotopy.
\end{definition}

\begin{definition}
$\hat{X} \in \kCapBraid{\Sigma}$ is \textbf{unlinked} if $\hat{X} \simeq_0 \hat{0}$. An ordered capped braid is \textbf{linked} if it is not unlinked. An unordered capped braid is unlinked (resp. linked) if some, hence every, ordered representative is unlinked (resp. linked).
\end{definition}

\begin{definition}\label{Def-PositiveBraid}
$\hat{X} \in \kCapBraid{\Sigma}$ is said to be \textbf{positive} (resp. \textbf{negative}) if there exists a positive $0$-cobordism from some, hence any, trivial capped braid $\hat{0}$ to $\hat{X}$. An unordered capped braid is positive (resp. negative) if some, hence every, ordered representative is positive (resp. negative).
\end{definition}

\subsection{Linking of capped loops with close strands}\label{LocalLinkSec}
For use in the sequel, we explain here a minor adaptation of the classical linking number of two loops in the plane $x,y: S^1 \rightarrow \R^2$ to two capped loops $\hat{x}$ and $\hat{y}$ such that the underlying loops $x$ and $y$ lie sufficiently close to each other in an arbitrary symplectic surface $(\Sigma,\omega)$.
\par
To any $x \in \mathcal{L}_0(\Sigma)$, we may associate the set
\begin{align*}
S_x &:= \lbrace y \in \mathcal{L}_0(\Sigma): \exists t \in S^1 \; \text{such that} \; x(t)=y(t) \rbrace,
\end{align*}
with the property that $\mathcal{L}_0(\Sigma) \setminus S_x$ consists of precisely those loops $y$ such that $(x,y) \in \mathcal{L}_0(\Sigma)^2$ is a braid.
\par
We fix some family $J=(J_t)_{t \in S^1}$ of $\omega$-compatible almost complex structures, and let $g_J=(g_{J_t})_{t \in S^1}$ denote the associated family of compatible metrics. This data provides us with an exponential neighbourhood $\mathcal{O} \subseteq \mathcal{L}_0(\Sigma)$ of $x$, along with a diffeomorphism
\begin{align*}
Exp: U \subseteq \Gamma^\infty(x^* T \Sigma) &\rightarrow \mathcal{O} \\
\xi &\mapsto \lbrace t \mapsto exp^{J_t}_{x(t)}(\xi(t)) \rbrace
\end{align*}
from a neighbourhood of the zero section onto $\mathcal{O}$.
\par
Remark that, via the lifting property of the covering map $p: \cL{\Sigma} \rightarrow \mathcal{L}_0(\Sigma)$, any choice of a lift $\hat{x}=[x,\alpha] \in \cL{\Sigma} $ of $x$ gives rise to a unique lift $\tilde{\mathcal{O}}_\alpha$ of $\mathcal{O}$ and a lift 
\begin{align*}
\widetilde{Exp}_{\alpha}: U \subseteq \Gamma^\infty(x^* T \Sigma) &\rightarrow \tilde{\mathcal{O}}_{\alpha}.
\end{align*} 
Moreover, $(x^*T \Sigma,J,\omega)$ comes equipped with a homotopically unique unitary trivialization
\begin{align*}
T_{\hat{x}}: S^1 \times (\R^2, J_0, \omega_0) &\rightarrow (x^*T \Sigma,J,\omega).
\end{align*}
That is, $T_{\hat{x}}$ is a smooth map such that, for $t \in S^1$,
\begin{align*}
T_{\hat{x}}^t: (\R^2, J_0, \omega_0) &\rightarrow (T_{x(t)} \Sigma,(J_t)_{x(t)},(\omega)_{x(t)}) \\
v &\mapsto T_{\hat{x}}(t,v)
\end{align*}
is a linear isomorphism such that $(T_{\hat{x}}^t)^*(J_t)_{x(t)}=J_0$ and $(T_{\hat{x}}^t)^*(\omega)_{x(t)}=\omega_0$. Concretely, one may obtain $T_{\hat{x}}$ by choosing any capping disk $w:D^2 \rightarrow \Sigma$ representing $\alpha$, and choosing a unitary trivalization of $(w^*T\Sigma,J,\omega)$ (which exists by the contractibility of the disk), and one then obtains $T_{\hat{x}}$ by restricting this trivialization to the boundary.
\\ \\ 
For any $y \in  \mathcal{O} \setminus S_x$ and any capping $\hat{x}_{\alpha}:=[x,\alpha]$ of $x$, let $\hat{y}_\alpha$ denote the unique lift of $y$ lying in $\widetilde{\mathcal{O}}_\alpha$. We define the linking number of $\hat{x}_\alpha$ and $\hat{y}_\alpha$ as
\begin{align*}
\ell(\hat{x}_\alpha,\hat{y}_\alpha) &:= wind((T_{\hat{x}}^{-1} \circ \widetilde{Exp}_{\alpha}^{-1})(\hat{y}_\alpha)),
\end{align*}
where $wind(\xi)$ denotes the classical winding number of a non-vanishing family of vectors $t \mapsto \xi(t)$ for $t \in S^1$ in $\R^2$ (for a definition, see for example, Section 66 of \cite{Mu00}). Note that for $A \in \pi_2(\Sigma)$, it is not hard to show that
\begin{align*}
\ell(A \cdot \hat{x}_\alpha,A \cdot \hat{y}_\alpha) &= \ell(\hat{x}_\alpha,\hat{y}_{\alpha}) + c_1(A),
\end{align*}
where $c_1(A)$ denotes the first Chern number of the homotopy class $A$ (or, more precisely, its image in homology under the Hurewicz morphism). This follows from remarking that $v(t):= (\widetilde{Exp}_{\alpha}^{-1})(\hat{y}_\alpha) \in (x^*T \Sigma,J,\omega)$ exponentiates under $\widetilde{Exp}_{A \cdot \alpha}$ (based at the point $A \cdot \hat{x}_{\alpha}$) to $A \cdot \hat{y}$, together with the fact (see Section 2.7 of \cite{McSa17} for further elaboration on this point) that $t \mapsto L_t:= (T^t_{A \cdot \hat{x}_{\alpha}})^{-1} \circ (T^t_{\hat{x}_{\alpha}})$ defines a loop with Maslov index $c_1(A)$ --- in our setting, this is equivalent to saying that the loop $(L_t)$ is homotopic in $Sp(2;\R)$ to the loop $t \mapsto \begin{pmatrix}
\cos 2 c_1(A) \pi t & -\sin 2 c_1(A) \pi t \\ \sin 2 c_1(A) \pi t & \cos 2 c_1(A) \pi t
\end{pmatrix}$. It is moreover not hard to show that $\ell$ is symmetric in its arguments. In order to extend this definition to arbitrary cappings of the loops $x$ and $y$, let $A, B \in \pi_2(\Sigma)$ and define
\begin{align*}
\ell(A \cdot \hat{x}_\alpha, B \cdot \hat{y}_\alpha) &:= \ell(\hat{x}_\alpha,\hat{y}_\alpha) +\frac{1}{2}(c_1(A)+c_1(B)).
\end{align*}
It is straightforward to check that this definition does not depend on the choice of $\alpha$, and agrees with the previous definition in the case that $\hat{x}$ and $\hat{y}$ are close in $\cL{\Sigma}$. Note, however, that our definition may depend in principle on our choice of the family of almost complex structures $J=(J_t)$. It will be a consequence (Corollary \ref{Cor: J-indep-local-link}) of the more topological view of this quantity that we will develop in the following two sections that in fact the above-defined $\ell$ is independent of the choice of $J$, but it would be interesting to have a more direct proof of this fact which evinced the invariance under changes in $J$ more directly.
\subsection{The homological linking number for capped braids}\label{HomLinkSec}
\begin{definition}
Let $\hat{X}, \hat{Y} \in \kCapBraid{\Sigma}$ and $A \in \pi_2(\Sigma)^k$. We define the \textbf{homological $A$-linking number of $\hat{Y}$ relative to $\hat{X}$}
\begin{align*}
L_A(\hat{X},\hat{Y}) &:= \sum_{1 \leq i < j \leq k} \#(\tilde{h}_i \pitchfork \tilde{h}_j),
\end{align*}
where $h=(h_1, \cdots, h_k)$ is any $A$-cobordism from $\hat{X}$ to $\hat{Y}$ such that the graphs of the strands of $h$ in $[0,1] \times S^1 \times \Sigma$ are all pairwise transverse, and $\#(\tilde{h}_i \pitchfork \tilde{h}_j)$ denotes the signed count of the intersections of the graphs $\tilde{h}_i$ and $\tilde{h}_j$ (recall that in our setting $\Sigma$ carries the orientation induced by $\omega$).
\end{definition}
\begin{remark}
\begin{enumerate}
\item The above definition may be generalized straightforwardly by replacing the cylinder $[0,1] \times S^1$ with a surface $S_{g,k^-,k^+}$ of genus $g$, having $k^-$ negatively oriented boundary components and $k^+$ positively oriented boundary components. This provides a family of homotopy invariants for collections of $k^-$ `input' and $k^+$ `output' capped braids in the obvious way. Much of the theory developed in this work can be adapted in a straightforward way to use such invariants to extract information about the field-theoretic operations in Floer theory described in \cite{PSS96} at the chain level for Hamiltonian isotopies on surfaces, but we will not pursue this extension in this paper.
\item Similarly to Remark 4 above, we will not actually have need outside of this section for $L_A(\hat{X};\hat{Y})$ with $A \neq 0$. We include the notions with $A \neq 0$ here mainly because they can be convenient if one wants to reason about more sophisticated situations on the sphere than we will have cause to encounter in this work.
\end{enumerate}
\end{remark}
The following proposition summarizes the main properties of the homological linking number which we will need in our investigations.

\begin{proposition}\label{LinkingProp}
For any $\hat{X}, \hat{Y}, \hat{Z} \in \kCapBraid{\Sigma}$ and $A, B \in \pi_2(\Sigma)^k$ we have that:
\begin{enumerate}
\item $L_A(\hat{X},\hat{Y})$ is well-defined.
\item For any $\sigma \in S_k$, $L_{\sigma \cdot A}(\sigma \cdot \hat{X}; \sigma \cdot \hat{Y})= L_{A}(\hat{X};\hat{Y})$. \label{LinkProps-permute}
\item $L_A(\hat{X},\hat{Y}) + L_B( \hat{Y}, \hat{Z}) = L_{A+B}(\hat{X},\hat{Z})$. \label{LinkProps-additive}
\item If $\hat{X}$ and $\hat{Y}$ are $A$-homotopic, then $L_A(\hat{X},\hat{Y})=0$. \label{LinkProps-Zero}
\item $L_A(\hat{X},\hat{Y})$ depends on $\hat{X}$ (resp. $\hat{Y}$) only up to $0$-homotopy. \label{LinkProps-HtpyInvar}
\item $L_A(\hat{X},\hat{Y})=-L_{-A}(\hat{Y},\hat{X})$. \label{LinkProps-Neg}
\item $L_A(\hat{X},B \cdot \hat{Y}) = L_{A+B}(\hat{X},\hat{Y})$, and $L_A(B \cdot \hat{X},\hat{Y})=L_{A-B}(\hat{X},\hat{Y})$. \label{LinkProps-recapping}
\item $L_A(\hat{X},\hat{Y})=L_0(\hat{X},\hat{Y}) + L_A(\hat{0},\hat{0})$.
\item $L_0(\hat{X},A \cdot \hat{X}) = (k-1) \sum_{i=1}^k \frac{c_1(A_i)}{2}$. 
\end{enumerate}
\end{proposition}
\begin{proof}
\begin{enumerate}
\item The fact that $L_A(\hat{X},\hat{Y})$ is well-defined independently of the $A$-cobordism $h$ chosen from $\hat{X}$ to $\hat{Y}$ in order to compute it follows from the standard transversality arguments that are typical in differential topology. Alternately, one may simply note that for any braid cobordism $h=(h_1, \cdots, h_k)$ and any $i=1, \cdots, k$, the graph of $h_i$ in $[0,1] \times S^1 \times \Sigma$ defines a compact surface with boundary $S_i \subseteq [0,1] \times S^1 \times \Sigma$ which induces a well-defined element of 
\begin{align*}
[S_i] & \in H_2([0,1] \times S^1 \times \Sigma; \check{X} \sqcup \check{Y}),
\end{align*} 
where $\check{X}$ and $\check{Y}$ denote the graphs of the braids $X$ and $Y$ respectively, thought of as submanifolds lying in $\lbrace 0 \rbrace \times S^1 \times \Sigma$ and $\lbrace 1 \rbrace \times S^1 \times \Sigma$, respectively. Note that $\partial S_i$ is disjoint from $S_j$ for all $i \neq j$ and so the intersection product of such classes is well-defined and 
\begin{align*}
L_A(\hat{X},\hat{Y})&= \sum_{1 \leq i < j \leq k} [S_i] \bullet [S_j],
\end{align*}
which, since $[0,1] \times S^1 \times (\Sigma,\omega)$ is canonically oriented by $(ds \wedge dt) \wedge \omega$, is obviously precisely what is computed by the sum of pairwise intersection numbers of the graphs when these are transverse.

\item This statement follows upon remarking that for any $\sigma \in S_k$ and any $A \in \pi_2(\Sigma)^k$, $h = (h_1,\ldots,h_k)$ is an $A$-cobordism from $\hat{X}$ to $\hat{Y}$ if and only if $\sigma \cdot h:=(h_{\sigma^{-1}(1)},\ldots, h_{\sigma^{-1}(k)})$ is a $(\sigma \cdot A)$-cobordism from $\sigma \cdot \hat{X}$ to $\sigma \cdot \hat{Y}$.

\item This is straightforward and follows directly from concatenating an $A$-cobordism from $\hat{X}$ to $\hat{Y}$ with a $B$-cobordism from $\hat{Y}$ to $\hat{Z}$.
\item An $A$-cobordism $h=(h_1,\cdots,h_k)$ is an $A$-homotopy precisely when the graphs of the $h_i$, $i=1, \cdots, k$ are all disjoint. Clearly this implies $L_A(\hat{X},\hat{Y})=0$.

\item Let $\hat{X}', \hat{Y}' \in \kCapBraid{\Sigma}$ be $0$-homotopic to $\hat{X}$ and $\hat{Y}$, respectively, then items \ref{LinkProps-additive} and \ref{LinkProps-Zero} imply
\begin{align*}
L_A(\hat{X},\hat{Y}) &= L_0(\hat{X},\hat{X}') + L_A(\hat{X}',\hat{Y}') + L_0(\hat{Y}',\hat{Y}) \\
&= 0 +  L_A(\hat{X}',\hat{Y}') +  0.
\end{align*}

\item This follows immediately from noting that $h$ is an $A$-cobordism from $\hat{X}$ to $\hat{Y}$ if and only if $\bar{h}(s,t):=h(1-s,t)$ is a $(-A)$-cobordism from $\hat{Y}$ to $\hat{X}$.

\item That $L_A(\hat{X},B \cdot \hat{Y}) = L_{A+B}(\hat{X},\hat{Y})$ follows immediately from the equivalence of the homotopy conditions
\begin{align*}
\alpha_i \# [h_i] \# -\beta_i &= A_i \# B_i= A_i + B_i, \; \text{and} \\
\alpha_i \# [h_i] \# -(B_i \cdot \beta_i) &= A_i. 
\end{align*}
That is, $h$ is an $A+B$-cobordism from $\hat{X}$ to $\hat{Y}$ if and only if $h$ is also a $A$-cobordism from $\hat{X}$ to $B \cdot \hat{Y}$. That $L_A(B \cdot \hat{X},\hat{Y})=L_{A-B}(\hat{X},\hat{Y})$ follows from parallel reasoning.
\item Item \ref{LinkProps-additive} implies
\begin{align*}
L_A(\hat{X},\hat{Y}) &= L_0(\hat{X},\hat{0}) + L_A(\hat{0},\hat{0}) + L_0( \hat{0},\hat{Y}) \\
&= L_0(\hat{X},\hat{0}) + L_0(\hat{0},\hat{Y}) + L_A(\hat{0},\hat{0}) \\
&= L_0(\hat{X},\hat{Y}) + L_A(\hat{0},\hat{0}).
\end{align*}
\item We note first that item \ref{LinkProps-additive} implies that
\begin{align*}
L_0(\hat{0},\hat{X}) + L_0(\hat{X},A \cdot \hat{X}) + L_0(A \cdot \hat{X}, A \cdot \hat{0}) &=L_0(\hat{0},A \cdot \hat{0}).
\end{align*}
Next, items \ref{LinkProps-Neg} and \ref{LinkProps-recapping} imply that
\begin{align*}
L_0(A \cdot \hat{X}, A \cdot \hat{0}) &=  L_A(A \cdot \hat{X},\hat{0})\\
&=  -L_{-A}(\hat{0}, A \cdot \hat{X}) \\
&= -L_0(\hat{0},\hat{X}),
\end{align*}
whence we need only show that the desired formula holds when $\hat{X} = \hat{0}$. To reduce to an even simpler case, let us write $A$ as
\begin{align*}
(A_1, \cdots, A_k) &=(A_1,0,\cdots,0) + (0,A_2,0, \cdots,0) + \cdots + (0,0,\cdots, A_k) \\
&=A_1' + \cdots + A_k'.
\end{align*}
By items \ref{LinkProps-additive} and \ref{LinkProps-recapping}, demonstrating the desired equality is therefore equivalent to showing that
\begin{align*}
L_0(\hat{0},A_i' \cdot \hat{0}) &= (k-1) \cdot \frac{c_1(A_i)}{2}
\end{align*}
for any $i=1, \cdots, k$. In what follows, let $(\hat{p}_1, \cdots, \hat{p}_k) = \hat{0}$ represent the trivial capped braid. Since the statement is trivial when $\Sigma \neq S^2$, as then $\pi_2(\Sigma)=0$ and every capped braid is $0$-homotopic to itself, we now suppose $\Sigma=S^2$. For $m \in \Z$, if $u_i: (S^2,\ast) \rightarrow (\Sigma,p_i)$ represents $A_i=m [S^2] \in \pi_2(\Sigma,p_i)$, we may pull $u_i$ back along the quotient $[0,1] \times S^1 \rightarrow S^2$ (given by collapsing the boundary circles to points) to a map which we will denote
\begin{align*}
h_i: [0,1] \times S^1 & \rightarrow \Sigma.
\end{align*}
If we take $h$ to be the $0$-cobordism from $\hat{0}$ to $A_i' \cdot \hat{0}$ given by $h_i$ as the $i$-th strand and the constant strand $h_j(s,t) \equiv p_j$ for all other strands $j \neq i$, then the important point is that $PD(c_1)=2[S^2]$ and hence the intersection of the graph of $h_i$ with the constant cylinder $h_j(s,t) \equiv p_j$ for $j \neq i$ contributes precisely 
\begin{align*}
(u_{i*}[S^2]) \cap [p_j]&= m [S^2] \\
&= \frac{c_1(A_i)}{2} 
\end{align*}
to the sum defining $L_0(\hat{0},A_i' \cdot \hat{0})$, and such intersections are the only ones that occur, since all other strands are constant and disjoint. The desired equality follows.
\end{enumerate}
\end{proof}

\begin{proposition}
For $A \in \pi_2(\Sigma)^k$, and $[\hat{X}], [\hat{Y}] \in \kCapBraid{\Sigma}/S_k$, the function $L_A([\hat{X}],[\hat{Y}]) := L_A(\hat{X},\hat{Y}) $, is well-defined.
\end{proposition}
\begin{proof}
Items \ref{LinkProps-permute} and \ref{LinkProps-Neg} of proposition \ref{LinkingProp} imply for any $\sigma, \tau \in S_k$ that
\begin{align*}
L_A(\sigma \cdot \hat{X}, \tau \cdot \hat{Y})&=L_0(\sigma \cdot \hat{X},\hat{0})  + L_A(\hat{0},\hat{0}) + L(\hat{0},\tau \cdot \hat{Y}) \\
&= L_A(\hat{0}, \hat{0}) - L_0(\hat{0}, \sigma \cdot \hat{X}) + L_0(\hat{0},\tau \cdot \hat{Y}),
\end{align*}
so it suffices to show that the expression $L_0(\hat{0}, \sigma \cdot \hat{X})$ is independent of $\sigma \in S_k$. To see this, note that item $2$ of the previous proposition, together with the fact that $0 \in Fix_{S_k}(\pi_2(\Sigma)^k)$ implies that $L_0(\hat{0}, \sigma \cdot \hat{X})=L_0(\sigma^{-1} \cdot \hat{0}; \hat{X})$. Moreover, it is easy to see that $\hat{0}=(\hat{p}_1, \ldots, \hat{p}_k)$ is $0$-homotopic to $\sigma \cdot \hat{0}=(\hat{p}_{\sigma^{-1}(1)},\ldots,p_{\sigma^{-1}(k)})$ for any $\sigma \in S_k$ (simply choose $k$ paths $s \mapsto \gamma_i(s) \in \Sigma$, $s \in [0,1]$ from $p_i$ to $p_{\sigma^{-1}(i)}$ such that $\gamma_i(s)=\gamma_j(s)$ implies $i=j$ for all $s \in [0,1]$ and define the $i$-th strand of the $0$-homotopy to be $h_i(s,t)=\gamma_i(s)$) and consequently, $L_0(\sigma^{-1} \cdot \hat{0}, \hat{X})=L_0(\hat{0},\hat{X})$, which is independent of $\sigma \in S_k$.
\end{proof}

\begin{proposition}\label{prop:HLinkEqualsLink}
Fix any $S^1$-family $J=(J_t)$ of $\omega$-compatible almost complex structures. Let $\hat{X}=(\hat{x}_1,\hat{x}_2) \in \CapBraid{2}{\Sigma}$ with $x_2$ lying in an exponential neighbourhood of $x_1$ in $\mathcal{L}_0(\Sigma)$ (as in Section \ref{LocalLinkSec}), then $L_0(\hat{0}, \hat{X}) = \ell(\hat{x}_1,\hat{x}_2)$.
\end{proposition}

\begin{proof}
We note first that if $A,B \in \pi_2(\Sigma)$, then 
\begin{align*}
L_0(\hat{0},(A,B) \cdot \hat{X}) - L_0(\hat{0}, \hat{X}) &= \frac{1}{2}(c_1(A)+c_1(B)) \\
&=\ell(A \cdot \hat{x}_1, B \cdot \hat{x}_2) - \ell(\hat{x}_1,\hat{x}_2),
\end{align*}
and so it suffices to prove the statement in the case in which $\hat{x}_2$ lies inside an exponential neighbourhood of $\hat{x}_1$ in $\cL{\Sigma}$. Note that this implies that if we write $\hat{x}_1=[x_1,\alpha]$ and $\xi:=(\widetilde{Exp}_{\alpha}^{-1})(\hat{x}_2)$, then $\hat{x}_2=[x_2,\alpha \# c]$, where $c(s,t):=\widetilde{Exp_{\alpha}}(s\xi(t))$ is the cylinder of geodesics from $x_1$ to $x_2$ and $\alpha \# c$ denotes the obvious capping class for $x_2$ obtained by concatenating a capping disk representing $\alpha$ and the cylinder $c$.    
\par
As discussed in Section \ref{LocalLinkSec}, $T_{\hat{x}} \cL{\Sigma}$ is naturally identified (up to a homotopy of trivializations) with $\Gamma^\infty(S^1 \times \R^2)$ and so, in the notation of that section, we may write $\hat{x}_2$ in local coordinates as 
\begin{align*}
v(t) &:= (T_{\hat{x}_1}^{-1})^{-1}(\xi(t))=(T_{\hat{x}_1}^{-1}\circ \widetilde{Exp}^{-1})(\hat{x}_2),
\end{align*}
and we have that $\ell(\hat{x}_1,\hat{x}_2)=wind(v)$ by definition. By the invariance of the homological linking number under $0$-homotopies of capped braids, and the homotopy invariance of the winding number in $\R^2 \setminus \lbrace0 \rbrace$, we may assume that  
\begin{align*}
v(t)&= r_0 e^{2 \pi \i l t} \in \C, \; \forall t \in [0,1],
\end{align*}
for some small $r_0 >0$ and $l= \ell(\hat{x}_1, \hat{x}_2)$. Taking the homotopies $h_1(s,t) \equiv 0$, $h_2(s,t)=(1-s)\frac{-r_0}{2} + s \frac{r_0}{2} e^{2 \pi \i l t}$, we get a $0$-cobordism $h=(h_1,h_2)$ from the trivial capped braid and $\hat{X}$. We see that the graphs $\tilde{h}_1$ and $\tilde{h}_2$ intersect only if $l \neq 0$, and in that case intersections occur when $s=\frac{1}{2}$ and $t=0, \frac{1}{l}, \cdots, \frac{l-1}{l}$. Moreover, using polar coordinates $(r, \theta)$ on $D^2$, since
\begin{align*}
\partial_s h_2 &= \partial_r \\
\partial_t h_2&= l \partial_\theta,
\end{align*}
each intersection is transverse and has orientation $sign(l)$. Consequently
\begin{align*}
L_0(\hat{0}, \hat{X})&= l = \ell(\hat{x}_1,\hat{x}_2),
\end{align*}
proving the claim.
\end{proof} 
The fact that the homological linking number is manifestly independent of the choice of $S^1$-family of $\omega$-compatible almost complex structures implies the following corollary.
\begin{cor}\label{Cor: J-indep-local-link}
The definition of $\ell(\hat{x}_1,\hat{x}_2)$ given for capped loops with close strands in Section \ref{LocalLinkSec} is independent of the choice of $J=(J_t)$ used in its definition.
\end{cor}
Proposition \ref{prop:HLinkEqualsLink} allows us to extend the notion of linking for capped loops $\hat{x}$ and $\hat{y}$ with $x(t) \neq y(t)$ for all $t \in S^1$ which was considered in Section \ref{LocalLinkSec} to pairs of capped loops which are not ``close'' (ie. which do not lie in some exponential neighbourhood of one another). Indeed, in virtue of the preceding proposition we may \textit{define}
\begin{align*}
\ell(\hat{x},\hat{y})&:=L_0(\hat{0}, \lbrace \hat{x},\hat{y} \rbrace)
\end{align*}
for any capped loops $\hat{x}, \hat{y} \in \cL{\Sigma}$ such that $x(t) \neq y(t)$ for all $t \in S^1$. For the duration of this article, it is in this sense that we will use the notation $\ell(\hat{x},\hat{y})$. We remark also for the convenience of the reader, as this will be relevant later, that the invariance of $L_0(\hat{X};\hat{Y})$ under $0$-homotopies (item \ref{LinkProps-HtpyInvar} of Proposition \ref{LinkingProp}) implies that $\ell(\hat{x},\hat{y})$ is constant as $\hat{x}$ varies in the set $\lbrace \hat{x}' \in  \cL{\Sigma}: x'(t) \neq y(t), \; for \; all \; t \in S^1 \rbrace$ (and similarly as $\hat{y}$ varies in $\lbrace \hat{y}' \in  \cL{\Sigma}: y'(t) \neq x(t), \; for \; all \; t \in S^1 \rbrace$).

\section{Elements of Floer theory and linking}\label{Local}
This section serves to collect and review the necessary facts that we will need from Hamiltonian Floer theory, along with the analysis of the asymptotic behaviour of Floer-type cylinders that proves crucial to our study. Section \ref{Sec:Floer} collects the Floer-theoretic preliminaries. Section \ref{Sec:Asymptotic analysis} explains how we may combine a result formulated by Siefring in \cite{Si08}, which relates the asymptotic behaviour of Floer-type cylinders tending to an orbit $x$ to the eigenvectors of the so-called \textit{asymptotic operator} associated to $x$, with Hofer-Wysocki-Zehnder's study in \cite{HWZ95} of the winding behaviour of such eigenvectors, in order to obtain information about the asymptotics of Floer-type cylinders. We then explain how we may combine this information with the homological linking number of Section \ref{HomLinkSec} to obtain control over the relative topology of collections of capped braids having Floer-type cylinders running between them. The results of the present section form the technical core of the rest of the work.
\par
The broad overview is this: one of the central geometric viewpoints in this work is that we may interpret collections of Floer-type cylinders as providing such deformations between the capped braid $\hat{X}$ formed by the capped orbits at their negative ends and the capped braid $\hat{Y}$ formed by the capped orbits at their positive ends. When we do this, the Gromov trick combined with positivity of intersections in dimension $4$ (discussed in Section \ref{sec-Gromov trick}) implies that the homological linking number $L_0(\hat{X},\hat{Y})$ must be non-negative, which gives strong topological restrictions on the putative existence of collections of Floer cylinders in terms of the relative topologies of the capped braids formed by their negative and positive ends. In the contact setting on $S^3$, the fact that positivity of intersection for pseudo-holomorphic cylinders in dimension $4$ is related to a monotonicity phenomenon for winding numbers seems to be a fairly well-appreciated phenomenon, going back at least to the work of Hofer-Wysocki-Zehnder in \cite{HWZ95}. The idea of relating the global topology of the braids formed by periodic orbits of a Hamiltonian to Floer theory via positivity of intersection arguments also has some precedent on the disk in \cite{BGVW15}, in which the authors use Floer homology to define braid invariants and obtain a forcing theory for periodic orbits of Hamiltonian. The approach sketched so far exhibits a slight deficiency when studying the deformations of braids induced by general collections of Floer-type cylinders however, in the sense that, in general, one may have multiple Floer cylinders emerging from, or converging to, the same orbit and it is not necessarily clear how one should treat this situation.
\par
In order to deal with cylinders which emerge from or converge to the same orbit, we make use of the analysis of the relative asymptotic behaviour of pseudo-holomorphic curves developed by Hofer-Wysocki-Zehnder and Siefring to connect the Conley-Zehnder index of an orbit to bounds on the winding behaviour of pairs of cylinders which emerge from or converge to that orbit. In the contact setting, these sorts of bounds (in the non-relative case), along with the insight that under appropriate index conditions on the asymptotic orbits families of pseudoholomorphic curves automatically form local foliations in the symplectization of a contact manifold, go back to the pioneering work of Hofer-Wysocki-Zehnder in \cite{HWZ96}, \cite{HWZ95}, \cite{HWZ99} and \cite{HWZ03}. Siefring has also more recently put this circle of ideas to use in \cite{Si11} to define an intersection number for arbitrary pseudoholomorphic curves in $4$-dimensional symplectic cobordisms which is invariant under homotopy. In any case, we may thereby split the analysis of the behaviour of general collections of Floer cylinders into two portions: an asymptotic portion on the ends where the linking behaviour of cylinders asymptotic to the same orbit is controlled by the Conley-Zehnder index, and a compact portion which interpolates between two capped braids such that their intersections are controlled by the homological linking number of the capped braids at each end of this compact portion.

\subsection{Floer theory}\label{Sec:Floer}
In this section, we give a rapid overview of the elements of Floer theory of which we will have need, mainly to fix notation and conventions. For a more detailed treatment, see \cite{AD14} or \cite{Sa97} for standard accounts of Hamiltonian Floer theory (see also \cite{HoSa95} for its adaptation to the semi-positive case) and \cite{Sc95}, \cite{PSS96}, \cite{Se97} or \cite{La04} for a more detailed treatment of how Floer theory fits into a field theory over surfaces. Throughout, we assume that $(M^{2n},\omega)$ is a strongly semi-positive compact symplectic manifold (ie. $2-n \leq c_1(A) < 0$ implies $\omega(A) \leq 0$ for all $A \in \pi_2(M)$). $\mathcal{J}(M,\omega)$ denotes the space of all smooth $\omega$-compatible almost complex structures. For convenience, we work only with $\Z_2$-coefficients, but this restriction is inessential.
\par
A smooth (not necessarily autonomous) Hamiltonian function 
\begin{align*}
H: S^1 \times M &\rightarrow \R \\
(t,x) &\mapsto H_t(x)
\end{align*}
induces a time-dependent vector field $(X^t_H)_{t \in [0,1]}$ on $M$ defined by the relation 
\begin{align*}
\omega(X^t_H,-)&= -dH_t.
\end{align*}
The Hamiltonian isotopy obtained as the flow by this vector field is denoted $\phi^H:=(\phi^H_t)_{t \in [0,1]}$. A Hamiltonian $H$ is said to be \textbf{normalized} if $\int_M H_t \omega^n =0$ for all $t \in S^1$. There is a group structure on the set of Hamiltonian functions $C^\infty(S^1 \times M)$ given by the operation
\begin{align*}
(H \# K)(t,x)&:= H(t,x) + K(t,(\phi^H_t)^{-1}(x)),
\end{align*}
which is such that $\phi^{H \# K}_t= \phi^H_t \circ \phi^K_t$. The inverse of $H$ with respect to this relation is given by
\begin{align*}
\bar{H}(t,x)&=-H(t,\phi^H_t(x)),
\end{align*}
which generates the isotopy $t \mapsto (\phi^H_t)^{-1}$.
\par
Recall that $\cL{M}$ denotes the Novikov covering of the loop space. That is to say, elements $[\gamma, v] \in \cL{M}$ are capped loops $(\gamma,v)$ in $M$ subject to the equivalence relation $(\gamma_1,v_1) \sim (\gamma_2,v_2)$ precisely when $\gamma_1=\gamma_2$ and $[v_1 \# \bar{v}_2] \in \ker c_1(M) \cap \ker [\omega] \subseteq \pi_2(M)$.
\par 
The Hamiltonian $H$ defines a corresponding \textbf{action functional} on the Novikov covering of the loop space
\begin{align*}
\mathcal{A}_H: \cL{M} &\rightarrow \R \\
[\gamma,v] &\mapsto \int_0^1 H_t(\gamma(t)) \; dt - \int_{D^2} v^* \omega.
\end{align*}
We write 
\begin{align*}
\Per{H} &:= Crit \; \mathcal{A}_H, \; \text{and} \\
Per_0(H)&:= \pi(\Per{H}) \subseteq \mathcal{L}_0(M),
\end{align*}
noting that the latter consists precisely of the contractible $1$-periodic orbits of $\phi^H$, while the former consists of capped such periodic orbits. We define the \textbf{spectrum} of $H$ as
\begin{align*}
Spec(H) &:= \mathcal{A}_H(Crit \; \mathcal{A}_H) \subseteq \R.
\end{align*}
$H$ is said to be \textbf{non-degenerate} if for all $x \in Per_0(H)$,  $(D \phi^H_1)_{x(0)}$ has no eigenvalues equal to $1$. When $H$ is non-degenerate, there exists a well-defined \textbf{Conley-Zehnder index}
\begin{align*}
\mu=\mu_{CZ}: \Per{H} &\rightarrow \Z.
\end{align*} 
See \cite{RS93} for details on the definition of $\mu$. We shall normalize the Conley-Zehnder index by insisting that if $H$ is a $C^2$-small Morse function and $x$ a critical point of $H$, then
\begin{align*}
\mu(\hat{x}) &= \mu_{Morse}(x) - n,
\end{align*}
where $\mu_{Morse}$ is the Morse index of $x$, and $\hat{x}$ denotes the trivial capping of the constant orbit $x$. For $k \in \Z$, and any $P \subseteq \Per{H}$ we define $P_{(k)}$ to be the collection of capped orbits in $P$ with Conley-Zehnder index $k$. Remark that for $A \in \pi_2(M)$,
\begin{align*}
\mu_{CZ}([x,A \#w]) &= \mu_{CZ}([x,w]) - 2c_1(A),
\end{align*}
where $c_1(A)$ denotes the first Chern number of $A$.
\par 
Given $\hat{x}^\pm=[x^\pm,w^\pm] \in \Per{H}$, we write $C^\infty(\R \times S^1;M)_{\hat{x}^-,\hat{x}^+}$ for the subspace of $C^\infty(\R \times S^1; M)$ consisting of cylinders which induce a $0$-homotopy from $\hat{x}^-$ to $\hat{x}^+$.
Letting 
\begin{align*}
\mathcal{E} &\rightarrow C^\infty(\R \times S^1; M)_{\hat{x}^-,\hat{x}^+}
\end{align*}
be the infinite dimensional vector bundle with fiber $\mathcal{E}_u= \Gamma^{\infty}(u^*TM)$ at $u$, any smooth $S^1$-family $J=(J_t)_{t \in S^1} \subseteq \mathcal{J}(M,\omega)$, permits the definition of the \textbf{Floer operator}, which is the section 
\begin{align*}
\mathcal{F}_{H,J}: C^\infty(\R \times S^1; M)_{\hat{x}^-,\hat{x}^+} &\rightarrow \mathcal{E} \\
 u &\mapsto \partial_s u + J(\partial_t u -X_H).
\end{align*} 
After passing to appropriate Banach space completions (see Section 8.2 of \cite{AD14} for instance), $\mathcal{F}_{H,J}$ defines a Fredholm operator with index $\mu(\hat{x}^-) - \mu(\hat{x}^+)$. The intersection of $\mathcal{F}_{H,J}$ with the $0$-section gives rise to \textbf{Floer's equation}
\begin{align}\label{FE}
\partial_s u + J_t(\partial_t u - X_H^t) &=0
\end{align}
for smooth maps $u: \R \times S^1 \rightarrow M$. If we define the \textbf{energy} of $u \in C^\infty(\R \times S^1;M)$ by
\begin{align*}
E(u)&:= \int_{\R \times S^1} \| \partial_{s} u \|^2_{J_t} \; dt \; ds,
\end{align*} 
then the finite energy solutions of Floer's equation may be thought of as the projections to $M$ of negative gradient flow lines of $\mathcal{A}_H$ with respect to the $L^2$-metric on $\cL{M}$ induced by $J$
\begin{align*}
\langle \xi,\eta \rangle_{[\alpha,w]}&:= \int_0^1 \omega_{\alpha(t)}(\xi(t),\eta(t)) \; dt, \; \text{for } [\alpha,w] \in \cL{M}, \; \xi, \eta \in T_\alpha \cL{M}.
\end{align*}
It follows easily from this that if $u \in C^\infty(\R \times S^1;M)_{\hat{x}^-,\hat{x}^+}$ is such a finite energy solution, then $E(u)= \mathcal{A}_H(\hat{x}^-) - \mathcal{A}_H(\hat{x}^+)$.
\par 
For any $\hat{x}^\pm \in \Per{H}$, we define $\widetilde{\mathcal{M}}(\hat{x}^-,\hat{x}^+;H,J)$ to be the zero set of $\mathcal{F}_{H,J}$ on $C^\infty(\R \times S^1;M)_{\hat{x}^-,\hat{x}^+}$. It carries an obvious $\R$-action given by translation in the $s$-coordinate. The \textbf{reduced moduli space} is defined by
\begin{align*}
\mathcal{M}(\hat{x},\hat{y};H,J) &:= \widetilde{\mathcal{M}}(\hat{x},\hat{y};H,J) / \R.
\end{align*}
The expected dimension of the reduced moduli space is $\mu(\hat{y}) - \mu(\hat{x}) -1$. In order to define the appropriate genericity condition on pairs $(H,J)$ such that we may define the Floer complex unproblematically, we follow \cite{Se97} in introducing the following sets. For $k \in \Z_{\geq 0}$, let $V_k$ be the set of pairs $(t,p) \in S^1 \times M$ such that $p \in \im v$ for $v$ some non-constant $J_t$-holomorphic sphere with $c_1(v) \leq k$.
\begin{definition}
A pair $(H,J)$ with $H$ and $J$ as above will be called a \textbf{Floer pair}. A Floer pair with $H$ non-degenerate will be called a \textbf{non-degenerate Floer pair}. A non-degenerate Floer pair $(H,J)$ will be said to be \textbf{Floer regular} if
\begin{enumerate}
\item for every $x \in Per_0(H)$, $(t,x(t)) \not \in V_1(J)$ for all $t \in S^1$;
\item for all $\hat{x}^\pm \in \Per{H}$, the linearization $(D \mathcal{F}_{H,J})_u$ of the Floer operator at $u \in C^\infty(\R \times S^1; M)_{\hat{x}^-,\hat{x}^+}$ is surjective for all $u \in \widetilde{\mathcal{M}}(\hat{x}^-,\hat{x}^+;H,J)$;
\item for all $\hat{x}^\pm \in \Per{H}$ with $\mu(\hat{x}^-) - \mu(\hat{x}^+) \leq 2$ and all $u \in \widetilde{\mathcal{M}}(\hat{x}^-,\hat{x}^+;H,J)$, $(t,u(s,t)) \not \in V_0(J)$ for all $(s,t) \in \R \times S^1$.
\end{enumerate}
\end{definition}
For $H$ non-degenerate, let $\mathcal{J}^{reg}(H) \subseteq C^\infty(S^1;\mathcal{J}(M,\omega))$ denote the space of $S^1$-families of complex structures such that $(H,J)$ is Floer regular. $\mathcal{J}^{reg}(H)$ is residual in $C^\infty(S^1;\mathcal{J}(M,\omega))$.
\par 
If $(H,J)$ is Floer regular, then $\mathcal{M}(\hat{x},\hat{y};H,J)$ is a compact manifold of dimension $0$ whenever $\mu(\hat{x}) - \mu(\hat{y})=1$, and in this case we may define the \textbf{Floer chain complex} $CF(H,J)$ to be the set of formal sums of the form
\begin{align*}
\sum_{\hat{x} \in \Per{H}} & a_{\hat{x}} \hat{x},
\end{align*}
where $a_{\hat{x}} \in \Z_2$ for all $\hat{x} \in \Per{H}$ and which moreover verifies the \textit{Novikov condition}: for all $c \in \R$,
\begin{align*}
\# \lbrace a_{\hat{x}} \neq 0: \mathcal{A}_H(\hat{x}) \geq c \rbrace& < \infty.
\end{align*}
$CF(H,J)$ is then graded by $\mu$ and has a differential defined on generators by 
\begin{align*}
\partial_{H,J} \hat{x} &:= \sum_{\mu(\hat{x}) - \mu(\hat{y})=1} n(\hat{x},\hat{y}) \hat{y},
\end{align*}
with $n(\hat{x},\hat{y})$ being the mod $2$ count of elements in $\mathcal{M}(\hat{x},\hat{y};H,J)$. The homology of this complex $HF_*(H)$ is the \textbf{Floer homology of $H$} and is independent of the choice of $J$.
\par 
The Floer complex has the structure of a \textit{filtered complex}, with the filtration coming from the action functional. Explicitly, for $\sigma = \sum_{\hat{x} \in \Per{H}} a_{\hat{x}} \hat{x} \in CF_*(H,J)$, we define 
\begin{align*}
\supp \sigma &:= \lbrace \hat{x} \in  \Per{H}: a_{\hat{x}} \neq 0 \rbrace,
\end{align*}
and we define the \textbf{level of $\sigma$} to be 
\begin{align*}
\lambda_H(\sigma) &:= \sup_{\hat{x} \in \supp \sigma} \mathcal{A}_H(\hat{x}).
\end{align*} 
Let $H^S_2(M)$ denote the image in $H_2(M; \Z)$ of the Hurewicz morphism, and let \\ $\Gamma_{\omega}:= H^S_2(M) / \ker c_1 \cap \ker [\omega]$. We define the \textbf{Novikov ring}
\begin{align*}
\Lambda_{\omega}&:= \lbrace \sum_{A \in \Gamma_{\omega}} \lambda_A e^A: \lambda_A \in \Z_2,  \# \lbrace \lambda_A \neq 0, \omega(A) \leq c \rbrace < \infty, \; \text{for all } c \in \R \rbrace.
\end{align*}
This is a graded commutative ring with grading given by declaring $deg(e^A):= 2c_1(A)$. $CF_*(H,J)$ is a $\Lambda_{\omega}$-module where the action of $\Lambda_{\omega}$ is defined on generators $\hat{x}=[x,v]$ of $CF_*(H,J)$ and $e^A$ of $\Lambda_{\omega}$ by $e^A \cdot \hat{x}:=[x,A \# v]$,
and extended linearly. Note that we have the relations
\begin{align*}
\mu(e^A \cdot \hat{x}) &= \mu(\hat{x}) - 2c_1(A), \\
\mathcal{A}_H(e^A \cdot \hat{x}) &= \mathcal{A}_H(\hat{x}) - \omega(A).
\end{align*}
Remark that the $\Lambda_{\omega}$-action needn't preserve the filtration.
\par
It is a standard fact in Floer theory that if $f \in C^\infty(M)$ is a sufficiently $C^2$-small Morse function and $J \in \mathcal{J}(M,\omega)$ is such that $(f,g_J)$ is Morse-Smale, then the Floer chain complex of $(f,J)$ may be identified (after a grading shift) with the \textbf{quantum chain complex} of $(f,g_J)$, which is by definition the Morse complex of $(f,g_J)$ with coefficients in the Novikov ring, ie.
\begin{align*}
CF_*(f,J) &\cong QC_{*+n}(f,g_J):= (C^{Morse}(f,g_J) \otimes \Lambda_{\omega})_{*+n}.
\end{align*}
Taking homology then gives a natural identification with the quantum homology of $(M,\omega)$: 
\begin{align*}
HF_*(f) \cong QH_{*+n}(M,\omega).
\end{align*}
\subsubsection{Continuation maps}\label{Sec-ContMaps}
\begin{definition}
For $X$ a smooth manifold, a function $F \in C^\infty(\R; X)$ is said to be \textbf{$T$-adapted} for $T \in (0,\infty)$ if $(\partial_s F)_{s_0} \equiv 0$ for all $\vert s_0 \vert \geq T$. $F$ is said to be \textbf{adapted} if it is $T$-adapted for some $T$. For $X=C^{\infty}(S^1 \times M)$, and $H^\pm \in C^{\infty}(S^1 \times M)$, we denote by $\mathscr{H}(H^-,H^+)$ the space of adapted homotopies $\mathcal{H}$ having $\lim_{s \rightarrow \pm \infty} \mathcal{H}(s) \equiv H^\pm$. We make a similar definition for $\mathscr{J}(J^-,J^+)$ in the case where $X=C^\infty(S^1; \mathcal{J}(M,\omega))$. 
\end{definition}
\begin{definition}
A pair $(\mathcal{H},\mathbb{J})$ is an \textbf{adapted homotopy of Floer data} from $(H^-,J^-)$ to $(H^+,J^+)$ if $\mathcal{H}=(H^s_t) \in \mathscr{H}(H^-,H^+)$ and $\mathbb{J}=(J^s_t) \in \mathscr{J}(J^-,J^+)$. We will write $\HJ(H^-,J^-;H^+,J^+)$ for the collection of all such adapted homotopies, often omitting the dependence on $(H^\pm,J^\pm)$ if it is clear from context.
\end{definition}
Just as in the $s$-independent case, for any adapted homotopy of Floer data $(\mathcal{H},\mathbb{J})$, we obtain a corresponding Floer operator $\mathcal{F}_{\mathcal{H},J}$. For any pair $\hat{x}^\pm \in \Per{H^\pm}$, consideration of the zeros of $\mathcal{F}_{\mathcal{H},\mathbb{J}}$ along $C^\infty(\R \times S^1; M)_{\hat{x}^-,\hat{x}^+}$ gives rise to the $s$-dependent Floer equation
\begin{align}\label{sFE}
\partial_s u + J^{s}_t(\partial_t u -X_{H_t^s}) &= 0,
\end{align}
and everything proceeds as before, with the proviso that now, if $u \in C^\infty(\R \times S^1;M)_{\hat{x}^-,\hat{x}^+}$ solves Equation \ref{sFE}, then its energy is given by 
\begin{equation}\label{Eq: s-dependent Action-Energy equation}
E(u)=\mathcal{A}_{H^-}(\hat{x}^-) - \mathcal{A}_H^+(\hat{x}^+) + \int_{- \infty}^\infty \int_0^1 (\partial_s \mathcal{H})(s,t,u(s,t)) \; dt \; ds.
\end{equation}
The moduli space $\mathcal{M}(\hat{x}^-,\hat{x}^+; \mathcal{H},J)$ is defined to be the zero set of $\mathcal{F}_{\mathcal{H},J}$ on $C^\infty(\R \times S^1;M)_{\hat{x}^-,\hat{x}^+}$.

\begin{remark}
When $(H^-,J^-)=(H^+,J^+)$, then the $s$-independent homotopy $(\mathcal{H},\mathbb{J})=(H^-,J^-)=(H^+,J^+)$ is a special case of an adapted homotopy. In this case, $\mathcal{M}(\hat{x},\hat{y};\mathcal{H},J)= \widetilde{\mathcal{M}}(\hat{x},\hat{y};H^\pm,J^\pm)$. In the sequel, when we speak of adapted homotopies of Floer data, this case is included.
\end{remark}
As in the $s$-independent case, in order to formulate the appropriate generic regularity criterion, we introduce the following set: we define $V_0(\mathbb{J})$ to be the set of tuples $(s,t,p) \in \R \times S^1 \times M$ such that $p \in \im v$ for $v$ a non-constant $\mathbb{J}^s_t$-holomorphic sphere with $c_1(v) \leq 0$.

\begin{definition}
Given $(H^\pm,J^\pm)$ Floer regular, $\hat{x}^\pm \in \Per{H^\pm}$, and $(\mathcal{H},\mathbb{J}) \in \HJ$, we will say that $(\mathcal{H},\mathbb{J})$ is \textbf{$(\hat{x}^-,\hat{x}^+)$-regular} if
\begin{enumerate}
\item the linearization $(D \mathcal{F}_{\mathcal{H},\mathbb{J}})_u$ of the Floer operator at $u \in C^\infty(\R \times S^1; M)_{\hat{x}^-,\hat{x}^+}$ is surjective for all $u \in \widetilde{\mathcal{M}}(\hat{x}^-,\hat{x}^+;H,J)$;
\item if $\mu(\hat{x}^-) - \mu(\hat{x}^+) \leq 1$, then $(s,t,u(s,t)) \not \in V_0(\mathbb{J})$ for all $(s,t) \in \R \times S^1$ and all $u \in \mathcal{M}(\hat{x}^-,\hat{x}^+;\mathcal{H},\mathbb{J})$.
\end{enumerate}
We denote the collection of all such adapted homotopies by $\HJ^{reg}_{\hat{x}^-,\hat{x}^+}$. $(\mathcal{H},\mathbb{J})$ will be said to be \textbf{Floer-regular} if it is $(\hat{x}^-,\hat{x}^+)$-regular whenever $\mu(\hat{x}^-)-\mu(\hat{x}^+) \leq 1$. We denote the space of Floer-regular adapted homotopies from $(H^-,J^-)$ to $(H^+,J^+)$ by $\HJ^{reg}(H^-,J^-;H^+,J^+)$, suppressing the dependence on $(H^\pm,J^\pm)$ when no confusion will arise.
\end{definition}
For any Floer regular $(H^\pm,J^\pm)$ and any fixed $\mathbb{J} \in \mathscr{J}(J^-,J^+)$, the set $\mathscr{H}^{reg}(\mathbb{J};H^-,H^+) \subseteq \mathscr{H}(H^-,H^+)$ of adapted homotopies $\mathcal{H}$ such that $(\mathcal{H},\mathbb{J})$ is Floer regular is residual.
\par 
For $(\mathcal{H},\mathbb{J}) \in \HJ^{reg}$, the spaces $\mathcal{M}(\hat{x}^-,\hat{x}^+; \mathcal{H},J)$ are all compact manifolds of dimension $0$ whenever $\mu^{H^-}(\hat{x}^-) =\mu^{H^+}(\hat{x}^+)$, and so we may define the \textbf{continuation morphism}
\begin{align*}
h_{\mathcal{H},J}: CF_*(H^-,J^-) &\rightarrow CF_*(H^+,J^+)
\end{align*} 
on generators by setting 
\begin{align*}
h_{\mathcal{H},\mathbb{J}} (x^-)&:= \sum_{\mu(\hat{x}^-) - \mu(\hat{x}^+)=0} n(\hat{x}^-,\hat{x}^+) \hat{x}^+,
\end{align*}
where $n(\hat{x}^-,\hat{x}^+)$ is the mod $2$ count of elements in the moduli space $\mathcal{M}(\hat{x}^-,\hat{x}^+; \mathcal{H},\mathbb{J})$. 
\par The continuation morphism is a morphism of complexes, which descends to an isomorphism at the level of homology. Moreover any two continuation maps between $(H^-,J^-)$ and $(H^+,J^+)$ define the same map at the level of homology, and further these isomorphisms satisfy the obvious composition law $h_{21} \circ h_{10} = h_{20}$, where $h_{ji}: HF(H_i) \rightarrow HF(H_j)$.
\subsubsection{The PSS isomorphism}\label{Sec: PSS defn Sect}
There is another type of morphism of chain complexes, in some way related to the continuation morphisms, which will concern us in this work. Introduced in \cite{PSS96} (see also \cite{Sc95}), PSS maps may be viewed as a variant on Floer continuation maps, with the exception that they consider adapted homotopies of Floer data from $(0,J^-)$ --- for $J^- \in \mathcal{J}(M,\omega)$ an autonomous almost complex structure \textemdash to the Floer pair $(H,J)$ whose Floer complex we are studying. The fact that the $0$ function is a heavily degenerate Hamiltonian forces some modifications to the definition of the chain morphism. Explicitly, we write
\begin{align*}
\HJ^{PSS}(H,J) &:= \lbrace (J^-;\mathcal{H},J): J^- \in \mathcal{J}(M,\omega), \; (\mathcal{H}, \mathbb{J}) \in \HJ(0,J^-;H,J) \rbrace
\end{align*}
and the set of all \textbf{PSS data} for the pair $(H,J)$ is 
\begin{align*}
PSS&(H,J):= \\
\lbrace (f,g;J^-;\mathcal{H},\mathbb{J}) &\in C^\infty(M) \times Met(M) \times \HJ^{PSS}(H,J): \; (f,g) \text{ is Morse-Smale} \rbrace,
\end{align*}
where $Met(M)$ denotes the space of smooth Riemannian metrics on $M$. There is a residual set $PSS_{reg}(H,J) \subseteq PSS(H,J)$ of \textbf{regular PSS data} such that for any $\mathcal{D}=(f,g;J^-;\mathcal{H},\mathbb{J}) \in PSS_{reg}(H,J)$, we may define a morphism of chain complexes
\begin{align*}
\Phi^{PSS}_{\mathcal{D}}: (C^{Morse}(f,g) \otimes \Lambda_{\omega})_{*+n} &\rightarrow CF_*(H,J) \\
p \otimes e^A &\mapsto \sum_{A \in \Gamma_{\omega}} \sum_{\hat{x}} n(p, (-A) \cdot \hat{x}) \hat{x},
\end{align*}
where for $q \in Crit(f)_{k-n}$, $\hat{y}=[y,v] \in \Per{H}_{k}$, $n(q,\hat{y})$ denotes the $\Z /2 \Z$-count of elements in the $0$-dimensional moduli space $\mathcal{M}(q,\hat{y};\mathcal{D})$  of finite energy maps $u \in C^{\infty}(\R \times S^1;M)$ which are $(\mathcal{H},\mathbb{J})$-Floer and which satisfy
\begin{align*}
\lim_{s \rightarrow - \infty} u(s,t) &\in W^u(q;f,g), \\
\lim_{s \rightarrow \infty} u(s,t) &= y(t), \; \text{and} \\
[\bar{u}] \# [v] &= 0 \in \Gamma_\omega,
\end{align*}
where $W^u(q;f,g)$ is the unstable manifold of $q$ with respect to $(f,g)$ and $\bar{u}: D^2 \rightarrow M$ is the disc map obtained by completing $u$ to a continuous map $\bar{u}: \bar{\R} \times S^1 \rightarrow M$ (note that since $u$ has finite energy and is by hypothesis $J^-$-holomorphic in some neighbourhood of $\lbrace - \infty \rbrace \times S^1$ for some fixed $J^- \in \mathcal{J}(M,\omega)$, removal of singularities for pseudo-holomorphic maps implies that there is some $m \in M$ such that $\lim_{s \rightarrow - \infty} u(s,t) =m$ for all $t \in S^1$). 
\par The map $\Phi^{PSS}_\mathcal{D}$ descends to a map on homology $\Phi^{PSS}_*: QH_{*+n}(M,\omega) \rightarrow HF_*(H)$ which is independent of the regular PSS data. 
\par It turns out (see \cite{PSS96} for more discussion on this point as well as the definitions of the relevant operations) that $\Phi^{PSS}_*$ is an isomorphism which intertwines the quantum product with the pair-of-pants product in Floer homology. This isomorphism permits the definition of the \textbf{Oh-Schwarz spectral invariants}
\begin{align*}
c_{OS}: QH_*(M,\omega) \setminus \lbrace 0 \rbrace \times C^{\infty}(S^1 \times M) &\rightarrow \R \\
(\alpha,H) &\mapsto c_{OS}(\alpha;H).
\end{align*}
Inspired by Viterbo's construction of spectral invariants via generating functions in \cite{V92}, these invariants were introduced by Schwarz in the symplectically aspherical case \cite{Sc00} and extended by Oh to more general symplectic manifolds in \cite{Oh05b}. These spectral invariants are defined by
\begin{align*}
c_{OS}(\alpha;H) &:= \inf \lbrace \lambda_H(\sigma): \sigma \in CF_*(H,J), \; [\sigma] = \Phi^{PSS}_*(\alpha) \rbrace.
\end{align*}
A cycle $\sigma \in CF_*(H,J)$ such that $[\sigma]=\Phi^{PSS}_*(\alpha)$ and $\lambda_{H}(\sigma)=c_{OS}(H;\alpha)$ is called \textbf{tight} (for $c_{OS}(\alpha;H)$). It is a non-trivial fact that such cycles always exist (see \cite{Us08} or \cite{Oh09}).
\par
The Oh-Schwarz spectral invariants are $C^0$-continuous in their Hamiltonian argument, take values in the spectrum of their Hamiltonian argument, and satisfy a bevy of formal properties (see Theorem III in \cite{Oh05b} for a representative list, for example) which make them useful in studying the behaviour of Hamiltonian diffeomorphisms. 

\subsubsection{The Gromov trick and positivity of intersections}\label{sec-Gromov trick}
We conclude this section by recalling the so-called `Gromov trick', which forms the basis of much of this paper by establishing that we may use pseudo-holomorphic techniques to analyze the graphs of Floer-type cylinders.
\begin{theorem}[1.4.C'. in \cite{Gr85}]\label{thm:Gromov}
Let $(\mathcal{H},\mathbb{J})$ be an adapted homotopy of Floer data, then there exists a unique almost complex structure $\tilde{\mathbb{J}}$ on $\R \times S^1 \times M$ with the property that a graph 
\begin{align*}
\tilde{u}: \R \times S^1 &\rightarrow \R \times S^1 \times M \\
(s,t) &\mapsto (s,t,u(s,t))
\end{align*}
is $(j_0, \tilde{\mathbb{J}})$-holomorphic if and only if $u$ satisfies Equation \ref{sFE}, where $j_0$ denotes the standard complex structure on the cylinder.
\end{theorem}
The almost complex structure $\tilde{\mathbb{J}}$ given by Theorem \ref{thm:Gromov} is the unique almost complex structure such that:
\begin{enumerate}
\item $\tilde{\mathbb{J}}$ restricts to $\mathbb{J}^s_t$ on the tangent space of $\lbrace (s,t) \rbrace \times M$ for each $(s,t) \in \R \times S^1$,
\item The distribution
\begin{align*}
\Gamma^{\mathcal{H}}_{(s,t,p)} &:= \langle \partial_s, \partial_t \oplus X^{\mathcal{H}^s}_t(p) \rangle \subset T_{(s,t,p)}(\R \times S^1 \times M),
\end{align*}
is invariant under $\tilde{\mathbb{J}}$,
\item The projection map 
\begin{align*}
pr_{cyl}: \R \times S^1 \times M &\rightarrow \R \times S^1
\end{align*}
is $(\tilde{\mathbb{J}},j_0)$-holomorphic.
\end{enumerate}
Note that it follows immediately from the above that the fibers $\lbrace (s,t) \rbrace \times M$ of $pr_{cyl}$ form a $\tilde{\mathbb{J}}$-holomorphic foliation of $\R \times S^1 \times M$.
\\ \\
The the primary importance of the Gromov trick for our purposes comes from the fact that in dimension $4$, one has the phenomenon of \textit{positivity of intersections} for pseudo-holomorphic curves, which we now recall for the convenience of the reader. Let $(M,J)$ be an almost complex $4$-manifold, $(S_i,j_i)$ (not necessarily closed) Riemannian surfaces and $u_i: (S_i,j_i) \rightarrow (M,J)$ $(j_i,J)$-holomorphic maps for $i=0,1$ with distinct images. If we denote by
\begin{align*}
I(u_0,u_1)&:= \lbrace (z_0,z_1) \in S_0 \times S_1: u_0(z_0)= u_1(z_1) \rbrace
\end{align*}
the pre-images of the intersection points of the images of the two pseudo-holomorphic curves, then the Carleman similarity principle implies that points of $I(u_0,u_1)$ are isolated in $S_0 \times S_1$. This implies that for any $(z_0,z_1) \in I(u_0,u_1)$, we may define $\iota(u_0,u_1;z_0,z_1) \in \Z$ --- the \textbf{local intersection number} of $u_0$ and $u_1$ at $(z_0,z_1)$ --- by fixing a contractible open set $V \subset M$ containing $u_0(z_0)=u_1(z_1)$ (and no other intersection points of the two curves), along with small disks $D_i \subset u_i^{-1}(V)$ centered at $z_i$, and $C^\infty$-small perturbation $v$ of $u_0$ such that $v(\partial D_0)$ and $u_1(\partial D_1)$ are disjoint, $v\vert_{D_0}$ is transverse to $u_1\vert_{D_1}$, and $v$ agrees with $u_0$ outside $D_0$. We then set
\begin{align*}
\iota(u_0,u_1;z_0,z_1)&:= v_0 \cdot_V u_1,
\end{align*}
where $v_0 \cdot_V u_1$ denotes the usual intersection number of the maps $v_0\vert_{D_0}$ and $u_1\vert_{D_1}$, viewed as maps into $V$. The contractibility of $V$ assures that this definition is independent of the choices made in this construction. For our purposes, positivity of intersections for pseudo-holomorphic curves may then be phrased as follows.

\begin{theorem}[Positivity of intersections]
Let $(M,J)$ be an almost complex $4$-manifold, $(S_i,j_i)$ (not necessarily closed) Riemannian surfaces and $u_i: (S_i,j_i) \rightarrow (M,J)$ $(j_i,J)$-holomorphic maps for $i=0,1$. Suppose that $u_0$ and $u_1$ have distinct images, then for every $(z_0,z_1) \in I(u_0,u_1)$, we have that $\iota(u_0,u_1;z_0,z_1) \geq 1$ and moreover $\iota(u_0,u_1;z_0,z_1)=1$ if and only if the intersection of $u_0$ and $u_1$ at $z_0$ and $z_1$ respectively is transverse. 
\end{theorem} 
We refer the interested reader to Appendix E of \cite{McSa12} for a proof and further details on positivity of intersections.

\subsection{Asymptotic analysis for pseudoholomorphic cylinders}\label{Sec:Asymptotic analysis}
The main analytic fact that gives us control over the asymptotic winding behaviour of Floer cylinders, as well as that of vector fields lying in the kernel of the Floer differential, is the following theorem which describes the asymptotic behaviour of solutions to an appropriately perturbed Cauchy-Riemann equation. This result is originally due to \cite{M03}, although the version we reproduce here for the convenience of the reader is from the appendix of \cite{Si08}.
\begin{theorem}\label{LocalAsympTheorem}
Let $w: [0, \infty) \times S^1 \rightarrow \R^{2n}$ satisfy the equation
\begin{align}\label{LocalAsymptoticFloerEquation}
\partial_s w +J_0 \partial_t w + (S(t)- \Delta(s,t))w&=0,
\end{align}
where $S: S^1 \rightarrow End(\R^{2n})$ is a smooth family of symmetric matrices and 
\begin{align*}
\Delta: [0,\infty) \times S^1 &\rightarrow End(\R^{2n})
\end{align*}
is smooth. Suppose that for $\beta \in \N^2,$ there exist constants $M_\beta, d>0$ such that 
\begin{align*}
\vert (\partial^\beta \Delta)(s,t) \vert &\leq M_\beta e^{-ds}, \; \text{and} \\
\vert (\partial^\beta w)(s,t) \vert &\leq M_\beta e^{-ds}.
\end{align*}
Then either $w \equiv 0$ or $w(s,t) = e^{\lambda s}(\xi(t) + r(s,t))$,
where $\lambda$ is a negative eigenvalue of the self-adjoint operator
\begin{align*}
\mathbf{A}: H^1(S^1; \R^{2n}) \subseteq L^2(S^1; \R^{2n}) &\rightarrow  L^2(S^1; \R^{2n}) \\
h &\mapsto -J_0( \partial_t - J_0S)h,
\end{align*}
$\xi: S^1 \rightarrow \R^{2n}$ is an eigenvector of $\mathbf{A}$ with eigenvalue $\lambda$, and $r$ satisfies the decay estimates 
\begin{align*}
\vert (\partial^\beta r)(s,t) \vert &\leq e^{-d' s}M_\beta' 
\end{align*}
for $d', M_\beta' >0$, for all $\beta \in \N^2$.
\end{theorem}
This theorem is useful in the following setting. 
Let $(H,J)$ be Floer regular. To any $x \in Per_0(H)$, we may assign the \textbf{asymptotic operator}
\begin{align*}
A_{x,J}: \Gamma(x^*TM) \rightarrow \Gamma(x^*TM)
\end{align*}
as follows. Let $\mathcal{V}:= \ker D pr_{S^1}$ denote the vertical tangent bundle of $S^1 \times M$, where $pr_{S^1}:S^1 \times M \rightarrow S^1$ denotes the obvious projection map.  Viewing $\xi \in \Gamma(x^*TM)$ as a section of the vertical tangent bundle $\mathcal{V}\vert_{\check{x}}:= \check{x}^* \mathcal{V}$ along the graph $\check{x}$ of $x$, we let $\check{X}_H:= \partial_t \oplus X_H \in \mathcal{X}(S^1 \times M)$, and we view $J=(J_t)_{t \in S^1}$ as an endomorphism of the vertical tangent bundle by setting $\check{J}_{t,x}:= J_t(x)$. $A_{x,J}$ is then defined by setting $A_{x,J} (\xi) := -\check{J} \mathcal{L}_{\check{X}_H} \xi$, where $\mathcal{L}_XY$ denotes the Lie derivative of $Y$ along $X$. $A_{x,J}$ extends to an unbounded self-adjoint operator with discrete spectrum (still denoted $A_{x,J}$) from $W^{1,2}(x^*TM)$ to $L^2(x^*TM)$. 
\par By taking an exponential chart as in Section \ref{LocalLinkSec} on a neighbourhood $\tilde{\mathcal{O}}$ of $\hat{x} \in \Per{H}$, Floer's equation may be written in the local coordinates provided by this chart in the form of Equation $\ref{LocalAsymptoticFloerEquation}$, with $A_{x,J}$ being sent via these coordinates to $\mathbf{A}$. Following \cite{Si08}, we make the following definition.
\begin{definition}
Let $x \in \mathcal{L}_0(M)$ and suppose that $\lim_{s \rightarrow \infty} u_s \equiv x$ for a map $u: \R \times S^1 \rightarrow M$. For any $R>0$, a map
\begin{align*}
U^+: [R, \infty) &\rightarrow \Gamma(x^*TM)
\end{align*}
will be said to be a \textbf{positive asymptotic representative} of $u$ if 
$u(s,t) = Exp(U^+(s))(t)$ for all $(s,t) \in [R, \infty) \times S^1$, where $Exp$ is as in Section \ref{LocalLinkSec}. The notion of a \textbf{negative asymptotic representative} of $u$, 
\begin{align*}
U^-: (-\infty, -R] &\rightarrow \Gamma(x^*TM)
\end{align*}
is defined in the obvious analogous manner.
\end{definition}
Every Floer-type cylinder considered in this paper admits, due to exponential convergence at the ends, essentially unique positive and negative asymptotic representatives,  determined up to a restriction of the domains of $U^\pm$ to larger values of $\vert R \vert$. \par 
The main result that we will need from \cite{Si08} is the following (this is Theorem $2.2$ of the quoted work, paraphrased for our setting). 
\begin{theorem}\label{PosAsymp}
Let $(H,J)$ be a non-degenerate Floer pair, $x \in Per_0(H)$ and let $u, v$ solve Equation \ref{sFE} for $s >> 0$ (resp. for $s << 0$), where the adapted homotopy used in defining Equation \ref{sFE} satisfies $(H^+,J^+)=(H,J)$ (resp. $(H^-,J^-) = (H,J)$). Suppose moreover that $u_s$ and $v_s$ both converge to $x$ as $s \rightarrow \infty$ (resp. $ s \rightarrow -\infty$). Let $U$ and $V$ be positive (resp. negative) asymptotic representatives of $u$ and $v$ respectively. Then either $U \equiv V$ or there exists a strictly negative (resp. strictly positive) eigenvalue $\lambda \in \sigma(A_{x,J})$ and an eigenvector $\xi$ with eigenvalue $\lambda$ such that
\begin{align*}
(U-V)(s,t) &=e^{\lambda s}(\xi(t) + r(s,t)),
\end{align*}
where the remainder term satisfies the decay estimates $\vert \nabla_s^i \nabla_t^i r(s,t)\vert \leq M_{ij} e^{-ds}$ for all $(i,j) \in \N^2$ and some $M_{i,j},d>0$ (resp. $d<0$).
\end{theorem}
Whenever $u,v$ and $x$ are as above with $u$ and $v$ restricting to distinct maps outside of every compact set of $\R \times S^1$, we will write $\xi^+_{u,v}$ (resp. $\xi^-_{u,v}$) for the eigenvectors of $A_{x,J}$ whose existence is guaranteed by the above theorem. We will call $\xi^\pm_{u,v}$ the \textbf{positive (resp. negative) asymptotic eigenvector of $v$ relative $u$}. Note that the above result only requires that $u$ and $v$ solve Equation \ref{sFE} on some neighbourhood of $s=\infty$ (resp. $s= -\infty$), and that, for $(\mathcal{H}, \mathbb{J}) \in \HJ$, the trivial cylinder $v(s,t)=x(t)$ is always a solution to Equation \ref{sFE} outside some compact set. Provided $u$ is not the trivial cylinder $u(s,t)=x(t)$, we will write $\xi^\pm_u:= \xi^\pm_{u,x}$ and call these the (positive and negative) \textbf{asymptotic eigenvectors of $u$}.
\par
This asymptotic information becomes especially useful when combined with the following fact (see \cite{HWZ95} p. 285 or \cite{Si08} p.1637).
\begin{proposition}\label{NonVanish}
If $\xi \in \Gamma(x^*TM)$ is an eigenvector of $A_{x,J}$ (and hence is not identically zero) then for all $t \in S^1$, $\xi(t) \neq 0$.
\end{proposition}
Since Theorem \ref{PosAsymp} tells us that for sufficiently large $\lvert R \rvert \in \R$ the difference between the asymptotic representatives of Floer-type cylinders behaves to leading order like an eigenvector of $A_{x,J}$, the fact that these are nowhere vanishing implies:
\begin{cor}\label{EndsCor}
Let $u,v: \R \times S^1 \rightarrow M$ be distinct finite energy solutions of Equation \ref{sFE}, then there is a compact subset $K \subseteq \R \times S^1$ such that $u(s,t) = v(s,t)$ only if $(s,t) \in K$.
\end{cor}
These results become even stronger in the case when $\dim M=2$, as in this case Proposition \ref{NonVanish} implies that eigenvectors of the asymptotic operator have a well-defined winding number, once we fix a trivialization of $x^*TM$ via a choice of capping disk. More precisely, when $M=\Sigma$, if $\hat{x} \in \cL{\Sigma}$, and $T_{\hat{x}}: S^1 \times (\R^2,\omega_0) \rightarrow (x^*T\Sigma,\omega)$ is a symplectic trivialization as in Section \ref{Braids}, then for any eigenvector $\xi$ of $A_{x,J}$, the map $t \mapsto T_{\hat{x}}(t)^{-1}\xi(t)$ has a well-defined winding number $wind(\xi; \hat{x})$, by Proposition \ref{NonVanish}. Proposition \ref{prop:HLinkEqualsLink} then gives the following corollary.
\begin{cor}\label{EigenvectorWindingCap}
Let $u,v$ be distinct finite energy solutions of Equation \ref{sFE} with 
\begin{align*}
\lim_{s \rightarrow - \infty} u_s &= \lim_{s \rightarrow - \infty} v_s= x.
\end{align*}
Then there exists $R>0$ such that for all $s < - R$ and any capping $\hat{x}=[x,\alpha]$, we have 
\begin{align*}
\ell(\hat{v}_{s}^{\alpha},\hat{u}^{\alpha}_{s}) &= wind(\xi_{u,v}^-; \hat{x}),
\end{align*} 
where $\hat{u}^{\alpha}_{s}$ (resp. $\hat{v}^{\alpha}_s$) denotes the capping of $u_s$ (resp. $v_s$) such that $[x,\alpha]$ and $\hat{u}_s$ (resp. $\hat{v}_s$) are $0$-homotopic. The analogous statement when $\lim_{s \rightarrow \infty} u_s= \lim_{s \rightarrow \infty} v_s=x$ also holds.
\end{cor}

If we combine the positivity of intersection of holomorphic curves in dimension $4$ with the foregoing discussion, we arrive at the principal point of this section (recall from the discussion following Proposition \ref{prop:HLinkEqualsLink} that $\ell(\hat{x},\hat{y}):= L_0(\hat{0},\lbrace \hat{x},\hat{y} \rbrace)$).
\begin{lemma}\label{EndsLinkLemma}
Let $(H^\pm,J^\pm)$ be non-degenerate Floer pairs, $(\mathcal{H},\mathbb{J}) \in \HJ$ and let $u,v \in C^\infty(\R \times S^1;\Sigma)$ be distinct finite energy solutions to Equation \ref{sFE} for $(\mathcal{H},\mathbb{J})$. Then for any lifts $\hat{u}, \hat{v}$ of $u,v: \R \rightarrow \mathcal{L}_0(\Sigma)$, the function $\ell_{\hat{u},\hat{v}}(s):= \ell(\hat{u}_s,\hat{v}_s)$ is non-decreasing, locally constant, and well-defined for all but finitely many values $s \in \R$. Moreover, for $s,s' \in dom(\ell_{\hat{u},\hat{v}})$, with $s < s'$, $\ell_{\hat{u},\hat{v}}(s) \neq \ell_{\hat{u},\hat{v}}(s')$ if and only if there exists $s_0 \in (s,s')$ and some $t_0 \in S^1$ such that $u(s_0,t_0)=v(s_0,t_0)$.
\end{lemma}
\begin{proof}
That $\ell_{\hat{u},\hat{v}}$ has only finitely many points at which it is ill-defined follows the fact that, by definition, $\ell(\hat{u}_{s_0},\hat{v}_{s_0})$ is undefined only when there exists $t_0 \in S^1$ such that $u(s_0,t_0)=v(s_0,t_0)$. By Corollary \ref{EndsCor}, the set of all such $(s_0,t_0) \in \R \times S^1$ must lie inside some compact set, and we may then apply Theorem \ref{thm:Gromov} to choose an almost complex structure on $\R \times S^1 \times \Sigma$ such that the graphs $\tilde{u}$ and $\tilde{v}$ are pseudoholomorphic, whence all such intersections must be isolated, and so finite in number. That $\ell_{\hat{u},\hat{v}}$ is non-decreasing follows by applying the positivity of intersections for holomorphic curves in dimension $4$ at these intersections, which implies moreover that any such intersection contributes strictly positively to the change in $\ell_{\hat{u},\hat{v}}(s)$ as $s$ increases and passes from one connected component of $dom(\ell_{\hat{u},\hat{v}})$ to another.
\end{proof}
\begin{definition}\label{def-link-infty}
For $(H^\pm,J^\pm)$ non-degenerate Floer pairs, $(\mathcal{H},\mathbb{J}) \in \HJ$ and $u,v \in \mathcal{M}(\hat{x},\hat{y};\mathcal{H},\mathbb{J})$, $u \neq v$, we define 
\begin{align*}
\ell_{\pm \infty}(u,v) &:= \lim_{s \rightarrow \pm \infty} \ell(\hat{u}_s,\hat{v}_s),
\end{align*}
where $\hat{u}_s$ and $\hat{v}_s$ are the natural cappings of $u_s$ and $v_s$ (cf. Definition \ref{def:NaturalCapping}).
\end{definition}
\begin{remark}
We recall here for the convenience of the reader that the moduli space $\mathcal{M}(\hat{x},\hat{y};\mathcal{H},\mathbb{J})$ contains, by definition, only cylinders which are $0$-homotopies between $\hat{x}$ and $\hat{y}$ (cf. Section \ref{Sec:Floer} and \ref{Sec-ContMaps}). This ensures in particular that the notion of ``the natural cappings of $\hat{u}_s$ and $\hat{v}_s$'' used above is well-defined.
\end{remark}
Note that the previous lemma implies that these quantities exist and are finite. Indeed, if $\hat{u}_s$ and $\hat{v}_s$ tend to $\hat{x}$ as $s \rightarrow \pm \infty$, then $\ell_{\pm \infty}(u,v)= wind(\xi^\pm_{u,v};\hat{x})$ by Corollary \ref{EigenvectorWindingCap}, while if $\lim_{s \rightarrow \pm \infty} u_s=\hat{x}$ and $\lim_{s \rightarrow \pm \infty} v_s=\hat{y}$ with $x \neq y$, then $\ell_{\pm \infty}(u,v)=\ell(\hat{x},\hat{y})$.

\subsubsection{Winding of eigenvectors of $A_{x,J}$}\label{SubSec: Winding}
We summarize here some necessary facts from \cite{HWZ95} on the winding numbers of the eigenvectors of $A_{x,J}$ which appeared in the previous subsection (while \cite{HWZ95} works in the aspherical case, our previous discussion makes clear how this winding number depends on the choice of cappings of $x$ and $y$ and this is all that is needed to extend the results to capped orbits). For a loop $x \in \mathcal{L}_0(\Sigma)$, let $\pi_2(\Sigma;x)$ denote the set of homotopy classes of capping disks for $x$.
\begin{proposition}\label{SpectrumWindProps}
Let $(H,J)$ be a non-degenerate Floer pair, $x \in Per_0(H)$ with $H$ non-degenerate and $J: S^1 \rightarrow \mathcal{J}(\Sigma,\omega)$ arbitrary. There is a well-defined function
\begin{align*}
W=W_{x,J}: \pi_2(\Sigma; x) \times \sigma(A_{x,J}) &\rightarrow \Z \\
(\alpha,\lambda) &\mapsto wind(T_{[x,\alpha]}^{-1} \circ \xi),
\end{align*}
where $\xi \in \Gamma(x^*T\Sigma)$ is any eigenvector with eigenvalue $\lambda$. Moreover, $W$ satisfies the following properties
\begin{enumerate}
\item For any $\alpha \in \pi_2(\Sigma;x)$, $\lambda < \lambda' \; \Rightarrow  \; W(\alpha,\lambda) \leq W(\alpha,\lambda')$.
\item For any $\alpha \in \pi_2(\Sigma;x)$, and any $k \in \Z$, 
$\sum_{ \lambda \in W_\alpha^{-1}(k)} \dim E_\lambda =2$, \\
where $W_\alpha(\lambda)=W(\alpha,\lambda)$, and $E_\lambda$ is the eigenspace associated to the eigenvalue $\lambda$.
\item For any $A \in \pi_2(\Sigma)$, $W(A \cdot \alpha, \lambda)= W(\alpha, \lambda) + c_1(A)$.
\end{enumerate}
\end{proposition}
In view of the control over the sign of the eigenvalue provided by Theorem $\ref{PosAsymp}$, combined with the monotonicity of the winding number provided by item $(1)$ of the above proposition, we follow \cite{HWZ95} in making the following definition.
\begin{definition}
For $(H,J)$ and $\hat{x}=[x,\alpha] \in Per_0(H)$ as above, define
\begin{align*}
a(\hat{x})= a(\hat{x};H) &:= \max_{\lambda \in \sigma(A_{x,J}) \cap (- \infty, 0)} W(\alpha,\lambda) \\
b(\hat{x})= b(\hat{x};H) &:= \min_{\lambda \in \sigma(A_{x,J}) \cap (0, \infty)} W(\alpha,\lambda).
\end{align*} 
\end{definition}
Remark that by the monotonicity of $W$ and by Lemma $\ref{EndsLinkLemma}$ (see the remark following Definition $\ref{def-link-infty}$) we have the following corollary.
\begin{cor}\label{EmergeConvergeLinking}
Let $(H^\pm,J^\pm)$ be non-degenerate Floer pairs, $(\mathcal{H},\mathbb{J}) \in \HJ$, $x^\pm_i \in \Per{H^\pm}$ for $i=0,1$.
\begin{enumerate}
\item If $u_i \in \mathcal{M}(\hat{x}_0^-,\hat{x}_i^+; \mathcal{H},\mathbb{J})$, $i=0,1$, then $b(\hat{x}_0^-) \leq \ell_{- \infty}(u_0,u_1)$. 
\item If $u_i \in \mathcal{M}(\hat{x}_i^-,\hat{x}_0^+; \mathcal{H},\mathbb{J})$, $i=0,1$, then $\ell_{\infty}(u_0,u_1) \leq a(\hat{x}^+_0)$. 
\end{enumerate}
\end{cor}
The result which relates this discussion to the behaviour of the Floer complex is the following:
\begin{theorem}[\cite{HWZ95} Theorem 3.10]\label{thm:apEqualsCZindex}
\begin{align}\label{eqn:apEqualCZindex}
-\mu(\hat{x}) &= 2a(\hat{x}) + p(\hat{x}),
\end{align}
where $p([x,\alpha]) =0$ if there exists $\lambda \in \sigma(A_{x,J}) \cap (0,\infty)$ such that $W(\alpha,\lambda) = a(\hat{x})$ and $p([x,\alpha])=1$ otherwise.
\end{theorem}
\begin{remark}
Note that our sign convention for the Conley-Zehnder index is the negative of the convention adopted in \cite{HWZ95}.
\end{remark}
\begin{cor}\label{cor:abEqualsCZindex}
\begin{align*}
-\mu(\hat{x}) &= a(\hat{x}) + b(\hat{x}).
\end{align*}
\end{cor}
\begin{proof}
The point is that $b(\hat{x})=a(\hat{x})+p(\hat{x})$. To see this, write $\hat{x}=[x,\alpha]$, and note first that $b(\hat{x}) \in \lbrace a(\hat{x}), a(\hat{x}) + 1 \rbrace$. This follows directly from the definitions of $a(\hat{x})$ and $b(\hat{x})$ combined with the surjectivity of $W(\alpha;-): \sigma(A_{x,J}) \rightarrow \Z$ which is implied by item 2 of Proposition \ref{SpectrumWindProps}, together with the monotonicity of $W(\alpha;-)$ and the fact that $0 \not \in \sigma(A_{x,J})$ when $H$ is non-degenerate. By Theorem \ref{eqn:apEqualCZindex}, $p(\hat{x})=0$ implies that $b(\hat{x}) \leq a(\hat{x})$. On the other hand, if $p(\hat{x})=1$, then Theorem \ref{eqn:apEqualCZindex} implies that $b(\hat{x}) > a(\hat{x})$. In either case, we obtain $b(\hat{x})= a(\hat{x}) + p(\hat{x})$.
\end{proof}
\begin{lemma}\label{CZindexAsympWindCompLemma}
Let $k \in \Z$.
\begin{enumerate}
\item If $\mu(\hat{x};H) \in \lbrace 2k -1, 2k \rbrace$, then $a(\hat{x};H)=-k$. 
\item If $\mu(\hat{x};H) \in \lbrace 2k, 2k+1 \rbrace$, then $b(\hat{x};H)=-k$.
\end{enumerate}
\end{lemma}
\begin{proof}
We prove (1) with the argument for (2) being entirely analogous (but additionally making use of the fact that $b(\hat{x})=a(\hat{x}) + p(\hat{x})$ established in the proof of Corollary \ref{cor:abEqualsCZindex}). Theorem \ref{thm:apEqualsCZindex} states that $-\mu(\hat{x})= 2a(\hat{x}) + p(\hat{x})$, where $p(\hat{x}) \in \lbrace 0,1 \rbrace$ is the parity of $\mu(\hat{x})$. Consequently, if $\mu(\hat{x})=2k-1$, we see that
\begin{align*}
1-2k&=2a(\hat{x}) +1
\end{align*}
and therefore $a(\hat{x})=-k$. Similarly, if $\mu(\hat{x})=2k$, then
\begin{align*}
-2k&=2a(\hat{x})
\end{align*}
and so $a(\hat{x})=-k$, as claimed.
\end{proof}

\begin{cor}\label{Cor: Cont moduli space is 0 or 1}
Let $(H^\pm,J^\pm)$ be non-degenerate Floer pairs, $(\mathcal{H},\mathbb{J}) \in \HJ(H^-,J^-;H^+,J^+)$ and suppose that $\hat{x}^\pm \in \Per{H^\pm}$ satisfy $\mu(\hat{x}^\pm)=2k+1$ for some $k \in \Z$. Then $\vert \mathcal{M}(\hat{x}^-,\hat{x}^+;\mathcal{H},\mathbb{J}) \vert \in \lbrace 0, 1 \rbrace$. 
\end{cor}
\begin{proof}
Suppose for the sake of contradiction that there exist $u,v \in \mathcal{M}(\hat{x}^-,\hat{x}^+;\mathcal{H},\mathbb{J})$, $u \neq v$. Then we have by Lemma \ref{EndsLinkLemma} and Lemma \ref{CZindexAsympWindCompLemma}
\begin{align*}
-k=b(\hat{x}^-) \leq \ell_{- \infty}(u,v) &\leq \ell_{\infty}(u,v) \leq a(\hat{x}^+) = -k -1, 
\end{align*}
which is a contradiction. The lemma follows.
\end{proof}

Recall from Section \ref{Sec:Floer} that to any $(\mathcal{H},\mathbb{J}) \in \HJ(H^-,J^-;H^+,J^+)$, we associate the operator $\mathcal{F}_{\mathcal{H},\mathbb{J}}: C^{\infty}(\R \times S^1;\Sigma) \rightarrow \mathcal{E}$. Whenever $\mathcal{F}_{\mathcal{H},\mathbb{J}}(u)=0$, for $u \in C^\infty(\R \times S^1; \Sigma)_{\hat{x}^-,\hat{x}^+}$, $T_u \mathcal{E}$ splits canonically as $T_u C^{\infty}(\R \times S^1;\Sigma)_{\hat{x}^-,\hat{x}^+} \oplus \mathcal{E}_u$. In such a case, we denote by $D \mathcal{F}_{\mathcal{H},\mathbb{J}}$ the projection of the differential of $\mathcal{F}_{\mathcal{H},\mathbb{J}}$ onto $\mathcal{E}_u$, and we call $D \mathcal{F}_{\mathcal{H},\mathbb{J}}$ the \textbf{linearized Floer operator}. The transversality of $\mathcal{F}_{\mathcal{H},\mathbb{J}}$ to the $0$-section of $\mathcal{E}$ at $u$ is equivalent to the surjectivity of $(D\mathcal{F}_{\mathcal{H},\mathbb{J}})_u$, which is in turn related to the behaviour of its kernel by the Fredholm property.
The following result is essentially proved in \cite{HWZ95} as Proposition $5.6$ and serves to give significant control over elements in the kernel of the linearized Floer operator. We give a simple proof here in the Floer-theoretic setting for the convenience of the reader.
\begin{proposition}\label{FloerKernelWinding}
Let $(H^\pm,J^\pm)$ be non-degenerate Floer pairs, let $(\mathcal{H},\mathbb{J}) \in \HJ$, $u \in \mathcal{M}(\hat{x}^-,\hat{x}^+;\mathcal{H},\mathbb{J})$, and let $\xi \in \ker (D\mathcal{F}_{\mathcal{H},\mathbb{J}})_u$. Suppose that $\xi \not \equiv 0$ and denote by $Z(\xi)$ the algebraic count of the number of zeros of $\xi$, then $Z(\xi)$ is finite and satisfies the inequality $0 \leq Z(\xi) \leq a(\hat{x}^+) -b(\hat{x}^-)$. Moreover, $Z(\xi)=0$ if and only if $\xi$ is nowhere vanishing.
\end{proposition}
\begin{proof}
It is a standard result in Floer theory (see for instance \cite{Sa97}, Section $2.2$) that for any $u \in \mathcal{M}(\hat{x}^-,\hat{x}^+;\mathcal{H},\mathbb{J})$, any element $\xi \in \ker (D\mathcal{F}_{\mathcal{H},\mathbb{J}})_u$ may be expressed (with respect to the unitary trivialization $\Phi: \R \times S^1 \times (\R^{2},J_0) \rightarrow u^*(T \Sigma,J)$ along $u$ induced by the cappings of $\hat{x}^-$ and $\hat{x}^+$) as solving an equation of the form
\begin{align*}
\partial_s \xi + J_0 \partial_t \xi + S \xi&=0,
\end{align*}
where we may write $S$ on the positive and negative ends as 
\begin{align*}
S^{\pm}(s,t)&= \Phi^{-1}A_{x^\pm,J} - \Delta^{\pm}(s,t),
\end{align*}
with $\Delta^\pm$ satisfying the decay estimates of Theorem \ref{LocalAsympTheorem}. Consequently, any $\xi \in \ker (D \mathcal{F})_u$ that is not identically zero must be non-vanishing outside of some compact neighbourhood of $\R \times S^1$, and the Carlemann similarity principle implies that any zeros of such a $\xi$ are isolated and must contribute strictly positively to $Z(\xi)$. Consequently, we see that $Z(\xi)$ is finite and non-negative, with $Z(\xi)=0$ exactly when $\xi$ is nowhere vanishing.
\par
To see that $Z(\xi) \leq a(\hat{x}^+) -b(\hat{x}^-)$, we take $R > 0$ sufficiently large so that $\xi$ is non-vanishing outside of $(-R,R) \times S^1$ and consider the homotopy of $2$-braids in $\R^2$ induced by $h(s)=(0, \Phi^{-1} \xi_s) \in \mathcal{L}_0(\R^2)^2$, $s \in [-R,R]$. Theorem \ref{LocalAsympTheorem} implies that for $R>0$ sufficiently large,
\begin{align*}
\ell(0,\xi_{-R}) = wind(\Phi^{-1} \xi_{-R}) &\geq b(\hat{x}^-), \; and \\
\ell(0,\xi_{R}) = wind(\Phi^{-1} \xi_{R}) &\leq a(\hat{x}^+),
\end{align*}
and the algebraic count zeros of $\xi$ correspond to the algebraic count of the intersections of the graphs of the strands of $h$ from which the proposition follows.
\end{proof}
The following corollary is essentially the linear analogue of Corollary \ref{Cor: Cont moduli space is 0 or 1}.
\begin{cor}\label{Cor: Automatic Regularity for Odd orbits}
Let $(H^\pm,J^\pm)$ be Floer regular, let $(\mathcal{H},\mathbb{J}) \in \HJ$, and let $\hat{x}^\pm \in \Per{H^\pm}$ have $\mu(\hat{x}^\pm)=2k+1$ for some $k \in \Z$. Then $(\mathcal{H},\mathbb{J})$ is $(\hat{x}^-,\hat{x}^+)$-regular.
\end{cor}
\begin{proof}
If $(H^-,J^-)=(H^+,J^+)$ and $(\mathcal{H},\mathbb{J})$ is the constant homotopy of Floer data, then this is automatic by Floer regularity of $(H^\pm,J^\pm)$. The statement is also vacuously true if $\mathcal{M}(\hat{x}^-,\hat{x}^+;\mathcal{H},\mathbb{J})$ is empty. So we may suppose that $(\mathcal{H},\mathbb{J})$ is not $\R$-invariant and that there exists some $u \in \mathcal{M}(\hat{x}^-,\hat{x}^+;\mathcal{H},\mathbb{J})$. To see that $u$ is regular, note that $D \mathcal{F}_{\mathcal{H},\mathbb{J}}$ has Fredholm index $0$, and so to prove that $(D \mathcal{F}_{\mathcal{H},\mathbb{J}})_u$ is surjective, it suffices to show that its kernel is trivial. Suppose to the contrary that $\xi \in \ker  (D \mathcal{F}_{\mathcal{H},\mathbb{J}})_u$ is a non-trivial vector field along $u$. By the previous proposition, we must have that the algebraic count of its zeros $Z(\xi)$ satisfies
\begin{align*}
0 \leq &Z(\xi) \leq a(\hat{x}^+) - b(\hat{x}^-) = -k - (k-1) = -1,
\end{align*}
where the second to last equality follows from Lemma \ref{CZindexAsympWindCompLemma}. Clearly this is a contradiction, so we conclude that $\ker  (D \mathcal{F}_{\mathcal{H},\mathbb{J}})_u=0$ and so $u$ is regular.
\end{proof}

In the case of even index orbits, we have a somewhat weaker conclusion that will still be of use to us.
\begin{cor}\label{Automatic Codim 1 for Even orbits}
Let $(H^\pm,J^\pm)$ be Floer regular, let $(\mathcal{H},\mathbb{J}) \in \HJ$, and let $\hat{x}^\pm \in \Per{H^\pm}$ have $\mu(\hat{x}^\pm)=2k$ for some $k \in \Z$. Then for every $u \in \mathcal{M}(\hat{x}^-,\hat{x}^+;\mathcal{H},\mathbb{J})$, 
\begin{align*}
\dim \ker  (D \mathcal{F}_{\mathcal{H},\mathbb{J}})_u &\leq 1
\end{align*}
\end{cor}
\begin{proof}
As in the proof of the previous corollary, it suffices to consider the case where $(\mathcal{H},\mathbb{J})$ is not $\R$-invariant and $\mathcal{M}(\hat{x}^-,\hat{x}^+;\mathcal{H},\mathbb{J})$ is non-empty. Letting $u \in \mathcal{M}(\hat{x}^-,\hat{x}^+;\mathcal{H},\mathbb{J})$, consider the behaviour of some $\eta_1, \eta_2 \in \ker  (D \mathcal{F}_{\mathcal{H},\mathbb{J}})_u$. Let 
\begin{align*}
\Phi: \R \times S^1 \times \C &\rightarrow u^*T\Sigma
\end{align*}
be a unitary trivialization of the tangent bundle along $u$ which extends over the cappings of $\hat{x}^\pm$. Since $\eta_1$ and $\eta_2$ satisfy an appropriate perturbed Cauchy-Riemann equation, Theorem \ref{LocalAsympTheorem} implies that $v_i(s,t):=(\Phi^{-1}\circ \eta_i)(s,t)$, $i=1,2$ satisfies
\begin{align*}
\lim_{s \rightarrow - \infty} v_i(s,t)= \Phi^{-1}(-\infty,t) \circ \xi_i(t)
\end{align*}
for $\xi_i \in \Gamma^{\infty}((x^-)^*T\Sigma)$ an eigenvector of the asymptotic operator $A_{x^-,J^-}$ associated to $x^-$, whose associated eigenvalue is positive. Let $\lambda_0$ be the smallest positive eigenvalue of $A_{x^-,J^-}$ and note that every eigenvector $\xi' \in E_{\lambda_0}$ is such that $\Phi^{-1}(-\infty,t) \circ \xi'(t)$ has winding number $-k$ by Corollary \ref{cor:abEqualsCZindex}, and by the same reasoning as in the proof of that corollary, $\dim E_{\lambda_0} =1$. We claim first that $\xi_i \in E_{\lambda_0}$ for $i=1,2$. Indeed, suppose not, then without loss of generality we may suppose that $\xi_1 \in E_\lambda$ for $\lambda > \lambda_0$. This implies that $wind( \Phi^{-1}(-\infty,t) \circ \xi_1(t)) \geq -k +1$, by points $(1)$ and $(2)$ of Proposition \ref{SpectrumWindProps} combined with the fact that $a(\hat{x}^-)=b(\hat{x}^-)$ (so the only other eigenspace whose eigenvector have winding $-k$ is the eigenspace associated to the largest negative eigenvalue of $A_{x^-,J^-}$). Whence, 
\begin{align*}
wind (\Phi^{-1}(-R,t)\circ \eta_1(-R,t)) &\geq -k +1
\end{align*}
for $R >0$ sufficiently large. As in the proof of Proposition \ref{FloerKernelWinding}, we have that
\begin{align*}
Z(\eta_1) &\leq a(\hat{x}^+) - wind (\Phi^{-1}(-R,t)\circ \eta_1(-R,t))
\end{align*} 
for $R>0$ sufficiently large, but this implies
\begin{align*}
Z(\eta_1) &\leq -k - (-k +1)=-1
\end{align*}
which violates the positivity of $Z(\eta_1)$. So $\xi_i \in E_{\lambda_0}$ for $i=1,2$. 
\par
Now consider $\delta:=\eta_1 - \eta_2$. $\delta$ once again satisfies a perturbed Cauchy-Riemann equation satisfying the required asymptotic decay conditions, because $\eta_1$ and $\eta_2$ do. Let $\xi_\delta(t) = \lim_{s \rightarrow - \infty} \delta(s,t)$ be the associated negative asymptotic eigenvector. We claim that $\xi_\delta \in E_{\lambda_0}$. Indeed, if not then we may repeat the above argument with $\delta$ in the place of $\eta_1$ and $\xi_\delta$ in the place of $\xi_1$ to derive a contradiction. It follows that the map which sends $\eta \in \ker  (D \mathcal{F}_{\mathcal{H},\mathbb{J}})_u$ to its negative asymptotic eigenvector in $E_{\lambda_0}$ is linear, and it is injective by Theorem \ref{LocalAsympTheorem} (since the negative asymptotic eigenvector of $\eta$ is zero if and only if $\eta$ vanishes on the negative end, and thus by the Carlemann similarity principle, $\eta$ vanishes everywhere). Since $\dim E_{\lambda_0} =1$, this proves the claim.
\end{proof}

\section{Constructing chain-level continuation maps with prescribed behaviour}\label{Ch: Construct chain maps}
In this section, we introduce a technique for designing chain-level continuation maps such that certain $0$-dimensional moduli spaces may be guaranteed to be non-empty. The basic idea is to show that, given any collection of smooth cylinders which topologically \textit{could} be a collection of ($s$-dependent) Floer-type cylinders (in the sense that their graphs intersect only positively and they exhibit the appropriate asymptotic winding behaviour), we may --- after a series of suitable small perturbations --- always find regular Floer homotopy data $(\mathcal{H},\mathbb{J})$ such that the perturbed cylinders are all $(\mathcal{H},\mathbb{J})$-Floer. 
\\ \\
For the remainder of the section, we fix Floer regular $(H^\pm,J^\pm) \in C^{\infty}(S^1 \times M) \times C^{\infty}(S^1; \mathcal{J}(M,\omega))$.
\begin{definition}\label{Def-Model}
We will say that a finite collection of smooth maps $\lbrace u_i: \R \times S^1 \rightarrow \Sigma \rbrace$, $i=1,\ldots,k$ is a \textbf{model for a continuation cobordism} (from $(H^-,J^-)$ to $(H^+,J^+)$) if there exists $(\mathcal{H},\mathbb{J}) \in \HJ(H^-,J^-;H^+,J^+)$ such that $u_i$ is an $(\mathcal{H},\mathbb{J})$-Floer cylinder with finite energy for each $i=1,\ldots,k$. Such a model for a continuation cobordism will be called an \textbf{$(\mathcal{H},\mathbb{J})$-model}.
\end{definition}
If we are doing Floer theory on a surface, then the results of Section \ref{Sec:Asymptotic analysis} tell us that the graphs $\tilde{u}_i$, when restricted to a sufficiently large compact set $[-K,K] \times S^1$, define a braid cobordism having only positive intersections (outside this compact set, the maps $u_i$ solve the $s$-independent Floer equations on the ends, and their behaviour is controlled by the winding behaviour of the eigenvectors of the asymptotic operator as discussed in Section \ref{Sec:Asymptotic analysis}).
\\ \\ 
Recall (Definition \ref{Def: Positive Cobordism}) that a braid cobordism $h=(h_1,\ldots,h_k)$ is said to be positive if the graphs of its strands are all transverse with positive intersections. The main result of this section is to show that this essentially topological condition, combined with the obvious necessary condition on the asymptotic behaviour of the strands, is sufficient to guarantee the existence of some regular homotopy of Floer data $(\mathcal{H},J)$ such that each cylinder is $(\mathcal{H},J)$-Floer.
\begin{definition}
We will say that a finite collection of smooth maps $u_i: \R \times S^1 \rightarrow \Sigma$, $i=1, \ldots,k$, is a \textbf{pre-model for a continuation cobordism} (from $(H^-,J^-)$ to $(H^+,J^+)$) if they satisfy the following
\begin{enumerate}
\item There exists $K >0$ such that the maps $u_i \vert_{[-K,K]}$ are the strands of a positive braid cobordism, and
\item for the same $K$ as above, $u_i \vert_{(-\infty,-K]}$ satisfies the $(H^-,J^-)$-Floer equation and has finite $(H^-,J^-)$-energy and, while $u_i \vert_{[K,\infty)}$ satisfies the $(H^+,J^+)$-Floer equation and has finite $(H^+,J^+)$ energy, for each $i=1,\ldots,k$.
\end{enumerate}
\end{definition}
Our aim is to show:
\begin{theorem}\label{Theorem: Pre-models give models of cont cobordisms}
If $u_i: \R \times S^1 \rightarrow \Sigma$, $i=1, \ldots,k$ defines a pre-model for a continuation cobordism from $(H^-,J^-)$ to $(H^+,J^+)$, then there exists $u_i': \R \times S^1 \rightarrow \Sigma$, $i=1, \ldots,k$ which is a model for a continuation cobordism from $(H^-,J^-)$ to $(H^+,J^+)$ and such that each $u_i'$ differs from $u_i$ only by a $C^0$-small homotopy inside a compact subset of $(-K,K) \times S^1$ (where here $K>0$ is as in the definition of a pre-model for a continuation cobordism).
\end{theorem}
The principal virtue of this result is that the existence or non-existence of pre-models may largely be reduced to the question of the relative topologies of the collections of capped orbits (\textit{qua} capped braids) being connected by these cylinders, along with asymptotic information provided by the asymptotic operator. 
\par
The argument itself is straightforward and rather hands-on. Our goal is to show that given a pre-model for a cobordism from $(H^-,J^-)$ to $(H^+,J^+)$, we may build a homotopy of Floer data $(\mathcal{H},J) \in \HJ(H^-,J^-;H^+,J^+)$ such that each $u_i$ is $(\mathcal{H},J)$-Floer. When \\ $(u_1(s_0,t_0), \ldots, u_k(s_0,t_0)) \in C_k(\Sigma)$, it is not difficult to simply choose $(\mathcal{H}^{s_0}_{t_0},J^{s_0}_{t_0})$ such that each $u_i$ satisfies the relevant Floer-equation at $(s_0,t_0)$, and this assignment may be made smoothly in $(s,t)$ as long as $u_i(s,t) \neq u_j(s,t)$ for $i \neq j$. The difficulty, therefore, arises when the graphs of the cylinders intersect. To resolve this, in Section \ref{Sec: Perturbing Pos Int Strands} we explain how to perturb the cylinders in a neighbourhood $V$ of an intersection point to another pre-model which satisfies the Floer equations for a judiciously chosen pair $(\mathcal{H}_V,J_V)$ on $V$. In Section \ref{Sec: Proof of Pre-models to Models theorem}, we prove Theorem \ref{Theorem: Pre-models give models of cont cobordisms} by performing this perturbation at every intersection point, and then arguing that we may extend the locally judiciously chosen s-dependent Hamiltonian-almost complex structure pairs to an adapted homotopy of Floer data $(\mathcal{H},J)$ such that the resulting perturbed cylinders are all $(\mathcal{H},J)$-Floer. Finally, in Section \ref{Sec: Models to Regular models}, we address the question of obtaining model cobordisms for \textit{regular} homotopies of Floer data from the existence of the model cobordisms constructed in Theorem \ref{Theorem: Pre-models give models of cont cobordisms}. This gives a method for designing Floer continuation maps which have prescribed chain-level behaviour.

\subsection{Perturbing positively intersecting strands to solve Floer's equation near the intersection}\label{Sec: Perturbing Pos Int Strands}
For $\kappa > 0$, let $D(\kappa) \subset \C$ denote the closed disk of radius $\kappa$, centered at the origin. We write $D=D(1)$ for the closed unit disk.
\begin{lemma}\label{Lemma: Positive intersection perturbation}
Let $(M^4,J)$ be a smooth almost complex $4$-manifold, and let $\mathcal{F}$ be a smooth oriented codimension $2$ foliation of $M$ such that $T \mathcal{F}$ is $J$-invariant. Suppose that $u,v: D \rightarrow M$ are smooth embeddings which are positively transverse to the leaves of $\mathcal{F}$ and which intersect positively at $u(0)=v(0)=p \in M$. Suppose moreover that $v$ is $J$-holomorphic on some neighbourhood of $p$. Then for any $C^0$-neighbourhood $\mathcal{U}$ of $u$ in $C^\infty(D;M)$, there exists $\kappa >0$, $\delta >0$ and $u' \in \mathcal{U}$ such that
\begin{enumerate}
\item $u'\vert_{D(\kappa + \delta)}$ intersects $v$ only at $p=u'(0)$ and this intersection is transverse,
\item $u'$ is $J$-holomorphic on $D(\kappa)$,
\item $u' = u$ on $D \setminus D(\kappa + \delta)$,
\item $u'$ is positively transverse to the leaves of $\mathcal{F}$.
\end{enumerate}

\end{lemma}
\begin{proof}
By the local existence theorem for holomorphic curves (see theorem $3.1.1$ in \cite{Sik94}, for instance), there exists a $\rho >0$ and a $J$-holomorphic map $f: D(\rho) \rightarrow M$ such that $f(0)=p$ and $(\partial_s f)_p = (\partial_s u)_p$. Up to taking $\rho$ to be smaller and restricting the domain of $f$, we may assume that $f$ is an embedding. Note that, $(\partial_s u)_p = (\partial_s f)_p \neq 0$ and hence $(\partial_t f)_p \neq 0$, so $f$ and $v$ intersect transversally and positively, since both are $J$-holomorphic. Note that similar reasoning shows that $f$ is positively transverse to the leaf of $\mathcal{F}$ passing through $p$, and so $f$ is positively transverse to $\mathcal{F}$ for all $z \in D(\rho)$ sufficiently close to $0$. We will construct $u'$ by interpolating between $u$ and $f$. In order to construct this interpolation, it will be useful to introduce a convenient local coordinate system. 
\par
To this end, let us note that there is a neighbourhood $U \subset M$ of $p$, an $\epsilon >0$, and an orientation-preserving chart
\begin{align*}
\phi: U &\rightarrow \C \times \C,
\end{align*}
such that for all $z \in D(\epsilon)$, $(\phi \circ v)(z)=(z,0)$. Moreover, if the foliation $\mathcal{F}$ is as in the statement of the lemma, then the fact that $v$ intersects positively and transversally with the leaves of $\mathcal{F}$ implies that $\phi$ may be chosen such that for any leaf $F \in \mathcal{F}$, $\phi(F \cap U) \subset \lbrace z_0 \rbrace \times \C$ for some $z_0 \in \C$. That is, $\phi$ locally diffeomorphically sends leaves of $\mathcal{F}$ into the fibers of the projection map onto the first coordinate. Since $u$ and $f$ are both positively transverse to $v$, up to shrinking $\epsilon >0$, for $z \in D(\epsilon)$ we may write
\begin{align*}
f_{loc}(z):= (\phi \circ f)(z)=(z,h_0(z)), \\ 
u_{loc}(z):= (\phi \circ u)(z)=(z,h_1(z)) 
\end{align*} 
for $h_0, h_1: D(\epsilon) \rightarrow \C$ smooth functions such that $h_0(0)=h_1(0)=0$, $(\partial_s h_0)_0=(\partial_s h_1)_0 \neq 0$, and $\lbrace (\partial_s h_i)_0, (\partial_t h_i)_0 \rbrace$ is a positively-oriented basis for $\C$ for $i=0,1$. Next, for $\tau \in [0,1]$, let us define $h_\tau: D(\epsilon) \rightarrow \C$ by
\begin{align*}
h_{\tau}(z)&:= (1-\tau)h_0(z) +\tau h_1(z).
\end{align*}
Note that $(\partial_s h_\tau)_0$ is constant in $\tau$, and so it is easy to see that
\begin{align*}
\lbrace (\partial_s h_{\tau})_0, (\partial_t h_\tau)_0 \rbrace &= \lbrace (\partial_s h_{0})_0,(1-\tau)(\partial_t h_0)_0 +\tau (\partial_t h_1)_0 \rbrace
\end{align*}
is positively oriented, as the set of all vectors $w \in \C$ such that $\lbrace (\partial_s h_{0})_0, w \rbrace$ is positively oriented is clearly a convex set and contains both $(\partial_t h_0)_0$ and $(\partial_t h_1)_0$ by hypothesis. Consequently, for each $\tau \in [0,1]$ there exists $\epsilon_\tau \in (0,\epsilon)$, varying continuously with $\tau$, such that $h_\tau(z)=0$ if and only if $z=0$ for all $z \in D(\epsilon_\tau)$. Posing
\begin{align*}
\epsilon'&:= \inf_{\tau \in [0,1]} \epsilon_\tau,
\end{align*} 
we see that $\epsilon'>0$ and the maps $u_\tau(z):=(z,h_\tau(z))$ intersect $D(\epsilon') \times \lbrace 0 \rbrace$ only in $(0,0)$ for all $\tau \in [0,1]$. Fix $\kappa \in (0,\epsilon')$ and some $\delta \in (0, \epsilon' - \kappa)$, and let $\beta: D(\epsilon') \rightarrow [0,1]$ be a smooth radially non-decreasing function which is identically $0$ on $D(\kappa)$ and identically $1$ outside of $D(\kappa + \delta)$. For $z \in D(\epsilon')$, we define 
\begin{align*}
u'_{loc}(z):=u_{\beta(z)}(z)=(1-\beta(z))f_{loc} + \beta(z)u_{loc}.
\end{align*}
Note that $\phi^{-1} \circ u'_{loc}$ then obviously satisfies the first three items listed in the lemma, by construction, and is moreover transparently positively transverse to leaves of $\mathcal{F}$ by construction, since $\phi^{-1}$ sends graphs of maps $\C \rightarrow \C$ to submanifolds transverse to leaves of $\mathcal{F}$ by our choice of $\phi$. We then obtain the claimed $u': D \rightarrow M$ by setting
\begin{equation}
u'(z) = \begin{cases}
(\phi^{-1} \circ u'_{loc})(z) & z \in D(\epsilon') \\
u(z) & z \in D \setminus D(\epsilon').
\end{cases}
\end{equation}
It is then clear from construction that $u'$ satisfies the four properties listed in the lemma. Moreover, because $f(0)=u(0)=p$ and $f$ and $u$ are both continuous, $u'$ may obviously be taken to be as $C^0$-close to $u$ as we wish, simply by taking $\kappa$ and $\delta$ smaller in the preceding argument if necessary.
\end{proof}

\begin{cor}\label{Corollary: Perturbing graphs of Floer cylinders}
Let $u,v: \R \times S^1 \rightarrow \Sigma$ be smooth maps such that their graphs $\tilde{u}, \tilde{v}: \R \times S^1 \rightarrow \R \times S^1 \times \Sigma$ intersect positively and transversally at $(s_0,t_0,p) \in \R \times S^1 \times \Sigma$. Then for any neighbourhood $V \subset \R \times S^1$ of $(s_0,t_0)$, there exists an open set $U \subset \bar{U} \subset V$ containing $(s_0,t_0)$, smooth maps
\begin{align*}
\mathcal{H}_{\bar{U}}: \bar{U} \rightarrow C^\infty(\Sigma), \\
J_{\bar{U}}: \bar{U} \rightarrow \mathcal{J}(\Sigma,\omega),
\end{align*}
and a smooth map $u': \R \times S^1 \rightarrow \Sigma$ agreeing with $u$ outside of $V$, such that $\tilde{u}'$ intersects $\tilde{v}$ positively and transversally at $(s_0,t_0,p)$, and such that both $u'$ and $v$ solve the $(\mathcal{H}_{\bar{U}},J_{\bar{U}})$-Floer equation for $(s,t) \in \bar{U}$.
\end{cor}
\begin{proof}
All that needs to be shown here is that for $V$ a small enough neighbourhood of $(s_0,t_0)$, we may always choose $\mathcal{H}_{V}: V \rightarrow C^\infty(\Sigma)$ and $J_V: V \rightarrow \mathcal{J}(\Sigma,\omega)$ such that $v$ is $(\mathcal{H}_V,J_V)$-Floer on $V$. Once this is established, we simply apply Lemma \ref{Lemma: Positive intersection perturbation} to $\tilde{u}$ and $\tilde{v}$ after using Gromov's trick to construct an almost complex structure $\tilde{J}^\mathcal{H}$ on $V \times \Sigma$ such that maps solve the $(\mathcal{H}_V,J_V)$-Floer equations if and only if their graphs are $\tilde{J}^\mathcal{H}$-holomorphic over $V$. Recall from Section \ref{sec-Gromov trick} that the fibers of the projection map $pr_{cyl}: \R \times S^1 \times \Sigma \rightarrow \R \times S^1$ are pseudo-holomorphic for the almost complex structure given by the Gromov trick, and so Lemma \ref{Lemma: Positive intersection perturbation} gives us an open set $U \subset \bar{U} \subset V$ containing $(s_0,t_0)$ and a map $\tilde{u}': \R \times S^1 \rightarrow \R \times S^1 \times \Sigma$ which is transverse to $\tilde{v}$ on $U$ and which is transverse to the fibers of $pr_{cyl}$ --- and thus $\tilde{u}'$ is the graph of some function $u': \R \times S^1 \rightarrow \Sigma$ which is $\tilde{J}^{\mathcal{H}}$-holomorphic on $\bar{U}$, and $\tilde{u}'$ agrees with $\tilde{u}$ outside of $V$. We then simply take $\mathcal{H}_{\bar{U}}:=\mathcal{H}_V \vert_{\bar{U}}$, $J_{\bar{U}}:=J_V \vert_{\bar{U}}$. 
\par
To see that $\mathcal{H}_{V}: V \rightarrow C^\infty(\Sigma)$ and $J_V: V \rightarrow \mathcal{J}(\Sigma,\omega)$ may be chosen as claimed, let $J_V$ be arbitrary, and define 
\begin{align*}
X: V &\rightarrow v^*T\Sigma \vert_V \\
(s,t) &\mapsto (\partial_t v - J_V \partial_s v)_{(s,t)}.
\end{align*}
It's straightforward to see that we may choose $\mathcal{H}_V$ to satisfy $X^{\mathcal{H}_V}(v(s,t))=X(s,t)$ for all $(s,t) \in V$. Indeed, define
\begin{align*}
\mathcal{H}(X)&:= \lbrace (s,t,H) \in V \times C^\infty(\Sigma): X^H(v(s,t))=X(s,t) \rbrace.
\end{align*}
It is not difficult to verify that $\pi: \mathcal{H}(X) \rightarrow V$ is a locally trivial fibration (see Corollary \ref{Cor: loc triv fib} in Appendix \ref{App: Construct Chain Maps App}), with $\pi$ being the obvious projection map. The fiber over $z \in V$ is diffeomorphic to the subspace of functions on $\Sigma$ having a critical point at $v(z)$. This shows that the fibers are contractible, and so there exists a section $\mathcal{H}_V$. By construction, $v$ is $(\mathcal{H}_V,J_V)$-Floer on $V$.
\end{proof}

\subsection{Proof of Theorem \ref{Theorem: Pre-models give models of cont cobordisms}}\label{Sec: Proof of Pre-models to Models theorem}


Let us now prove Theorem \ref{Theorem: Pre-models give models of cont cobordisms}.
\begin{proof}
We first note that if  $\lbrace u_i: \R \times S^1 \rightarrow \Sigma \rbrace_{i=1}^k$, is a pre-model for a continuation cobordism, then
\begin{align*}
I&:= \lbrace (s,t) \in [-K,K] \times S^1: u_i(s,t)=u_j(s,t), 1 \leq i < j \leq k \rbrace,
\end{align*}
where $K$ is as in the definition of a pre-model for a continuation cobordism, contains only finitely many points and lies in the interior of $[-K,K] \times S^1$. By Corollary \ref{Corollary: Perturbing graphs of Floer cylinders} for each $z \in I$, there exists an arbitrarily small open set $U_z$ containing $z$, local sections 
\begin{align*}
\mathcal{H}^z_{loc}: \bar{U}_z &\rightarrow C^{\infty}(\Sigma) \\
J^z_{loc}: \bar{U}_z &\rightarrow \mathcal{J}(\Sigma,\omega),
\end{align*}
and small perturbations $u_i'$ of each $u_i$, $i=1, \ldots, k$ such that each $u_i'$ satisfies the $(\mathcal{H}^z_{loc},J^z_{loc})$-Floer equation on $\bar{U}_z$. Let us write 
\begin{align*}
V&:= ([-K,K] \times S^1) \setminus \bigcup_{z \in I} U_z.
\end{align*}
Our goal is to find smooth maps 
\begin{align*}
\mathcal{H}_V: V &\rightarrow C^{\infty}(\Sigma) \\
J_V: V &\rightarrow \mathcal{J}(\Sigma,\omega),
\end{align*}
which agree with $\mathcal{H}^z_{loc}$ and $J^z_{loc}$ respectively on $\partial \bar{U}_z$, which in addition agree with $H^\pm$ and with $J^\pm$ respectively on $\lbrace \pm K \rbrace \times S^1$, and such that, moreover, each $u_i'$ is $(\mathcal{H}_V,J_V)$-Floer on $V$. Clearly, if this can be done, these data will patch together to give a homotopy $(\mathcal{H},J)$ of Floer data between $(H^-,J^-)$ and $(H^+,J^+)$ such that each $u_i'$ is $(\mathcal{H},J)$-Floer on all of $\R \times S^1$, and we will be done.
\par
To see that $(\mathcal{H}_V,J_V)$ may be chosen as desired we argue essentially identically to the proof of Corollary \ref{Corollary: Perturbing graphs of Floer cylinders}. First take $J_V$ to be any smooth map $J_V: V \rightarrow \mathcal{J}(\Sigma,\omega)$ which restricts to the desired almost complex structures on $\partial V$. Clearly, this can be done, since the space of $\omega$-compatible almost complex structures is contractible. Next, for each $i=1,\ldots,k$, we define a section of $(u_i' \vert_V)^*T\Sigma$ by
\begin{align*}
X_i: V &\mapsto (u_i'\vert_V)^*T\Sigma \\
(s,t) &\mapsto (\partial_tu_i' - J_V \partial_su_i')_{(s,t)},
\end{align*}  
and we define $\mathcal{H}(\vec{X}) \rightarrow V$ to be the fiber bundle with fiber
\begin{align*}
\mathcal{H}(\vec{X})_{s,t}&:= \lbrace H \in C^\infty(\Sigma): X^H(u_i'(s,t))=X_i(s,t), \; \forall (s,t) \in V, \; \forall i=1, \ldots, k \rbrace.
\end{align*}
It is not difficult to see that $\mathcal{H}(\vec{X})$ is a locally trivial fibration (apply Corollary \ref{Cor: loc triv fib} from Appendix \ref{App: Construct Chain Maps App} in the setting where the underlying symplectic manifold is the configuration space $C_k(\Sigma)$) and that each $\mathcal{H}^z_{loc}\vert_{\partial \bar{U}_z}$ defines a section of $\mathcal{H}(\vec{X})$ over $\partial \bar{U}_z$, in addition to $(\pm K, t) \mapsto H^\pm_t$ defining a section over $\lbrace \pm K \rbrace \times S^1$. Moreover, each fiber of $\mathcal{H}(\vec{X})$ is diffeomorphic to the space of functions on $\Sigma$ having critical points at $k$ distinct points. In particular, $\mathcal{H}(\vec{X})$ has contractible fibers and therefore admits a smooth section extending the given aforementioned data on $\partial V$, completing the proof.
\end{proof}

\subsection{From models of cobordisms to regular models}\label{Sec: Models to Regular models}
Our ultimate goal is to use the topological construction of pre-models for cobordisms to guarantee the more rigid existence of Floer continuation maps with prescribed behaviour. As such, we will need to consider issues of regularity.

\begin{definition}
Let $(H^\pm,J^\pm)$ be Floer regular. A model for a continuation cobordism from $(H^-,J^-)$ to $(H^+,J^+)$ will be called  \textbf{regular} if it is an $(\mathcal{H},\mathbb{J})$-model for some \\ $(\mathcal{H},\mathbb{J}) \in \HJ_{reg}(H^-,J^-;H^+,J^+)$.
\end{definition}

The following lemma follows essentially by combining Corollaries \ref{Cor: Automatic Regularity for Odd orbits} and \ref{Automatic Codim 1 for Even orbits} with results which are likely folkloric: that local transversality of the $s$-dependent Floer operator at some cylinder $u$ only requires perturbations with support contained in an appropriately dense open neighbourhood of the graph of $u$, and the fact that near a minimally degenerate Floer cylinder of index $0$ (ie. when $\dim \ker (D \mathcal{F}_{\mathcal{H},\mathbb{J}})_u)=1$), the universal moduli space locally has the structure of a `fold' singularity over the space of perturbation data. I am unaware of a place in the published literature where such results are stated, so for the convenience of the reader, a detailed proof sketch is provided in Appendix \ref{App: Construct Chain Maps App} (see Section \ref{Subsec: Proof of Local pert to regular lemma}).
\begin{lemma}\label{Lem: Local Pert Non-Reg to Reg}
Let $(H^\pm,J^\pm)$ be Floer-regular, $(\mathcal{H},J) \in \HJ(H^-,J^-;H^+,J^+)$, $\hat{x}^{\pm} \in \Per{H^\pm}$, and $u \in C^{\infty}(\R \times S^1;\Sigma)_{\hat{x}^-,\hat{x}^+}$. Suppose that $\mathcal{F}_{\mathcal{H},J}(u)=0$ and that $ind \; (D\mathcal{F}_{\mathcal{H},J})_{u}=0$. Let $N \subset \R \times S^1 \times \Sigma$ denote a neighbourhood of $\im \tilde{u}$ and let $U \subset N$ be an open dense subset of $N$ such that $U \cap \im \tilde{u}$ is dense in $\im \tilde{u}$. For any neighbourhood $\mathcal{U} \subset C^{\infty}(\R \times S^1;\Sigma)_{\hat{x}-,\hat{x}^+}$, there exists $\mathcal{H}' \in \mathscr{H}(H^-;H^+)$ and $u' \in  \mathcal{U}$ such that $\supp \mathcal{H} -\mathcal{H}' \subset U$, $\mathcal{F}_{\mathcal{H}',J}(u')=0$, and $(D\mathcal{F}_{\mathcal{H}',J})_{u'}$ is surjective.
\end{lemma}
With this lemma in hand, we can readily prove:
\begin{theorem}\label{Thm: Regular Cont Cobordism}
Let $(H^\pm,J^\pm)$ be Floer regular and let $u_i: \R \times S^1 \rightarrow \Sigma$, $i=1, \ldots,k$ define an $(\mathcal{H},\mathbb{J})$-model for a continuation cobordism for some $(\mathcal{H},\mathbb{J}) \in \HJ(H^-,J^-;H^+,J^+)$. Suppose that for each $i=1, \ldots,k$, the index of the linearized Floer operator satisfies
\begin{align*}
ind \;(D\mathcal{F}_{\mathcal{H},\mathbb{J}})_{u_i}&= 0,
\end{align*}  
then for any choice of neighbourhoods $\mathcal{U}_i$ of $u_i \in C^{\infty}(\R \times S^1;\Sigma)_{\hat{x}_i,\hat{y}_i}$, $i=1, \ldots, k$, there exists a regular model for a continuation cobordism $\lbrace u_i'' \rbrace_{i=1}^k$ such that $u_i'' \in \mathcal{U}_i$ for $i=1, \ldots,k$.
\end{theorem}
\begin{proof}
For each $i=1, \ldots,k$, let $\mathcal{N}(\im \tilde{u}_i)$ be an open neighbourhood of $\im \tilde{u}_i$ and let 
\begin{align*}
U_i &:= \mathcal{N}(\im \tilde{u}_i) \setminus \bigcup_{j \neq i} \im \tilde{u}_i.
\end{align*}
Since graphs $\tilde{u}_i, \tilde{u}_j$ have only finitely many intersections for $i \neq j$, $U_i$ clearly satisfies the hypotheses of Lemma \ref{Lem: Local Pert Non-Reg to Reg} with $u=u_i$ and $\mathcal{U}=\mathcal{U}_i$ for each $i=1, \ldots,k$. It follows that there exists $h_i \in C^{\infty}(\R \times S^1 \times \Sigma)$ with compact support contained in $U_i$ and $u_i' \in \mathcal{U}_i$ such that $\mathcal{F}_{\mathcal{H}+h_i,J}(u_i')=0$, and $u_i'$ is a regular point of $\mathcal{F}_{\mathcal{H}+h_i,J}$. Note that, because the support of $h_i$ is compactly contained in $U_i$ which is disjoint from $\im \tilde{u}_j$ for each $j \neq i$, there exists an open neighbourhood $V_j$ of $\im \tilde{u}_j$ such $\supp h_i \cap V_j =\emptyset$. Consequently, since the local behaviour of the Floer operator near a map depends only on the behaviour of the continuation data near the graph of that map, if we set
\begin{align*}
\mathcal{H}'=\mathcal{H} + \sum_{i=1}^k h_i 
\end{align*}
we have that $\mathcal{F}_{\mathcal{H}',J}(u'_i)=0$ for each $i=1, \ldots, k$, and each $u_i'$ is a regular point of $\mathcal{F}_{\mathcal{H}',J}$. Finally, in light of the regularity of each $u_i'$ for $\mathcal{F}_{\mathcal{H}',J}$ and the index condition on the Floer operator, we may perturb $\mathcal{H}'$ to some $\mathcal{H}'' \in \mathscr{H}(H^-,H^+)$ such that $(\mathcal{H}'',J) \in \HJ_{reg}(H^-,J^-;H^+,J^+)$, such that there exist maps $u_i'' \in \mathcal{U}_i$ with $\mathcal{F}_{\mathcal{H}'',J}(u_i'')=0$ for $i=1,\ldots,k$. This completes the proof.
\end{proof}
As an easy consequence of the previous theorem, we obtain the following corollary.
\begin{cor}\label{Cor: Premodel implies regular model with non-empty moduli spaces}
Let $(H^\pm,J^\pm)$ be Floer regular, and let $\hat{x}^\pm_i \in \Per{H^\pm}$, $i=1, \ldots,k$, be such that $\mu(\hat{x}_i^-)=\mu(\hat{x}_i^+)$. If there exists a model for a continuation cobordism $\lbrace u_i \rbrace_{i=1}^k$ with $u_i \in C^\infty_{\hat{x}^-_i,\hat{x}^+_i}(\R \times S^1;\Sigma)$, then there exists $(\mathcal{H},\mathbb{J}) \in \HJ_{reg}(H^-,J^-;H^+,J^+)$ such that 
\begin{align*}
\mathcal{M}(\hat{x}_i^-,\hat{x}^+_i;\mathcal{H},\mathbb{J}) &\neq \emptyset
\end{align*}
for all $i=1,\ldots,k$.
\end{cor}

\section{Application: A new spectral invariant and its computation on surfaces}\label{Ch: App 1}
In this section, we bring to bear the theory developed in the preceding pages in order to give an explicit dynamical characterization --- for any non-degenerate Hamiltonian $H$ on a symplectic surface $(\Sigma,\omega)$ --- of those Floer cycles which both represent a non-trivial homology class in $HF_{1}(H)$ \textit{and} which lie in the image of some PSS map at the chain level.
\par
This dynamical characterization motivates us to study the quantity obtained by modifying the definition of the Oh-Schwarz spectral invariants by looking at the minimal action required to represent a given homology class by a cycle which lies in the image of some PSS map at the chain level. It turns out that this defines a family of action selectors, which we call the \textit{PSS-image spectral invariants}. These novel spectral invariants share many of the same properties as the Oh-Schwarz spectral invariants, including a triangle inequality for the fundamental class, and bound the Oh-Schwarz spectral invariants from above.
\par
The section begins with a general study of the PSS-image spectral invariants in Section \ref{Sec: New spectral invariant}, establishing their basic formal properties, whose proofs are largely analogous to the corresponding proofs for the the Oh-Schwarz spectral invariants. In Section \ref{Sec: Computing spec invar on surfaces}, we compute the PSS-image spectral invariant of a non-degenerate Hamiltonian on a surface by deriving the aforementioned dynamical characterization of Floer cycles which represent the fundamental class and which lie in the image of some PSS map at the chain level. 
\par 
The main line of the proof of Theorem \ref{MainThm: Top Char} is as follows. The fact that the cycle in question lies in the image of some chain-level PSS map implies the existence of various PSS disks, which may be thought of as providing a deformation from a trivial capped braid to the support of the cycle in question. The theory from Sections \ref{Braids} and \ref{Local} may then be used to force the support of this cycle to be positive, and when combined with the cycle's homological non-triviality this theory forces the support to be \textit{maximally} positive (relative index $1$). Conversely, in order to deduce that \textit{every} such capped braid lies in the image of some chain level PSS map, it is sufficient to construct an appropriate continuation map from a small Morse function which sends the fundamental class of the Morse function to a cycle supported on the capped braid in question. The technique for designing continuation maps with prescribed behaviour developed in Section \ref{Ch: Construct chain maps} may be seen to provide precisely such a continuation map.
\par
Section \ref{Sec: Dynamic Consequences} presents two fairly immediate dynamical consequences of this characterization, and the section closes with a brief discussion of the relationship between the Oh-Schwarz spectral invariants and the PSS-image spectral invariants in Section \ref{Sec: On equiv of OS and im}.

\subsection{A new spectral invariant: definitions and properties}\label{Sec: New spectral invariant}
Let $(H,J)$ be Floer-regular. Recall that the set $PSS(H,J)$ of \textit{PSS data} for $(H,J)$ is the collection of tuples
\begin{align*}
\mathcal{D}=(f,g;\mathcal{H},\mathbb{J}) \in C^{\infty}(M) \times Met(M) \times \mathscr{H}(0;H) \times \mathcal{J}^{PSS}(J),
\end{align*}
where $(f,g)$ is a Morse-Smale pair. There is a residual set $PSS_{reg}(H,J)$ of \textit{regular PSS data} such that for any $\mathcal{D} \in \mathcal{D}_{reg}^{PSS}(H,J)$, we may define a chain-level PSS map $\Phi^{PSS}_{\mathcal{D}}: QC_{*+n}(f,g) \rightarrow CF_*(H,J)$ which is $\Lambda_\omega$-linear and induces a natural isomorphism at the level of homology.

\begin{definition}
Let $(H,J)$ be Floer-regular. For $\alpha \in QH_*(M,\omega)$, $\alpha \neq 0$, we define the \textbf{PSS-image spectral invariant}
\begin{align*}
c_{im}(\alpha;H,J) &= c_{im}(\alpha;H) \\
&:= \inf_{\mathcal{D} \in PSS_{reg}(H,J)}  \lbrace \lambda_H(\sigma): \sigma \in \im \Phi_{\mathcal{D}}^{PSS}, \; \text{and} \; [\sigma]= (\Phi_{\mathcal{D}}^{PSS})_* \alpha \rbrace. 
\end{align*}
\end{definition}
In general, it is not immediately obvious that $c_{im}(\alpha;H) \neq -\infty$ for all $\alpha \in QH_*(M,\omega) \setminus \lbrace 0 \rbrace$. To see that this is so, recall the definition of the Oh-Schwarz spectral invariant $c_{OS}(\alpha;H)$ for $0 \neq \alpha \in QH_*(M)$ from Section \ref{Sec: PSS defn Sect}. These spectral invariants are always finite (see theorem 5.3 in \cite{Oh05b}, for instance). Moreover, it is an obvious consequence of the definitions that we have:
\begin{proposition}\label{Prop: OS <= IM}
For any $\alpha \in QH_*(M) \setminus \lbrace 0 \rbrace$ and any $J \in \mathcal{J}_\omega(M)$ such that $(H,J)$ is Floer regular
\begin{align*}
c_{OS}(\alpha;H) &\leq c_{im}(\alpha;H,J).
\end{align*}
\end{proposition}
Consequently, we obtain the finiteness of $c_{im}(\alpha;H,J)$ for all non-zero quantum homology classes $\alpha$. 

\subsubsection{Basic Properties of $c_{im}$}\label{Subsec: Basic props of image spec invar}
We describe now some of the basic properties of the PSS-image invariants. The proofs are essentially routine adaptations of the corresponding proofs for the Oh-Schwarz spectral invariants (see for example \cite{Oh05b}), requiring only that one additionally exhibits some family of chain-level PSS maps which realizes the infimizing process defining $c_{PSS}(\alpha;H,J)$. As such, we omit most of the proofs and defer some illustrative ones to Appendix \ref{App: New Spec Invar}.

\begin{proposition}\label{Prop: Energy Bounds on Diff Of imSpec Invar}
Let $(H,J)$, $(K,J')$ be Floer regular pairs. For any $\alpha \in QH_*(M,\omega)$, $\alpha \neq 0$, we have
\begin{align*}
\int_0^1 \min_{x \in M} (K-H)(t,x) \; dt \leq c_{im}(\alpha;K,J') &- c_{im}(\alpha;H,J)  \leq \int_0^1 \max_{x \in M} (K-H)(t,x) \; dt 
\end{align*} 
\end{proposition}

\begin{cor}\label{Cor: Spec invar indep of acs}
$c_{im}(\alpha;H)=c_{im}(\alpha;H,J)$ is well-defined and independent of the choice of regular $J$. 
\end{cor}
We defer the proof of the following to Appendix \ref{App: New Spec Invar} (Section \ref{App: Proof of Basic Props of PSS-image invars}).
\begin{proposition}\label{Prop: Basic properties of PSS-image invars Main}
For any $H,K \in C^\infty(S^1 \times M)$ and any $\alpha \in QH_*(M,\omega) \setminus \lbrace 0 \rbrace$
\begin{enumerate}
\item if $r: [0,1] \rightarrow \R$ is smooth, then 
\begin{align*}
c_{im}(\alpha;H+r) &= c_{im}(\alpha;H) + \int_0^1 r(t) \;dt.
\end{align*}
\item $c_{im}(\psi_* \alpha; \psi_*H) = c_{im}(\alpha;H)$ for any symplectic diffeomorphism $\psi$.
\item $\vert c_{im}(\alpha;H) - c_{im}(\alpha;K)\vert \leq \| H-K \|_{L^{1,\infty}}$.
\item (Weak triangle inequality) $c_{im}(\alpha; H \#K) \leq c_{im}(\alpha;H) + c_{im}([M];K)$.
\end{enumerate}
\end{proposition}
The above proposition implies that $c_{im}(\alpha;H)$ is $C^0$-continuous in its Hamiltonian argument. By the density of non-degenerate Hamiltonians in the space of all Hamiltonians, we may extend the definition of $c(\alpha;H)$ to any $H \in C^{\infty}(S^1 \times M)$. The above properties then obviously still hold by approximation even if the Hamiltonians in question are not presumed non-degenerate. 

Moreover, if we define, for $\alpha=\sum \alpha_Ae^A \in QH_*(M)$ its \textbf{valuation}
\begin{align*}
\nu(\alpha)&:= \max \lbrace -\omega(A): \alpha_A \neq 0 \rbrace,
\end{align*}
then we see that the PSS-image spectral invariants satisfy the same normalization condition as the Oh-Schwarz spectral invariants.
\begin{proposition}
For $\alpha=\sum \alpha_Ae^A \in QH_*(M) \setminus \lbrace 0 \rbrace$, we have
\begin{align*}
c_{im}(\alpha;0) &= \nu(\alpha).
\end{align*}
\end{proposition}
The proof is identical to the one for the Oh-Schwarz spectral invariants.

\begin{cor}\label{Cor: Linfty bounds on spectral invars}
For any $\alpha \in QH_*(M,\omega)$, $\alpha \neq 0$, we have
\begin{align*}
\int_0^1 \min_{x \in M} H(t,x) \; dt + \nu(\alpha) \leq c_{im}(\alpha;H) & \leq \int_0^1 \max_{x \in M} H(t,x) \; dt + \nu(\alpha). 
\end{align*}
In particular, if $\alpha \in H_*(M)$, then
\begin{align*}
\int_0^1 \min_{x \in M} H_t \; dt \leq c_{im}(\alpha;H) & \leq \int_0^1 \max_{x \in M} H_t \; dt
\end{align*}  
\end{cor}
The PSS-image spectral invariant also shares the following properties of the Oh-Schwarz spectral invariant, whose proofs are once more identical to those for the latter. We refer the interested reader to Theorem II and its proof in \cite{Oh05b}.
\begin{proposition}\label{Prop:BasicPropsTwo}
\begin{enumerate}
If $(M,\omega)$ is rational in the sense that the subgroup $\omega(\pi_2(M)) \subset \R$ is discrete, then
\item \begin{align*}
c_{im}(\alpha;H) &\in Spec(H)
\end{align*}
for all $\alpha \in QH_*(M,\omega) \setminus \lbrace 0 \rbrace$.
\item If $H \in C^\infty(S^1 \times M)$ is assumed to be normalized so that $\int H_t \omega^n=0$ for all $t \in S^1$, then $c_{im}(\alpha;H)$ depends only on the homotopy class of the isotopy $(\phi^H_t)_{t \in [0,1]}$ relative endpoints.
\end{enumerate}
\end{proposition}

\subsubsection{A new spectral norm}\label{Subsec: New Spec Norm}
The fact that the PSS-image spectral invariants satisfy the triangle inequality for the fundamental class allows us to define an associated symplectic conjugation invariant pseudo-norm on $\widetilde{Ham}(M,\omega)$ which descends to a non-degenerate conjugation invariant norm on $Ham(M,\omega)$.
\begin{definition}
Let $(M,\omega)$ be a compact rational semi-positive symplectic manifold. We define the function
\begin{align*}
\tilde{\gamma}_{im}: C^\infty(S^1 \times M) &\rightarrow \R \\
H &\mapsto c_{im}([M];H) + c_{im}([M]; \bar{H}),
\end{align*}
which descends to the \textbf{PSS-image pseudo-norm} on $\widetilde{Ham}(M,\omega)$ by defining
\begin{align*}
\tilde{\gamma}_{im}: \widetilde{Ham}(M,\omega) &\rightarrow \R \\
\tilde{\phi}^H &\mapsto c_{im}([M];H) + c_{im}([M]; \bar{H}),
\end{align*}
where $\tilde{\phi}^H$ denotes the homotopy class relative endpoints in $Ham(M,\omega)$ of the path $(\phi^H_t)_{t \in [0,1]}$.
\end{definition}

\begin{proposition}\label{Prop: Pseudonorm Props}
Let $H \in C^\infty(S^1 \times M)$, then
\begin{enumerate}
\item $\tilde{\gamma}_{im}(H)=\tilde{\gamma}_{im}(\bar{H})$,
\item $\tilde{\gamma}_{im}(H \# F) \leq \tilde{\gamma}_{im}(H) + \tilde{\gamma}_{im}(F)$, $\forall H,F \in C^\infty(S^1 \times M)$. In particular
\begin{align*}
0 &\leq \tilde{\gamma}_{im}(H).
\end{align*}
\item $\tilde{\gamma}_{im}(H) \leq \int_{0}^1 \max_{x \in M} H_t - \min_{x \in M} H_t \; dt$
\end{enumerate}
\end{proposition}

\begin{definition}
To obtain a norm on the group of Hamiltonian diffeomorphisms, we define the \textbf{PSS-image spectral norm}
\begin{align*}
\gamma_{im}: Ham(M,\omega) &\rightarrow [0, \infty) \\
\phi &\mapsto \inf_{\tilde{\phi}} \tilde{\gamma}_{im}(\tilde{\phi}),
\end{align*}
where the infimum is taken over all $\tilde{\phi} \in \widetilde{Ham}(M,\omega)$ which project to $\phi \in Ham(M,\omega)$ under the natural projection $\widetilde{Ham}(M,\omega) \rightarrow Ham(M,\omega)$.
\end{definition}
\begin{proposition}
$\gamma_{im}$ is a norm. Moreover, for all $\phi \in Ham(M,\omega)$,
\begin{align*}
\gamma_{OS}(\phi) \leq \gamma_{im}&(\phi) \leq \| \phi \|_{Hof},
\end{align*}
where $\gamma_{OS}$ denotes the Oh-Schwarz spectral norm and $\| \cdot \|_{Hof}$ denotes the Hofer norm. 
\end{proposition}

\subsection{Computing $c_{im}([\Sigma];H)$ on surfaces}\label{Sec: Computing spec invar on surfaces}
In this section, we restrict our attention to the case of $M=\Sigma$ an arbitrary closed orientable surface, and we use the theory developed in Sections \ref{Braids}-\ref{Ch: Construct chain maps} to give a topological characterization of all homologically non-trivial $\Z / 2 \Z$-cycles in $CF_1(H,J)$ which lie in the image of some PSS-map. Specifically, we prove:

\begin{theorem}\label{Thm: Top Char of PSS cycles}
Let $\sigma \in CF_1(H,J)$. $\sigma$ is a non-trivial cycle with $\sigma \in \im \Phi^{PSS}_{\mathcal{D}}$ for some $\mathcal{D} \in PSS_{reg}(H,J)$ if and only if $\supp \sigma$ is a maximal positive capped braid relative index $1$.
\end{theorem}
\begin{remark}\label{Remark: Integer Coefficients}
The dependence of spectral invariants (and so \textit{a fortiori}, of cycles representing the fundamental class) on the choices\ of coefficients for Floer and quantum homology is a matter of active interest. Since we have restricted our presentation to $\Z/2$-Novikov coefficients in order to avoid dealing with orientations, the above theorem is not particularly informative on this issue. However, for those interested in such questions we will note that the proof of Theorem \ref{Thm: Top Char of PSS cycles} actually implies the following with coefficients in any Novikov ring: If $\sigma \in CF_1(H,J)$ and $\sigma \in \im \Phi^{PSS}_{\mathcal{D}}$ for some $\mathcal{D} \in PSS_{reg}(H,J)$, then $\supp \sigma$ is a maximal positive capped braid relative index $1$. On the other hand, if $\hat{X} \in mp_{(1)}(H)$, then there exist choices $a_{\hat{x}} \in \lbrace -1,1 \rbrace$ for $\hat{x} \in \hat{X}$ such that $\sum_{\hat{x} \in \hat{X}} a_{\hat{x}} \hat{x}$ is a cycle representing the fundamental class in $HF_*(H)$ and lying in the image of some regular PSS-map. Note that because the level of a Floer chain depends only on its support, this suffices to determine the PSS-image spectral invariant $c_{im}([\Sigma],-)$ on surfaces with arbitrary Novikov coefficients (and in fact, it determines $c_{im}([\Sigma]; -)$ for any suitable choice of Novikov coefficients, and shows moreover that its value is \textit{independent} of this choice of coefficients).
\end{remark}
Let us spell out the topological condition that is used in the above characterization. Recall from Definition \ref{Def-PositiveBraid} that a capped braid is positive if there exists a $0$-cobordism $h$ from a trivially capped braid to $\hat{X}$ such that the graphs of the strands of $h$ are transverse with any intersections occurring positively. We remind the reader that for $k \in \Z$, $\Per{H}_{(k)}$ denotes the collection of capped $1$-periodic orbits of $H$ having Conley-Zehnder index $k$.

\begin{definition}
For $H \in C^\infty(S^1 \times \Sigma)$ non-degenerate, a collection of capped $1$-periodic orbits $\hat{X} \subseteq \Per{H}$ will be said to be \textbf{maximally positive} (resp. \textbf{negative}) \textbf{relative index $1$} (resp. \textbf{relative index $-1$}) if
\begin{enumerate}
\item $\hat{X} \subseteq \Per{H}_{(1)}$ (resp. $\hat{X} \subseteq \Per{H}_{(-1)}$) ,
\item $\hat{X}$ is a positive (resp. negative) capped braid, and
\item $\hat{X}$ is maximal among all subsets of $\Per{H}$ satisfying the two previous items.
\end{enumerate}
We denote by $mp_{(1)}(H)$ (resp. $mn_{(-1)}(H)$) the set of all such capped braids.
\end{definition}

This leads immediately to the following characterization of $c_{im}([\Sigma];-)$ in the non-degenerate case.
\begin{theorem}\label{Thm: Computation of im-spec invar}
Let $H$ be non-degenerate.
\begin{align*}
c_{im}([\Sigma];H)&= \min_{\hat{X} \in mp_{(1)}(H)} \max_{\hat{x} \in \hat{X}} \mathcal{A}_H(\hat{x})
\end{align*}
\end{theorem}
We begin with the following lemma.
\begin{lemma}\label{lem: Perturb PSS Disks}
Let $\mathcal{D} \in PSS(H,J)$, and let $u_i: D^2 \rightarrow \Sigma$, $i=1, \ldots,k$, solve the PSS-equation induced by $\mathcal{D}$. Suppose that for some $(s_0,t_0) \in D^2$ and some $p \in \Sigma$,
\begin{align*}
u_i(s_0,t_0) &=p, \; \text{for all } i=1, \ldots,k.
\end{align*}
Then there exists an open neighbourhood $U \subseteq D^2$ of $(s_0,t_0)$ and smooth maps $v_i: D^2 \rightarrow \Sigma$, $i=1,\ldots,k$ such that 
\begin{enumerate}
\item $v_i(s,t)=u_i(s,t)$ for all $(s,t) \not \in U$, 
\item $v_i(s_0,t_0) \neq p$ for all $i=1,\ldots,k$, and
\item any pairwise intersections of the graphs $\tilde{v}_i$ and $\tilde{v}_j$, $1 \leq i < j \leq k$, are both transverse and positive.
\end{enumerate}
\end{lemma}
\begin{proof}
This is a straightforward consequence of the intersection theory for pseudoholomorphic curves in $4$-dimensional almost complex manifolds as worked out in Appendix E of \cite{McSa12}. In slightly more detail, the proof of Proposition E.2.2 of \cite{McSa12} shows that if $f: (\Sigma,j) \rightarrow (M,J)$ is any simple pseudoholomorphic map with $(\Sigma,j)$ any Riemannian surface and $(M,J)$ any $4$-dimensional almost complex manifold, then for any self -intersection point $(z_0,z_1) \in \Sigma \times \Sigma \setminus \Delta$ with $f(z_0)=f(z_1)$, there exist (disjoint) neighbourhoods $U_i \subseteq  \Sigma$ of $z_i$, $i=0,1$, and a perturbation $f'$ of $f$ differing from $f$ only on $U_0 \bigcup U_1$ such that $f'$ has only positive and transverse self-intersections for $(z_0',z_1') \in U_0 \times U_1$. Noting that for each $i=1, \ldots, k$, the graph $\tilde{u}_i$ is $\tilde{J}^{\mathcal{D}}$-holomorphic, where $\tilde{J}^{\mathcal{D}}$ is the almost complex structure on $D^2 \times \Sigma$ associated to the PSS data $\mathcal{D}$ by the Gromov trick, we may apply this proposition with $f= \sqcup_{i=1}^k \tilde{u}_i$. The perturbation may be chosen small enough such that each component of $f'= \sqcup \tilde{v}_i$ is still graphical, and without loss of generality (up to another perturbation of $f'$), we may suppose that self-intersections of $f'$ do not occur at $(s_0,t_0,p) \in D^2 \times \Sigma$. Setting $v_i:=\pi_{\Sigma} \circ \tilde{v}_i$ provides the desired maps.
\end{proof}

\begin{cor}\label{cor: PSS image is Pos}
Let $\mathcal{D}=(f,g;\mathcal{H},\mathbb{J}) \in PSS_{reg}(H,J)$ and suppose that $\sigma \in \im \Phi^{PSS}_{\mathcal{D}} \subseteq CF_1(H,J)$. Then $\supp \sigma$ is a positive capped braid.
\end{cor}
\begin{proof}
Write $\hat{X} = \lbrace \hat{x}_1, \ldots, \hat{x}_k \rbrace = \supp \sigma$. Because $\sigma \in \im \Phi^{PSS}_{\mathcal{D}}$ for each $\hat{x} \in \hat{X}$, there exists $p \in Crit(f)$ such that $\mathcal{M}(p,\hat{x};\mathcal{D}) \neq \emptyset$. For $i=1, \ldots, k$, we choose $u_i \in \mathcal{M}(p_i,\hat{x}_i;\mathcal{D})$. Owing to the Gromov trick, and positivity of intersections in dimension $4$, the algebraic intersection number of the graphs $\tilde{u}_i$ and $\tilde{u}_j$ is positive for any $1 \leq i < j \leq k$. Moreover, using Lemma \ref{lem: Perturb PSS Disks}, we may suppose without loss of generality that none of the graphs intersect over $0 \in D^2$. We observe next that no two PSS-disks are asymptotic as $s \rightarrow \infty$ to the same orbit (that is, $x_i \neq x_j$ for $i \neq j$). Indeed, if we set $v^i(s,t)=u_i(s e^{2 \pi \i t})$, for $i=1, \ldots, k$, then for $i \neq j$ we have by Lemma \ref{EndsLinkLemma} that
\begin{align*}
0= \ell(p_i,p_j) = &\ell_{-\infty}(u_i,u_j) \leq \ell_{\infty}(u_i,u_j),
\end{align*}
and if $x_i=x_j$, then $\ell_{\infty}(u_i,u_j)\leq a(\hat{x}_i)$ by Corollary \ref{EmergeConvergeLinking}, where $\hat{x}_i$ is the unique capping of $x_i$ such that $\mu_{CZ}(\hat{x}_i)=1$, and so Lemma \ref{CZindexAsympWindCompLemma} implies that $a(\hat{x}_i)=-1$, a contradiction. 
\par
As a consequence, the map $s \mapsto (v^1(s,\cdot),\ldots,v^k(s,\cdot)) \in \mathcal{L}_0(\Sigma)^k$ defines a positive $0$-cobordism from the trivial capped braid $(\widehat{u_1(0)}, \ldots, \widehat{u_k(0)})$ to the capped braid $\hat{X}$, showing that $\hat{X}$ is a positive capped braid.
\end{proof}

\begin{proposition}\label{Prop: Top Char Nec Condition}
Let $\mathcal{D} \in PSS_{reg}(H,J)$ and suppose that $\sigma \in CF_1(H,J)$ is such that $\supp \sigma$ is a positive capped braid. If $\sigma$ is a cycle satisfying $(\Phi_{\mathcal{D}}^{PSS})_*([\Sigma])=[\sigma]$, then $\supp \sigma$ is maximally positive relative index $1$.
\end{proposition}
\begin{proof}
We argue by contradiction. Suppose that $(\Phi_{\mathcal{D}}^{PSS})_*([\Sigma])=[\sigma]$, but that $\hat{X} = \supp \sigma$ is not maximally positive relative index $1$. Then there must exist $\hat{y} \in \widetilde{Per}_0(H)_{(1)}$ with $\ell(\hat{x},\hat{y}) \geq 0$ for all $\hat{x} \in \hat{X}$, but $\hat{y} \not \in \hat{X}$. Choose $f \in C^\infty(\Sigma)$ to be a $C^2$-small Morse function with a unique maximum $M \in \Sigma$, and let $J^+ \in \mathcal{J}(\Sigma,\omega)$ be such that $(f,J^+)$ is Floer-regular. Here, the $C^2$-smallness of $f$ is taken to be such that such that $CF_*(f,J^+)$ may be canonically identified with $QC_{*+1}(f,g_{J^+})$, and the fact that $f$ has a unique maximum at $M$ then implies that $CF_1(f,J^+)$ is generated by the constant orbit based at $M$, equipped with the constant capping.
\par
Let $u: \R \times S^1 \rightarrow \Sigma$ be a smooth cylinder such that $u(s,t)= y(t)$ for all $s \in \R$ sufficiently small, and $u(s,t)=M$ for all $s \in \R$ sufficiently large. $u$ is obviously a pre-model for a continuation cobordism, and so by Theorem \ref{Theorem: Pre-models give models of cont cobordisms}, there exists a homotopy of Floer data $(\mathcal{H},\mathbb{J}) \in \HJ(H,J;f,J^+)$ such that $u$ solves the $(\mathcal{H},\mathbb{J})$ Floer equation and so in particular defines a $(\mathcal{H},\mathbb{J})$-model as in Definition \ref{Def-Model} (strictly speaking, Theorem \ref{Theorem: Pre-models give models of cont cobordisms} states that this is true up to a small perturbation of $u$, but if one works through the proof in the case of a single cylinder, one sees that this perturbation is not necessary. Using $u$ is mainly a matter of notational convenience). By Corollary \ref{Cor: Automatic Regularity for Odd orbits}, $u$ is regular, and so up to perturbing $(\mathcal{H},\mathbb{J})$ to a regular pair $(\mathcal{H}',\mathbb{J}')$, and perturbing $u$ to some $u'$ which solves the $(\mathcal{H}',\mathbb{J}')$-Floer equation if necessary, we conclude that
\begin{align*}
\mathcal{M}(\hat{y},\hat{M};\mathcal{H}',\mathbb{J}') &\neq \emptyset.
\end{align*}
Consequently, for all $\hat{x} \in \hat{X}$, we must have that
\begin{align*}
\mathcal{M}(\hat{x},\hat{M};\mathcal{H}',\mathbb{J}') &=\emptyset,
\end{align*}
for if $v \in \mathcal{M}(\hat{x},\hat{M};\mathcal{H}',\mathbb{J}')$, then we must have
\begin{align*}
0 \leq \ell(\hat{x},\hat{y})= &\ell_{-\infty}(u',v) \leq \ell_{\infty}(u',v) \leq a(\hat{M}) = -1,
\end{align*}
which is a contradiction. Consequently, $h_{\mathcal{H}'}(\sigma)=0$, but this is absurd, since $\sigma$ represents a non-trivial homology class in $CF_*(H,J)$, and so $h_{\mathcal{H}'}(\sigma) \neq 0$, since $h_{\mathcal{H}'}$ is an isomorphism on homology. It follows that no $\hat{y}$ as above may exist and hence $\hat{X} \in mp_{(1)}(H)$.
\end{proof}

\begin{proof}[Proof of Theorem \ref{Thm: Top Char of PSS cycles}]
Corollary \ref{cor: PSS image is Pos} together with Proposition \ref{Prop: Top Char Nec Condition} immediately implies that the support of every cycle representing the fundamental class and lying in the image of some PSS map must be a maximal positive braid relative index $1$. Let us show the converse. Let $\hat{X} = \lbrace \hat{x}_1, \ldots, \hat{x}_k \rbrace \subseteq \Per{H}_{(1)}$ be maximally positive relative index $1$, and fix some $J \in C^{\infty}(S^1;\mathcal{J}_\omega(\Sigma))$ such that $(H,J)$ is Floer regular. We will show that there exists $\mathcal{D}=(f,g;\mathcal{H},\mathbb{J}) \in PSS_{reg}(H,J)$ such that $f$ is a Morse function and
\begin{align*}
\Phi^{PSS}_\mathcal{D}(\sum_{\substack{M \in Crit(f) \\ \mu^{Morse}(M)=2}} M ) &= \sum_{\hat{x} \in \hat{X}} \hat{x}.
\end{align*}
To do this, note first that it will suffice to show that there exists some Floer regular pair $(f',J^-)$ with $f' \in C^\infty(\Sigma)$ a $C^2$-small Morse function and a regular homotopy of Floer data $(\mathcal{H}',\mathbb{J}') \in \HJ(f',J^-;H,J)$ such that 
\begin{align*}
h_{\mathcal{H}'}( \sum_{\hat{p} \in \Per{f'}_{(1)}} \hat{p}) &=\sum_{\hat{x} \in \hat{X}} \hat{x}.
\end{align*} 
Indeed, if $(\mathcal{H}',\mathbb{J}')$ as above exists, then because $ \sum_{\hat{p} \in \Per{f'}_{(1)}} \hat{p}$ represents the unique non-trivial cycle in $CF_{1}(f',J^-)$, \textit{any} choice of $\mathcal{D}' \in PSS^{reg}(f',J^-)$ will satisfy 
\begin{align*}
\Phi^{PSS}_\mathcal{D'}(\sum_{\substack{M \in Crit(f) \\ \mu^{Morse}(M)=2}} M ) &= \sum_{\hat{p} \in \Per{f'}_{(1)}} \hat{p}
\end{align*}
at the chain level, and using the gluing results of \cite{Sc95}, we may define $\mathcal{D} \in PSS_{reg}(H,J)$ as an appropriate gluing $\mathcal{D}:=\mathcal{D}' \# (\mathcal{H}',\mathbb{J}')$ such that at the chain level, we have
\begin{align*}
\Phi^{PSS}_{\mathcal{D}}&= h_{\mathcal{H}'} \circ \Phi^{PSS}_{\mathcal{D}'}.
\end{align*}
With the above understood, note next that since $\hat{X}$ is a positive braid, there exists a positive $0$-cobordism $s  \mapsto (v^1_s, \ldots,v^k_s) \in \mathcal{L}_0(\Sigma)^k$ from some trivial braid $(\hat{p}_1,\ldots, \hat{p}_k)$ to $\hat{X}$, where $\hat{p}_i$ denotes the trivially capped constant loop based at $p_i$. We choose $f' \in C^\infty(\Sigma)$ to be any $C^2$-small Morse function having local maxima precisely at the points $p_i$, $i=1, \ldots,k$. Let $J^- \in \mathcal{J}_\omega(\Sigma)$ be any compatible almost complex structure such that $(f',J^-)$ is Floer regular. Next, we define a pre-model for a continuation cobordism from $(f',J^-)$ to $(H,J)$. Fix some $K>0$ and for each $i=1, \ldots, k$, define
\begin{equation*}
u_i(s,t):= \begin{cases}
p_i & s \in (-\infty,-K] \\
v^i_{s+K}(t) & s \in (-K,-K+1) \\
x_i(t) & s \in [K+1,\infty).
\end{cases}
\end{equation*}
It is not hard to verify that $\lbrace u_i \rbrace_{i =1}^k$ defines a pre-model for a continuation cobordism from $(f',J^-)$ to $(H,J)$.  Theorem \ref{Theorem: Pre-models give models of cont cobordisms} implies that we may perturb this to a continuation cobordism $\lbrace u_i' \rbrace_{i=1}^k$ for some $(\mathcal{H}',\mathbb{J}') \in \HJ(f',J^-;H,J)$. Since each $u_i'$ is asymptotic at both ends to a non-degenerate orbit of Conley-Zehnder index $1$, Corollary \ref{Cor: Automatic Regularity for Odd orbits} implies that without loss of generality, we may take $(\mathcal{H}',\mathbb{J}')$ to be regular. We claim that for each $1 \leq i,j \leq k$, 
\begin{equation*}
\vert\mathcal{M}(\hat{p}_i,\hat{x}_j; \mathcal{H}',\mathbb{J}')\vert= \begin{cases}
0 & \text{if } i \neq j \\
1 & \text{if } i=j
\end{cases}.
\end{equation*}
The latter case is guaranteed by Corollary \ref{Cor: Cont moduli space is 0 or 1} together with the existence of $u_i' \in \mathcal{M}(\hat{p}_i,\hat{x}_i;\mathcal{H}',\mathbb{J}')$. To establish the former case, note that since $\lbrace \hat{p}_1, \ldots, \hat{p}_k \rbrace$ is a trivial capped braid, we certainly have $\ell(\hat{p}_i,\hat{p}_j)=0$ for all $i \neq j$. Suppose for the sake of contradiction that there exists $v \in \mathcal{M}(\hat{p}_i,\hat{x}_j;\mathcal{H}',\mathbb{J}')$. Then we must have
\begin{align*}
0=\ell(\hat{p}_i,\hat{p}_j) = \ell_{-\infty}(v,u_j') &\leq \ell_{\infty}(v,u_j') \leq a(\hat{x}_j) =-1,
\end{align*}
which is a contradiction. It follows that 
\begin{align*}
\hat{X} &\subseteq \supp h_{\mathcal{H}'}(\sum_{i=1}^k \hat{p}_i).
\end{align*}
To conclude that the above containment is in fact an equality, note that in combination with the gluing trick outlined at the beginning of this proof, Proposition \ref{Prop: Top Char Nec Condition} implies that $\supp h_{\mathcal{H}'}(\sum_{i=1}^k \hat{p}_i) \in mp_{(1)}(H)$, but $\hat{X} \in mp_{(1)}(H)$ by hypothesis, and so maximality of $\hat{X}$ implies 
\begin{align*}
\hat{X} &= \supp h_{\mathcal{H}'}(\sum_{i=1}^k \hat{p}_i)
\end{align*}
as desired.
\end{proof}

\subsection{Dynamical Consequences}\label{Sec: Dynamic Consequences}
\begin{theorem}
For any closed symplectic surface $(\Sigma,\omega)$, and any non-degenerate Hamiltonian $H \in C^\infty(S^1 \times \Sigma)$,
\begin{align*}
\tilde{\gamma}_{im}(\tilde{\phi}^H)&= \min_{\hat{X} \in mp_{(1)}(H)} \max_{\hat{x} \in \hat{X}} \mathcal{A}_H(\hat{x}) - \max_{\hat{X} \in mn_{(-1)}(H)} \min_{\hat{x} \in \hat{X}} \mathcal{A}_H(\hat{x}).
\end{align*}
\end{theorem}
\begin{proof}
Since, for any strongly-semipositive $(M,\omega)$, the PSS-image spectral invariants are well-defined on $\widetilde{Ham}(M,\omega)$
\begin{align*}
\gamma_{im}(\phi)&= \inf_{\tilde{g} \in \pi_1(Ham(M,\omega), \id{M})} \tilde{\gamma}_{im}(\tilde{g} \cdot \tilde{\phi})
\end{align*}
for $\tilde{\phi}$ any lift of $\phi$ to $\widetilde{Ham}(M,\omega)$. The theorem follows directly from the fact that
\begin{align*}
\tilde{\gamma}_{im}(\tilde{\phi}^H)&= c_{im}([\Sigma];H) + c_{im}([\Sigma];\bar{H}) \\
&=\min_{\hat{X} \in mp_{(1)}(H)} \max_{\hat{x} \in \hat{X}} \mathcal{A}_H(\hat{x}) - \max_{\hat{X} \in mn_{(-1)}(H)} \min_{\hat{x} \in \hat{X}} \mathcal{A}_H(\hat{x}),
\end{align*}
where the identification of the second term in the last line with $c_{im}([\Sigma];\bar{H})$ follows from the fact that $\bar{H}$ is homotopic relative endpoints to the Hamiltonian $\tilde{H}(t,x)=-H(1-t,x)$ which generates the isotopy $\phi^H_{1-t} \circ (\phi^H_1)^{-1}$, whose time-$1$ flowlines are precisely the time reversal of the time-$1$ flowlines of the isotopy induced by $H$. Consequently, there is a natural bijection
\begin{align*}
\Per{H} &\mapsto \Per{\tilde{H}} \\
[x(t),w(se^{2 \pi \i t})] &\mapsto [x(1-t),w(se^{-2 \pi \i t})]
\end{align*}
which sends capped orbits with Conley-Zehnder index $k$ to capped orbits with Conley-Zehnder index $-k$. It is easy to see that positive braids are sent to negative braids under the above bijection, which establishes the desired formula.
\end{proof}
We recall for the reader that for a closed surface $\Sigma$, $\pi_1(Ham(\Sigma))$ is trivial whenever $\pi_2(\Sigma)=0$, while $\pi_1(Ham(S^2)) \simeq \Z/2\Z$ (see Section 7.2 of \cite{Pol12} for further details and references for these computations).

\begin{cor}[Theorem \ref{Mainthm-Norm}]
If $\pi_2(\Sigma) = 0$, then for any non-degenerate $\phi \in Ham(\Sigma,\omega)$,
\begin{align*}
\gamma_{im}(\phi)&= \min_{\hat{X} \in mp_{(1)}(H)} \max_{\hat{x} \in \hat{X}} \mathcal{A}_H(\hat{x}) - \max_{\hat{X} \in mn_{(-1)}(H)} \min_{\hat{x} \in \hat{X}} \mathcal{A}_H(\hat{x})
\end{align*}
for any Hamiltonian $H$ with $\phi^H_1=\phi$.
\end{cor}
\begin{cor}\label{Cor: Sphere norm}
If $\Sigma = S^2$, then for any non-degenerate $\phi \in Ham(\Sigma,\omega)$,
\begin{align*}
\gamma_{im}(\phi)&= \min \lbrace \min_{\hat{X} \in mp_{(1)}(K)} \max_{\hat{x} \in \hat{X}} \mathcal{A}_K(\hat{x}) - \max_{\hat{X} \in mn_{(-1)}(K)} \min_{\hat{x} \in \hat{X}} \mathcal{A}_K(\hat{x}): \; K \in \lbrace H, G \# H \rbrace \rbrace,
\end{align*}
where $H$ is any Hamiltonian $H$ with $\phi^H_1=\phi$ and $G$ is any Hamiltonian such that $(\phi^G_t)_{t \in [0,1]}$ is a non-contractible loop in $Ham(S^2,\omega)$.
\end{cor}

Another interesting consequence of our dynamical characterization of the PSS-image spectral invariants is that it permits us to use the work of Entov in \cite{En04} to obtain computable bounds on the commutator length of the isotopy $(\phi^H_t)_{t \in [0,1]}$ in $\widetilde{Ham}(S^2,\omega)$. Recall that for any group $G$, the \textbf{commutator length} of $g \in G$ is defined to be 
\begin{align*}
cl(g)&:= \inf \lbrace k: g= \Pi_{i=1}^k [f_i,h_i], \; f_i, h_i \in G \rbrace.
\end{align*}
It is a classical result due to Banyaga \cite{Ba78} that for compact symplectic manifolds both $Ham(M,\omega)$ and its universal cover $\widetilde{Ham}(M,\omega)$ are perfect groups, and so in particular every element $\tilde{\phi}^H \in \widetilde{Ham}(S^2,\omega)$ has a finite commutator length. We will show:
\begin{theorem}[Theorem \ref{Mainthm-Entov}]
Assume that $H \in C^{\infty}(S^1 \times S^2)$ is non-degenerate and normalized, then
\begin{align*}
\min \big\lbrace  \min_{\hat{X} \in mp_{(1)}(H)} \max_{\hat{x} \in \hat{X}} \mathcal{A}_{H}(\hat{x}), -\max_{\hat{X} \in mn_{(-1)}(H)} \min_{\hat{x} \in \hat{X}} \mathcal{A}_H(\hat{x}) \big\rbrace &< -k Area(S^2,\omega)
\end{align*}
implies $cl(\tilde{\phi}^H) >2k+1$.
\end{theorem}
This result essentially says that if a Hamiltonian isotopy on a sphere is to be `simple' in the sense that it has small commutator length in $\widetilde{Ham}$, then the actions of orbits forming maximally positive (relative index $1$) capped braids cannot be uniformly small, and nor can the actions of orbits forming maximally negative (relative index $-1$) capped braids be uniformly large.

\begin{proof}
This is essentially a direct application of Theorem 2.5.1 in \cite{En04}, combined with our dynamical characterization of $c_{im}([S^2];H)$. Note that in \cite{En04}, Entov uses a cohomological convention for spectral invariants along with opposite sign conventions in defining the Hamiltonian vector field, and so his spectral invariants are functions $(a,H) \mapsto c_{Ent}(a;H)$, for $a \in QH^*(M,\omega)$. For the reader's convenience, we note that accounting for the differences in conventions leads to the relation
\begin{align*}
c_{Ent}(a;H) &=c_{OS}(PD(a);\overline{H}),
\end{align*}
where $PD(a) \in QH_{2n-*}(M,\omega)$ is the quantum homology class which is Poincar\'{e} dual to $a$, and $\overline{H}(t,x)=-H(t,\phi^H_t(x))$. Thus in the case that $g=2k+1$ and $(M,\omega)=(S^2,\omega)$, Theorem 2.5.1 in \cite{En04} says that $c_{OS}([pt] \otimes e^{k [S^2]};\overline{H}) > 0$ implies $cl(\tilde{\phi}^{\overline{H}}) > 2k+1$. Using that \begin{align*}
c_{OS}([pt] \otimes e^{k [S^2]};\overline{H}) &= c_{OS}([pt];\overline{H}) - k Area(S^2,\omega) \\
&= -c_{OS}([S^2];H) - k Area(S^2,\omega),
\end{align*}
and that $c_{OS}([S^2];H) \leq c_{im}([S^2];H)$, we conclude that if
\begin{align*}
c_{im}([S^2];H)  &< - k Area(S^2,\omega),
\end{align*}
then $cl(\tilde{\phi}^{\overline{H}}) > 2k+1$. Symmetric reasoning gives that $c_{im}([S^2];\bar{H})  < - k Area(S^2,\omega)$ implies $cl(\tilde{\phi}^{H}) > 2k+1$, and using $cl(\tilde{\phi}^H)=cl(\tilde{\phi}^{\overline{H}})$ allows us to conclude the result.
\end{proof}
Since $\pi_1(Ham(S^2)) \simeq \Z / 2 \Z$, it is relatively easy to pass from commutator lengths on the universal cover to commutator lengths in $Ham(S^2)$ itself. For notational convenience, we write 
\begin{align*}
m(H)&:= \min \big\lbrace  \min_{\hat{X} \in mp_{(1)}(H)} \max_{\hat{x} \in \hat{X}} \mathcal{A}_{H}(\hat{x}), -\max_{\hat{X} \in mn_{(-1)}(H)} \min_{\hat{x} \in \hat{X}} \mathcal{A}_H(\hat{x}) \big\rbrace.
\end{align*}
\begin{cor}
Assume that $\phi \in Ham(S^2,\omega)$ is non-degenerate, and let $H$ be any normalized Hamiltonian such that $\phi^H_1=\phi$. If $G$ denotes any normalized Hamiltonian which generates a non-contractible loop in $Ham(S^2,\omega)$ based at the identity, then
\begin{align*}
\max \lbrace m(H), m(G \# H) \rbrace &< -k Area(S^2,\omega)
\end{align*}
implies $cl(\phi) >2k+1$.
\end{cor}
\begin{proof}
This follows directly from the preceding proposition, together with the observation that for any compact symplectic manifold $(M,\omega)$ and any $\phi \in Ham(M,\omega)$, we have $cl(\phi) = \min \lbrace cl(\tilde{\phi}^H): \phi^H_1=\phi \rbrace$. To see this, note that commutator length decreases along group homomorphisms, since commutators get sent to commutators under homomorphisms. This, together with the fact that the projection $p: \widetilde{Ham}(M,\omega) \rightarrow Ham(M,\omega)$ is surjective, implies that $cl(\phi) \leq \min \lbrace cl(\tilde{\phi}^H): \phi^H_1=\phi \rbrace$. To see that the reverse inequality holds, let $cl(\phi)=k$ and write $\phi= \Pi_1^k [a_i,b_i]$ for $a_i,b_i \in Ham(M,\omega)$. Letting $\tilde{a}_i, \tilde{b}_i \in \widetilde{Ham}(M,\omega)$ be arbitrary lifts for $i=1,\ldots, k$, then $\Pi_1^k [\tilde{a}_i,\tilde{b}_i] \in p^{-1}(\phi)$, which implies the desired result.
\end{proof}

\subsection{On the equivalence $c_{OS}=c_{im}$}\label{Sec: On equiv of OS and im}
Proposition \ref{Prop: OS <= IM} assures us that
\begin{align*}
c_{OS}(\alpha;H) &\leq c_{im}(\alpha;H),
\end{align*}
and the PSS-image spectral invariants satisfy many of the same formal properties as the Oh-Schwarz spectral invariants. One is naturally led to ask when these two invariants coincide, and what to make of their difference. By the definition of the PSS-image spectral invariants, $c_{im}(\alpha;H) - c_{OS}(\alpha;H)$ represents the obstruction to our ability to find a tight cycle $\sigma \in CF_*(H,J)$ for $c_{OS}(\alpha;H)$ which additionally lies in the image of some PSS map at the chain level. Note that, on a surface, our knowledge that the image of a chain-level PSS map must be form a positive capped braid, while in general, the support of Floer cycles need not be positive implies that, at least in low dimensions, there are genuine obstructions to the equivalence of these two quantities. Thus, in some sense, while the Oh-Schwarz spectral invariants arise from purely (filtered) homological considerations, the PSS-image spectral invariants in principle encode some additional geometric information about the degree to which this filtration information may be probed by PSS-type maps. In fact, it turns out that one can also interpret the difference of the functions 
\begin{align*}
c_{OS}(\alpha;-), c_{im}(\alpha;-): C^\infty(S^1 \times M) &\rightarrow \R
\end{align*} 
as measuring the failure of the PSS-image spectral invariants to satisfy the Poincar\'e duality relation discovered by Entov-Poterovich for the Oh-Schwarz spectral invariants in \cite{EP03}.
\begin{proposition}
Let $\alpha, \beta \in QH_*(M,\omega) \setminus \lbrace 0 \rbrace$ be such that
\begin{align*}
c_{OS}(\alpha; H) &= - \inf_{\Pi(\alpha,\beta) \neq 0} c_{OS}(\beta; \bar{H}),
\end{align*}
where $\Pi$ is the bilinear pairing described in Section 2.3 of \cite{EP03} (see also Section 20.4 of \cite{Oh15}), then for all $H \in C^\infty(S^1 \times M)$,
\begin{align*}
c_{im}(\alpha;H)- c_{OS}(\alpha;H) = c_{im}(\beta;\bar{H}) &- c_{OS}(\beta;\bar{H}) =0
\end{align*}
if and only if
\begin{align*}
c_{im}(\alpha;H) &= - c_{im}(\beta;\bar{H}).
\end{align*}
\end{proposition}
\begin{proof}
The proof is a straight-forward consequence of Entov-Polterovich's  Poincar\'e duality relation for the Oh-Schwarz spectral invariants. Indeed, since the PSS-image spectral invariants always bound the Oh-Schwarz spectral invariants from above, we have
\begin{align*}
- c_{im}(\beta; \bar{H}) \leq -c_{OS}(\beta; \bar{H}) &= c_{OS}(\alpha;H) \leq c_{im}(\alpha;H),
\end{align*}
thus it is clear that if $c_{im}(\alpha;H)=- c_{im}(\beta;\bar{H})$, we must have
\begin{align*}
c_{im}(\alpha;H)- c_{OS}(\alpha;H) = c_{im}(\beta;\bar{H}) &- c_{OS}(\beta;\bar{H}) =0.
\end{align*}
The converse statement is immediate.
\end{proof}
One may also show that $c_{OS}([\Sigma];H)$ and $c_{im}([\Sigma];H)$ coincide for all autonomous $H \in C^\infty(\Sigma)$ when $\Sigma \neq S^2$. Indeed, in \cite{HRS16}, Humilière-Le Roux-Seyfaddini showed that when $(\Sigma,\omega)$ is a closed aspherical symplectic surface, then any function $c: C^\infty(S^1 \times \Sigma) \rightarrow \R$ satisfying
\begin{enumerate}
\item \textbf{(Spectrality)} $c(H) \in Spec(H)$ for all $H \in C^{\infty}(S^1 \times \Sigma)$
\item \textbf{(Nontriviality)} There exists a topological disk $D$ and a Hamiltonian $H$ supported in $D$ such that $c(H) \neq 0$
\item \textbf{(Continuity)} $c$ is continuous with respect to the $C^{\infty}$ topology on $C^\infty(S^1 \times \Sigma)$
\item \textbf{(Max formula)} If $H_i \in C^\infty(S^1 \times \Sigma)$, $i=1, \ldots,k$ are supported in pairwise disjoint disks then
\begin{align*}
c(H_1+\ldots + H_k)&=\max \lbrace c(H_1),\ldots, c(H_k) \rbrace
\end{align*}
\end{enumerate}
agrees with $c_{OS}([\Sigma];-)$ on the space of autonomous Hamiltonians. Propositions \ref{Prop:BasicPropsTwo} and \ref{Prop: Basic properties of PSS-image invars Main} imply that $c_{im}([\Sigma];-)$ satisfies the first and third properties respectively, while the fact that $c_{OS}([\Sigma],-)$ satisfies the non-triviality condition and bounds $c_{im}([\Sigma];-)$ from below implies that the non-triviality condition holds for the PSS-image spectral invariant.  The following corollary may be deduced from the proof given in Section 5.2.2 of \cite{HRS16} of the max formula for $c_{OS}([\Sigma];-)$ simply by noting that the arguments provided therein apply equally well to $c_{im}([\Sigma];-)$. In particular, the proof of the Max formula therein proceeds by carefully transferring the fundamental class of a $C^2$-small Morse function along well-chosen continuation maps, and so the resulting cycle manifestly lies in the image of some chain-level PSS map, and so establishes the max formula for $c_{im}([\Sigma];-)$.
\begin{cor}
For $\Sigma \neq S^2$ and any (time-independent) $H \in C^\infty(\Sigma)$, 
\begin{align*}
c_{im}([\Sigma];H)&=c_{OS}([\Sigma];H).
\end{align*}
\end{cor}

\section{Application: Positively transverse foliations from Floer theory}\label{Sec: Pos Transverse foliations}
In this section, we turn our attention to the structure of the moduli spaces of Floer cylinders for a Floer regular pair $(H,J)$, and their relation to the capped braid-theoretic topology of $\Per{H}$. The main result in this section is that there is a certain topological condition that one may place on capped braids formed by collections of elements of $\Per{H}$ --- which we call being \textit{maximally unlinked relative the Morse range} --- such that whenever $\hat{X} \subseteq \Per{H}$ satisfies this property, we may construct a singular foliation $\mathcal{F}^{\hat{X}}$of $S^1 \times \Sigma$ having singular leaves precisely the graphs of loops $x$ such that $\hat{x} \in \hat{X}$ and with regular leaves parametrized by annuli of the form $(s,t) \mapsto (t,u(s,t))$ for $u \in \widetilde{\mathcal{M}}(\hat{x},\hat{y};H,J)$ some Floer cylinder, with $\hat{x}, \hat{y} \in \hat{X}$. Owing to the fact that the regular leaves are parametrized by (projections of the graphs of) Floer cylinders, $\mathcal{F}^{\hat{X}}$ is \textit{positively transverse} to the vector field $\partial_t \oplus X^H$, and can be used to induce a singular foliation $\mathcal{F}_0^{\hat{X}}$ on $\Sigma$ itself, which is \textit{positively transverse in the sense of Le Calvez} (see Le Calvez's work in \cite{LeC05} and \cite{LeC06}). Thus, the results in this section may be viewed as providing a way to construct foliations of Le Calvez-type by purely Floer-theoretic considerations. Moreover, we show that $\mathcal{F}_0^{\hat{X}}$ is a singular foliation \textit{of Morse type}, and that $\mathcal{F}^{\hat{X}}$ may be viewed as a Morse-Bott foliation associated to a finite-dimensional reduction of the action functional $A^{\hat{X}} \in C^\infty(S^1 \times \Sigma)$. This provides us with Morse-theoretic models for those parts of the Hamiltonian Floer complex lying in the homologically non-trivial range.
\par
Section \ref{FoliatedSectorsSubsection} studies the relationship of $\widetilde{M}(\hat{x},\hat{y};H,J)$ to the topology of the capped braid $\lbrace \hat{x}, \hat{y} \rbrace$ and deduces conditions under which the maps $(s,t) \mapsto (t,u(s,t)) \in S^1 \times \Sigma$, $u \in  \widetilde{M}(\hat{x},\hat{y};H,J)$, provide a smooth foliation of some subset $W(\hat{x},\hat{y}) \subset S^1 \times \Sigma$. We call $W(\hat{x},\hat{y})$ a \textit{foliated sector}. That such moduli spaces can be used to construct such foliated sectors is essentially due to Hofer-Wysocki-Zehnder in \cite{HWZ99} (see also \cite{HWZ03}). In an effort to understand how these foliated sectors may be glued together to form a singular foliation of $S^1 \times \Sigma$, Section \ref{LinkingAndTheFloerComplex} introduces the notion of capped braids $\hat{X} \subseteq \Per{H}$ which are \textit{maximally unlinked relative the Morse range}, and associates to each such capped braid a chain complex $CF_*(\hat{X};H,J)$ which is not quite a subcomplex of $CF_*(H,J)$, but whose differential counts Floer cylinders which run between orbits in $\hat{X}$ --- equivalently, once the existence of the singular foliation $\mathcal{F}^{\hat{X}}$ is established, the differential counts the `rigid' leaves of $\mathcal{F}^{\hat{X}}$ which run between the graphs of orbits of index difference $1$.  Section \ref{ConstructingFoliationsSection} establishes the existence of the foliation $\mathcal{F}^{\hat{X}}$ and some of its basic properties, which imply in particular that $CF_*(\hat{X};H,J)$ has a Morse-theoretic model, given by a finite-dimensional reduction of the action functional. Section \ref{Sec: Consequences2} concludes with some novel consequences for the structure of Hamiltonian isotopies on surfaces, and a short discussion of the relationship between the singular foliations we produce and those appearing in Le Calvez's theory of transverse foliations.

\subsection{Foliated sectors and Floer moduli spaces as leaf spaces}\label{FoliatedSectorsSubsection}
We begin with some observations on the relationship between the topology of the capped braid $\lbrace \hat{x},\hat{y} \rbrace$ for $\hat{x},\hat{y}  \in \Per{H}$, and the existence of Floer cylinders running between them.
\begin{lemma}\label{LinkingBoundLem}
Let $(\mathcal{H},J)$ be an adapted homotopy of Floer data with $(H^\pm,J^\pm)$ Floer regular. Suppose that $\hat{x}^\pm \in \Per{H^-} \cap \Per{H^+}$, with $x^- \neq x^+$, and that $\mathcal{M}(\hat{x}^-,
\hat{x}^+;\mathcal{H},J), \mathcal{M}(\hat{x}^-,\hat{x}^-;\mathcal{H},J),$ and $\mathcal{M}(\hat{x}^+,\hat{x}^+;\mathcal{H},J)$ are all non-empty, then $b(\hat{x}^-; H^-) \leq \ell(\hat{x}^-,\hat{x}^+) \leq a(\hat{x}^+; H^+)$.
\end{lemma}
\begin{proof}
That $b(\hat{x}^-; H^-) \leq \ell(\hat{x}^-,\hat{x}^+)$ follows from applying Corollary \ref{EmergeConvergeLinking} with $u_0 \in \mathcal{M}(\hat{x}^-,\hat{x}^+;\mathcal{H},J)$ and $u_1 \in \mathcal{M}(\hat{x}^-,\hat{x}^-;\mathcal{H},J)$. The second inequality uses $u_1 \in \mathcal{M}(\hat{x}^+,\hat{x}^+;\mathcal{H},J)$ instead.
\end{proof}
Applying the preceding lemma in the case where $(\mathcal{H},J)$ is $s$-independent yields:
\begin{cor}\label{cor:FloerBoundaryModLinkingBounds}
Let $(H,J)$ be a non-degenerate Floer pair and suppose that $\widetilde{\mathcal{M}}(\hat{x},\hat{y};H,J) \neq \emptyset$, $x \neq y$, then 
\begin{align*}
b(\hat{x}) \leq &\ell(\hat{x},\hat{y}) \leq a(\hat{y}).
\end{align*}
\end{cor}

Applying Lemma \ref{EndsLinkLemma} to the constant cylinder $u \in \widetilde{\mathcal{M}}(\hat{x}',\hat{x}';H,J)$ and $v \in \widetilde{\mathcal{M}}(\hat{x},\hat{y};H,J)$ establishes the following proposition.
\begin{proposition}\label{prop:LinkingMonotoneAlongFloerDifferential}
Let $(H,J)$ be Floer regular and let $\hat{x}, \hat{y} \in \Per{H}$ with $\hat{y} \in \supp \partial_{H,J} \hat{x}$, then for all $\hat{x}' \in \Per{H}$, $x' \not \in \lbrace x,y \rbrace$, we have $\ell(\hat{x},\hat{x}') \leq \ell(\hat{y},\hat{x}')$.
\end{proposition}

Lemma \ref{CZindexAsympWindCompLemma} combines with Corollary \ref{cor:FloerBoundaryModLinkingBounds} to show:
\begin{cor}\label{KlinkingLemma}
If $\mu(\hat{x}), \mu(\hat{y}) \in \lbrace 2k-1, 2k, 2k+1 \rbrace$ for some $k \in \Z$ and $\widetilde{\mathcal{M}}(\hat{x},\hat{y};H,J) \neq \emptyset$, then $\ell(\hat{x},\hat{y})=-k$.
\end{cor}
The main geometric input for this section is the following proposition (cf. Theorem $5.6$ in \cite{HWZ99}).
\begin{proposition}\label{EvalDiffeoLemma}
Let $(H,J)$ be Floer regular, and $\hat{x},\hat{y} \in \Per{H}$. If $2k-1 \leq \mu(\hat{x}),\mu(\hat{y}) \leq 2k+1$, for some $k \in \Z$, then the map $\widetilde{Ev}: \R \times S^1 \times \widetilde{\mathcal{M}}(\hat{x},\hat{y};H,J) \rightarrow \R \times S^1 \times \Sigma$ defined by $\widetilde{Ev}(s,t,u):=\tilde{u}(s,t)=(s,t,u(s,t))$ is a diffeomorphism onto its image.
\end{proposition}
\begin{proof}
We may suppose that $\widetilde{\mathcal{M}}(\hat{x},\hat{y};H,J) \neq \emptyset$, or else the proposition is vacuously true. Moreover, if $\hat{x}=\hat{y}$, then the statement is obvious. Thus, we may suppose that $\mu(\hat{x})-\mu(\hat{y}) \in \lbrace 1,2 \rbrace$. We will show that $\widetilde{Ev}$ is a proper injective immersion. To see that $\widetilde{Ev}$ is one-to-one, note that $\widetilde{Ev}(s,t,u)=\widetilde{Ev}(s',t',v)$ if and only if $u \neq v$ and the graphs of $\tilde{u}$ and $\tilde{v}$ intersect over $(s,t)=(s',t')$. Thus, by Lemma \ref{EndsLinkLemma} $\widetilde{Ev}$ fails to be injective only if there exist $u,v \in \widetilde{M}(\hat{x},\hat{y};H,J)$, $u \neq v$ such that
\begin{align*}
b(\hat{x}) \leq \ell_{-\infty}(u,v) &< \ell_{\infty}(u,v) \leq a(\hat{y}),
\end{align*}
The hypothesized constraints on the Conley-Zehnder indices of $\hat{x}$ and $\hat{y}$ imply that $\mu(\hat{x}) \in \lbrace 2k+1, 2k \rbrace$ and $\mu(\hat{y}) \in \lbrace 2k,2k-1 \rbrace$, and so by Lemma \ref{CZindexAsympWindCompLemma} we deduce that $b(\hat{x})=-k=a(\hat{y})$, which contradicts the above inequality. Thus, $\widetilde{Ev}$ is injective.
\par
That $\widetilde{Ev}$ is proper essentially follows from compactness results in Floer theory; if \\ $\lbrace (s_n,t_n,u_n) \rbrace_{n \in \N} \subseteq \R \times S^1 \times \widetilde{\mathcal{M}}(\hat{x},\hat{y};H,J)$ is some sequence which eventually leaves any compact set, then either $s_n \rightarrow \pm \infty$ and $(s_n,t_n,u_n)$ converges to a point on either the graph $\tilde{x}(s,t)=(s,t,x(t))$ or on the graph $\tilde{y}(s,t)=(s,t,y(t))$, or $(s_n)_{n \in \N}$ remains bounded, in which case $(s_n,t_n,u_n)$ must converge to a point on the graph of some broken Floer cylinder between $x$ and $y$. In either case, the sequence $(s_n,t_n,u_n)$ eventually leaves every compact subset of $\im \widetilde{Ev}$.
\par 
It remains to show that $\widetilde{Ev}$ is an immersion when $\mu(\hat{x})-\mu(\hat{y}) \in \lbrace 1,2 \rbrace$. We note that
\begin{align*}
T (\R \times S^1 \times \widetilde{\mathcal{M}}(\hat{x},\hat{y};H,J))&= T (\R \times S^1) \oplus (\ker (D \mathcal{F}))\vert_{\widetilde{\mathcal{M}}(\hat{x},\hat{y};H,J)},
\end{align*}
and since $d \widetilde{Ev}$ may be computed with respect to this splitting as
\begin{align*}
(d \widetilde{Ev})_{(s,t,u)}(\partial_s,\partial_t,\xi)&=(\partial_s,\partial_t,\xi(s,t)) \in T_{\tilde{u}(s,t)} \R \times S^1 \times \Sigma,
\end{align*}
the problem reduces to showing that any $\xi \in \ker (D \mathcal{F}_{H,J})_u$ with $\xi \neq 0$ is a nowhere-vanishing vector field along $u$, where $D \mathcal{F}_{H,J}$ denotes the vertical part of the linearized Floer operator. This follows by combining Proposition \ref{FloerKernelWinding} with Lemma \ref{CZindexAsympWindCompLemma} to deduce that $Z(\xi)=0$ whenever $\xi$ is not identically zero (from which it follows by Proposition \ref{FloerKernelWinding} that $\xi$ is nowhere vanishing).
\end{proof}
Note that as a consequence of the previous lemma, whenever $\mu(\hat{x}), \mu(\hat{y}) \in \lbrace 2k-1, 2k , 2k+1 \rbrace$ for $k \in \Z$, then $\widetilde{Ev}(\R \times S^1 \times \widetilde{\mathcal{M}}(\hat{x},\hat{y};H,J))$ carries a smooth $2$-dimensional foliation $\widetilde{\mathcal{F}}^{\hat{x},\hat{y}}$, the leaves of which are nothing but the graphs $\tilde{u}$ of $u \in \widetilde{\mathcal{M}}(\hat{x},\hat{y};H,J)$. 
\begin{definition}
For $\hat{x},\hat{y} \in \Per{H}$, \textbf{the connecting subspace} of $\hat{x}$ and $\hat{y}$ will denote the subspace 
\begin{align*}
W(\hat{x},\hat{y}) &:=\lbrace (t,u(s,t)) \in S^1 \times \Sigma: s \in \R, u \in \widetilde{\mathcal{M}}(\hat{x},\hat{y};H,J) \rbrace.
\end{align*}
\end{definition}
Remark that if we write $\widetilde{W}(\hat{x},\hat{y}):= \widetilde{Ev}(\R \times S^1 \times \widetilde{\mathcal{M}}(\hat{x},\hat{y};H,J))$, then the map $\check{\pi}: \R \times S^1 \times \Sigma \rightarrow S^1 \times \Sigma$ restricts to a projection
$\check{\pi}: \widetilde{W}(\hat{x},\hat{y}) \rightarrow W(\hat{x},\hat{y})$, with fiber $\check{\pi}^{-1}(t,p) = \lbrace \widetilde{Ev}(s,t,u): u(s,t)=p \rbrace$, which under the hypotheses of Proposition \ref{EvalDiffeoLemma} may be identified via $\widetilde{Ev}$ with the orbit of any $(s_0,t,u_0) \in \R \times S^1 \times \widetilde{\mathcal{M}}(\hat{x},\hat{y};H,J)$ such that $u_0(s_0,t)=p$ under the $\R$-action $\tau \cdot (s,t,u)= (s-\tau,t,u^\tau)$, for $\tau \in \R$, where $u^\tau(s,t)=u(s+\tau,t)$. Consequently, $\widetilde{Ev}$ descends to a well-defined map 
\begin{align*}
\check{Ev}: (\R \times S^1 \times \widetilde{\mathcal{M}}(\hat{x},\hat{y};H,J))/ \R &\rightarrow  W(\hat{x},\hat{y}) \\
[s,t,u] &\mapsto (t,u(s,t)).
\end{align*}
Hence, under the hypotheses of Proposition \ref{EvalDiffeoLemma}, $\check{\pi}$ restricts to a submersion on $\widetilde{W}(\hat{x},\hat{y})$ with fiber diffeomorphic to $\R$. Moreover, if we choose a section $\sigma: \mathcal{M}(\hat{x},\hat{y};H,J) \rightarrow \widetilde{\mathcal{M}}(\hat{x},\hat{y};H,J)$, then we may thereby (non-canonically) identify
\begin{align*}
\phi_\sigma: \R \times S^1 \times \mathcal{M}(\hat{x},\hat{y};H,J) &\xrightarrow{\simeq} (\R \times S^1 \times \widetilde{\mathcal{M}}(\hat{x},\hat{y};H,J))/ \R  \\
(s,t,[u]) &\mapsto [s,t,\sigma([u])].
\end{align*}
Finally, to understand the behaviour of the foliation $\widetilde{\mathcal{F}}^{\hat{x},\hat{y}}$ under this projection, note that $\ker d \check{\pi}= \langle \partial_s \rangle$, and that since the tangent space of any leaf of $\widetilde{\mathcal{F}}$ is given by $\langle \partial_s + (\partial_s u)_{u(s,t)}, \partial_t + (\partial_t u)_{u(s,t)} \rangle$, where $\partial_s u \in \ker (d\mathcal{F})_u$, and is therefore nowhere-vanishing whenever $u$ is not an orbit cylinder $u(s,t)=x(t)$ for some $x \in Per(H)$, by our index constraint. So the leaves of the foliation are nowhere tangent to the fibers of the projection map whenever $\hat{x} \neq \hat{y}$. As a consequence, we deduce:
\begin{cor}\label{cor:FoliatedSectorCorollary}
Let $\hat{x},\hat{y} \in \Per{H}$ satisfy $\mu(\hat{x}),\mu(\hat{y}) \in \lbrace 2k-1,2k,2k+1 \rbrace$, for some $k \in \Z$ then for any section $\sigma: \mathcal{M}(\hat{x},\hat{y};H,J) \rightarrow \widetilde{\mathcal{M}}(\hat{x},\hat{y};H,J)$, as above,
\begin{align*}
\check{Ev} \circ \phi_{\sigma}:  \R \times S^1 \times  \mathcal{M}(\hat{x},\hat{y};H,J) &\mapsto W(\hat{x},\hat{y}) \\
(s,t,[u]) &\mapsto (t,\sigma([u])(s,t))
\end{align*}
is a smooth embedding. Moreover, writing $\sigma([u])=u_\sigma$, the partition $\mathcal{F}^{\hat{x},\hat{y}}:= \lbrace \im \check{u}_\sigma \rbrace_{[u] \in \mathcal{M}(\hat{x},\hat{y};H,J)}$ is a smooth $2$-dimensional foliation of $W(\hat{x},\hat{y})$ whenever $\hat{x} \neq \hat{y}$. 
\end{cor}

\subsection{The restricted complex associated to a capped braid}
\label{LinkingAndTheFloerComplex}
Let $(H,J)$ be a non-degenerate Floer pair. To any capped braid $\hat{X} \subseteq \Per{H}$, we may associate the submodule $C_*(\hat{X}):=\Lambda_\omega \langle \hat{x} \rangle_{\hat{x} \in \hat{X}}$ of $CF_*(H,J)$, which comes with the projection $\pi^{\hat{X}}: CF_*(H,J) \rightarrow C_*(\hat{X})$ associated to any splitting $CF_*(H,J)=C_*(\hat{X}) \oplus C_*(\hat{Y})$, for $\hat{Y} \subseteq \Per{H}$ any capped braid such that $Per_0(H) = X \sqcup Y$. $C_*(\hat{X})$ is not generally a subcomplex of $CF_*(H,J)$, since there is no reason that Floer cylinders should only run between strands of $\hat{X}$. However, we will see that if we define the \textbf{restricted differential} $\partial^{\hat{X}}:=\pi^{\hat{X}} \circ \partial_{H,J}$, then under suitable conditions on $\hat{X}$,
\begin{align*}
CF_*(\hat{X};H,J)&:= ( C_*(\hat{X}), \partial^{\hat{X}})
\end{align*}
is a chain complex.
\subsubsection{Maximal unlinkedness relative the Morse range}
\begin{definition}
If $(H,J)$ is Floer regular, then for any capped braid $\hat{X} \subseteq \Per{H}$, we define
\begin{align*}
Pos(\hat{X}) &:= \lbrace \sigma \in CF_*(H,J): \forall \hat{\gamma} \in \supp \sigma, \forall \hat{x} \in \hat{X}, \; \ell(\hat{x},\hat{\gamma}) \geq 0 \rbrace, \\
Pos^*(\hat{X}) &:= \lbrace \sigma \in Pos(\hat{X}): \forall \hat{\gamma} \in \supp \sigma, \exists \hat{x} \in \hat{X}, \text{such that\;} \ell(\hat{x}, \hat{\gamma}) > 0 \rbrace.
\end{align*}
We define $Neg(\hat{X})$ and $Neg^*(\hat{X})$ in the obvious manner simply by reversing the inequalities in the above. 
\end{definition}
\begin{definition}
Let $\hat{X} \subseteq \Per{H}$ be a capped braid for some Hamiltonian $H$. $\hat{X}$ will be said to be \textbf{maximally unlinked} if it is unlinked, and if for any $\hat{y} \in \Per{H}$ either $\hat{y} \in \hat{X}$ or else  $\hat{y}$ and $\hat{X}$ are linked. We write $mu(H)$ for the collection of all such capped braids.
\end{definition}

\begin{definition}
Let $\hat{X} \subseteq \Per{H}$ be a capped braid for some Hamiltonian $H$. $\hat{X}$ will be said to be \textbf{maximally unlinked relative the Morse range} if $\hat{X}$ is unlinked, $\mu(\hat{x}) \in \lbrace -1,0,1 \rbrace$ for all $\hat{x} \in \hat{X}$, and moreover if for any $\hat{y} \in \Per{H}$ such that $\mu(\hat{y}) \in \lbrace -1,0,1 \rbrace$, either $\hat{y} \in \hat{X}$ or $\hat{y}$ and $\hat{X}$ are linked. We write $murm(H)$ for the collection of all capped braids $\hat{X} \subseteq \Per{H}$ which are maximally unlinked relative the Morse range.
\end{definition}
The next lemma is a direct consequence of the definitions.
\begin{lemma}\label{lem:PosNegInKerOfProj}
Let $\hat{X} \subseteq \Per{H}$ be an unlinked braid, then $Pos^*(\hat{X}), Neg^*(\hat{X}) \subseteq \ker \pi^{\hat{X}}$.
\end{lemma}

The following situation will occur frequently enough that it will be useful to isolate it as a lemma.
\begin{lemma}\label{lem:SingularCobordGivesStrictPositivity}
Let $(H^\pm,J^\pm)$ be Floer regular, $(\mathcal{H},J) \in \HJ$, and $\hat{X}^\pm=(\hat{x}_1^\pm, \ldots, \hat{x}_k^\pm) \in \Per{H^\pm}$. Suppose that $\hat{X}^-$ and $\hat{X}^+$ are both unlinked, $\mu(\hat{x}_1^-)=\mu(\hat{x}^+_1) \in \lbrace 0,1 \rbrace$, and that there exists $u^i \in \mathcal{M}(\hat{x}^-_i,\hat{x}^+_i; \mathcal{H},J)$ for $i=1, \ldots,k$. For any $\hat{y}^+ \in \Per{H^+}$, if there exists $v \in \mathcal{M}(\hat{x}^-_1,\hat{y};\mathcal{H},J)$, such that $\hat{X} \cup \lbrace \hat{y}^+ \rbrace$ is linked, then $\hat{y}^+ \in Pos^*(\hat{X})$.
\end{lemma}
\begin{proof}
Note that we have $0=\ell(\hat{x}_1,\hat{x}_i) \leq \ell(\hat{x}^+,\hat{x}_i)$ for $i=2, \ldots, k$ by Lemma \ref{EndsLinkLemma} and $0=b(\hat{x}_1) \leq \ell(\hat{x}_1,\hat{x}^+)$ by Lemma \ref{LinkingBoundLem}, so we need only show that $\ell(\hat{x}^+,\hat{x}_i) >0$ for some $i=1, \ldots, k$.
We write $h(s)=(u^1_s, \ldots, u^k_s,v_s)$, $s \in \R$. $h$ does not induce a braid cobordism, because $u_1$ and $v$ limit to the same orbit as $s \rightarrow -\infty$, however Lemma \ref{EndsLinkLemma} implies that for $R >0$ sufficiently large, $h\vert_{(-R,\infty)}$ induces a braid cobordism from $\hat{X} \cup \lbrace \hat{v}_{-R} \rbrace$ to $\hat{X} \cup \lbrace \hat{x}^+ \rbrace$. Since $b(\hat{x}_1) = 0$, there are two possibilities: either $0 < \ell_{-\infty}(u^1,v)$, or $\ell_{-\infty}(u^1,v)=0$. In the former case, Lemma \ref{EndsLinkLemma} immediately implies that $0<\ell_{\infty}(u^1,v)=\ell(\hat{x}_1,\hat{x}^+)$, and we are done. 
\par We may therefore assume that $\ell_{-\infty}(u^1,v)=0$. In this case, $\hat{X} \cup \lbrace \hat{v}_{-R} \rbrace$ is unlinked. Indeed by Corollary \ref{EndsCor}, $R>0$ may be chosen such that $\tilde{v}$ has no intersections with $\tilde{u}_i$, $i=1, \ldots, k$ for $s <-R$, and the property of being unlinked is invariant under $0$-homotopies, so $\hat{X} \cup \lbrace \hat{v}_{-R} \rbrace$ is unlinked only if $\lbrace \hat{x}_1, \hat{v}_{-R} \rbrace$ is unlinked. But we may take $R>0$ sufficiently large such that $\hat{v}_{-R} \in \cL{\Sigma}$ lies in an exponential neighbourhood of $\hat{x}_1$, and in this neighbourhood the homological linking number reduces to the classical winding number by Proposition \ref{prop:HLinkEqualsLink}, and so that $\lbrace \hat{x}_1, \hat{v}_{-R} \rbrace$ is unlinked follows directly from the fact that the winding number classifies homotopy classes of loops into $\R^2 \setminus 0$. 
\par 
Thus, $\hat{X} \cup \lbrace \hat{v}_{-R} \rbrace$ is unlinked, while $\hat{X} \cup \lbrace \hat{x}^+ \rbrace$ is linked, whence the graphs of some of the strands of $h$ must intersect. Since $\ell_{- \infty}(u^i,u^j)=\ell_{\infty}(u^i,u^j)=\ell(\hat{x}_i,\hat{x}_j)=0$ for $i \neq j$, it follows from Lemma \ref{EndsLinkLemma} that the graphs of $u^i$ and $u^j$ are disjoint for $i \neq j$. Thus, there exists some $i=1, \ldots, k$ such that the graphs of $u^i$ and $v$ intersect, so $0 < \ell_{\infty}(u^i,v)=\ell(\hat{x}_i,\hat{x}^+)$, as claimed.
\end{proof}
\begin{remark}
In the case where $(\mathcal{H},J)$ are independent of $s$, then the hypothesis $u^i \in \mathcal{M}(\hat{x}_i,\hat{x}_i; \mathcal{H},J)$ is always satisfied for any $\hat{x}_i \in \Per{H}$, since the constant cylinder $u_i(s,t)=x_i(t)$ is always a solution of the $s$-independent Floer equation.
\end{remark}

\begin{lemma}
Let $\hat{X} \in murm(H)$, then for all $\hat{x} \in \hat{X}$,
$\partial_{H,J} \hat{x} \in \Z_2 \langle \hat{x} \rangle_{\hat{x} \in \hat{X}} \oplus Pos^*(\hat{X})$.
\end{lemma}
\begin{proof}
Let $\hat{x} \in \hat{X}$, and $\hat{y} \in \supp \partial_{H,J} \hat{x}$. Either $\mu(\hat{x})=-1$ or $\mu(\hat{x}) \in \lbrace 0,1 \rbrace$. If $\mu(\hat{x})=-1$, then $b(\hat{x})=1$ and so Corollary \ref{cor:FloerBoundaryModLinkingBounds} implies that $1 \leq \ell(\hat{x},\hat{y})$ while Proposition \ref{prop:LinkingMonotoneAlongFloerDifferential} implies that $0 \leq \ell(\hat{x}',\hat{y})$ for all $\hat{x}' \in \hat{X}$, $\hat{x}' \neq \hat{x}$, and so $\hat{y} \in Pos^*(\hat{X})$. If $\mu(\hat{x}) \in \lbrace 0,1 \rbrace$, then Corollary \ref{cor:FloerBoundaryModLinkingBounds} and Proposition \ref{prop:LinkingMonotoneAlongFloerDifferential} imply that $0 \leq \ell(\hat{x}',\hat{y})$ for all $\hat{x}' \in \hat{X}$. To see that $\hat{y} \in \hat{X} \cup Pos^*(\hat{X})$, note that either $\hat{X} \cup \lbrace \hat{y} \rbrace$ is unlinked, in which case $\hat{y} \in \hat{X}$ by the maximality of $\hat{X}$, or $\hat{X} \cup \lbrace \hat{y} \rbrace$ is linked, in which case Lemma \ref{lem:SingularCobordGivesStrictPositivity} directly implies that $\hat{y} \in Pos^*(\hat{X})$.
\end{proof}

Proposition \ref{prop:LinkingMonotoneAlongFloerDifferential} immediately implies the following lemma.
\begin{lemma}\label{lem:PosInvarUnderDiff}
Let $\hat{X} \subseteq \Per{H}$ be any capped braid, then $\partial_{H,J} Pos^*(\hat{X}) \subseteq Pos^*(\hat{X})$.
\end{lemma}

\begin{theorem}\label{thm:RestrictedDifferentialGivesComplex}
Let $\hat{X} \in murm(H)$, then $CF_*(\hat{X};H,J)$ is a chain complex. That is, $\partial^{\hat{X}} \circ \partial^{\hat{X}} = 0$.
\end{theorem}
\begin{proof}
First, consider that $\Sigma$ is either an aspherical surface, in which case $\Lambda_\omega= \Z_2$, or else a sphere, in which case $\Sigma$ has minimal Chern number $2$, and so in this case $CF_*(\hat{X};H,J)=\Lambda_\omega \langle \hat{x} \rangle_{\hat{x} \in \hat{X}}$ vanishes in any degree congruent to $2$ mod $4$. Thus, by the $\Lambda_\omega$-equivariance of the Floer boundary map, it suffices in all cases to prove that $(\partial^{\hat{X}})^2$ vanishes in the Morse range. To wit, by the previous two lemmas, we see that for any $\hat{x} \in \hat{X}$, 
\begin{align*}
\partial_{H,J} \hat{x} &= \partial^{\hat{X}} \hat{x} + (\pi^{Pos^*(\hat{X})} \circ \partial_{H,J})(\hat{x}),
\end{align*}
where $\pi^{Pos^*(\hat{X})}$ denotes projection onto $Pos^*(\hat{X})$. Thus, since $\partial_{H,J}^2=0$, $(\partial^{\hat{X}})^2 \hat{x} + \sigma=0$, where $\sigma \in Pos^*(\hat{X})$. It follows that $(\partial^{\hat{X}})^2 \hat{x} =0$.
\end{proof}
We will write $HF_*(\hat{X};H,J)$ for the homology of the complex $CF_*(\hat{X};H,J)$ when $\hat{X} \in murm(H)$. It will turn out that this homology is independent of $J$ (for instance, as a consequence of Corollary \ref{cor:MorseModels}) but we will not need this fact.

\subsubsection{Dominating Morse functions and dominating homotopies}\label{sec:ControllingPseud}
Throughout this section, let $(H,J)$ be a fixed Floer regular pair.
\begin{definition}
Let $\hat{X} = \lbrace \hat{x}_1, \ldots, \hat{x}_k \rbrace \subseteq \Per{H}$. We will say that a $C^2$-small Morse function $f \in C^\infty(\Sigma)$ is \textbf{$\hat{X}$-dominating} if there exist cappings $w_i: D^2 \rightarrow \Sigma$ for $x_i$, $i=1, \ldots,k$ such that
\begin{enumerate}
\item $w_i(0) \neq w_j(0)$ for $i \neq j$,
\item $w_i(0) \in Crit(f)$, for each $i=1, \ldots, k$ and
\item $\mu_{CZ}(\widehat{w_i(0)};f)=\mu_{CZ}(\hat{x}_i;H)$ for $i=1, \ldots, k$, where $\widehat{w_i(0)}$ denotes the trivially capped constant orbit of $f$ based at $w_i(0)$.
\end{enumerate}
The point $p_i:=w_i(0) \in Crit(f)$, will be called the \textbf{corresponding critical point} of $\hat{x}_i \in \hat{X}$ for $i=1, \ldots,k$.
\end{definition}
\begin{remark}
\begin{enumerate}
\item Remark that the $C^2$-smallness condition on $f$ in the definition of of $\hat{X}$-dominating functions forces 
\begin{align*}
\mu_{CZ}(\widehat{w_i(0)};f) = \mu_{CZ}(\hat{x}_i;H)  &\in \lbrace -1, 0 ,1 \rbrace
\end{align*}
for $i=1, \ldots,k$.

\item It is obvious that $\hat{X}$-dominating Morse functions exist for any $\hat{X} \subseteq \Per{H}$ such that $\mu(\hat{x}) \in \lbrace -1,0,1 \rbrace$ for all $\hat{x} \in \hat{X}$. Indeed, take any $0$-cobordism from a trivial capped $k$-braid to $\hat{X}$ and use local models for Morse critical points to introduce critical points with the desired index at the points which are the images of the constant orbits making up the trivial $k$-braid in question. Of course, there is no reason that a $\hat{X}$-dominating Morse function should have \textit{all} its critical points arising in such a way (although we shall see these latter functions exist when $\hat{X} \in murm(H)$). 
\end{enumerate}
\end{remark}

\begin{definition}
Let $\hat{X}=\lbrace \hat{x}_1, \ldots, \hat{x}_k \rbrace \subseteq \Per{H}$. We will say that an adapted homotopy of Floer data $(\mathcal{H},\mathbb{J})$ is \textbf{$\hat{X}$-dominating (as a homotopy to $(H,J)$)} if there exists a Floer regular pair $(f,J^-)$ for $f$ a $\hat{X}$-dominating Morse function such that
\begin{enumerate}
\item $(\mathcal{H},\mathbb{J}) \in \HJ_{reg}(f,J^-;H,J)$, and
\item $\mathcal{M}(\hat{p}_i,\hat{x}_i; \mathcal{H},\mathbb{J}) \neq \emptyset$ for each $i=1, \ldots, k$, where $\hat{p}_i$ is the trivially capped critical point of $f$ corresponding to $\hat{x}_i$.
\end{enumerate}
Whenever $f$ is a $\hat{X}$-dominating Morse function, and $(f,J^-)$ is Floer regular, we will write $\HJ^{\hat{X}}(f,J^-;H,J)$ for the set of all $\hat{X}$-dominating homotopies from $(f,J^-)$ to $(H,J)$.
\par
Similarly, we will say that $(\mathcal{H},\mathbb{J})$ is \textbf{$\hat{X}$-dominating (as a homotopy from $(H,J)$)} if there exists a Floer regular pair $(f,J^+)$ for $f$ a $\hat{X}$-dominating Morse function such that
\begin{enumerate}
\item $(\mathcal{H},\mathbb{J}) \in \HJ_{reg}(H,J;f,J^+)$, and
\item $\mathcal{M}(\hat{x}_i,\hat{p}_i; \mathcal{H},\mathbb{J}) \neq \emptyset$ for each $i=1, \ldots, k$.
\end{enumerate}
Whenever $f$ is a $\hat{X}$-dominating Morse function, and $(f,J^+)$ is Floer regular, we will write $\HJ^{\hat{X}}(H,J;f,J^+)$ for the set of all $\hat{X}$-dominating homotopies from $(H,J)$ to $(f,J^+)$.
\end{definition}

\begin{proposition}
Let $\hat{X} \subseteq \Per{H}$. 
\begin{enumerate}
\item There exist $\hat{X}$-dominating homotopies of Floer data to $(H,J)$ if and only if $\hat{X}$ is a positive capped braid and $\mu(\hat{x}) \in \lbrace -1,0,1 \rbrace$ for all $\hat{x} \in \hat{X}$. 
\item There exist $\hat{X}$-dominating homotopies of Floer data from $(H,J)$ if and only if $\hat{X}$ is a negative capped braid and $\mu(\hat{x}) \in \lbrace -1,0,1 \rbrace$ for all $\hat{x} \in \hat{X}$. 
\end{enumerate}
\end{proposition}
\begin{proof}
It is obviously a necessary condition for the existence of a $\hat{X}$-dominating homotopy of Floer data that $\hat{X}$ be a positive capped braid and have all strands lying in the Morse range.  Let us show that it is sufficient.
\par
Let $\hat{X} \subseteq \Per{H}$ be a positive capped braid and suppose that all strands lie in the Morse range. Since all strands of $\hat{X}$ lie in the Morse range, there exists a $\hat{X}$-dominating Morse function. Moreover, since $\hat{X}$ is a positive capped braid,  we may assume that the capping disks $w_i$ appearing in the definition of a $\hat{X}$-dominating Morse function are such that for $i=1, \ldots,k$, the maps
\begin{align*}
v_i(s,t) &:= w_i(se^{2 \pi \i t}), (s,t) \in [0,1] \times S^1
\end{align*} 
define a positive $0$-cobordism from $\lbrace \widehat{w_1(0)}, \ldots, \widehat{w_k(0)} \rbrace$ to $\hat{X}$ and that for $i=1, \ldots, k$ the maps
\begin{equation*}
u_i(s,t):= \begin{cases}
v_i(0,t), & (s,t) \in (-\infty,0) \times S^1, \\
v_i(s,t), & (s,t) \in [0,1] \times S^1, \\
v_i(1,t), & (s,t) \in (1, \infty) 
\end{cases}
\end{equation*}
define a pre-model for a continuation cobordism from $(f,J^-)$ to $(H,J)$ for $J^- \in C^\infty(S^1;\mathcal{J}(\Sigma,\omega))$ such that $(f,J^-)$ is Floer regular. By Corollary \ref{Cor: Premodel implies regular model with non-empty moduli spaces}, it follows that a $\hat{X}$-dominating homotopy to $(H,J)$ exists. The proof of the second point is entirely dual to the preceding proof.
\end{proof}
\begin{remark}
Note that if $\hat{X} \subseteq \Per{H}$ is unlinked and has all strands in the Morse range, then $\hat{X}$-dominating homotopies exist both \textit{to} and \textit{from} $(H,J)$.
\end{remark}

\begin{proposition}\label{ContinuationControlProp}
Let $(H,J)$ be Floer regular and suppose that $\hat{X}=\lbrace \hat{x}_1,\ldots,\hat{x}_k \rbrace \in murm(H)$.
If $(\mathcal{H},\mathbb{J}) \in \HJ^{\hat{X}}(f,J^-;H,J)$, then the continuation map  
\begin{align*}
h_{\mathcal{H}}: CF_*(f,J^-) \rightarrow CF_*(H,J)
\end{align*}
satisfies the following.
\begin{enumerate}
\item For all $\hat{x}_i \in \hat{X}_{(1)} \cup \hat{X}_{(-1)}$, $h_{\mathcal{H}}(\hat{p}_i) = \hat{x}_i + \sigma$, where $\supp \sigma \subseteq Pos^*(\hat{X})$.

\item For all $\hat{q} \in \Per{f} \setminus \lbrace \hat{p}_1, \ldots, \hat{p}_k \rbrace$ with $\mu(\hat{q}) =  1$, $h_{\mathcal{H}}(\hat{q}) \in Pos^*(\hat{X})$.

\item For all $\hat{q} \in \Per{f} \setminus \lbrace \hat{p}_1, \ldots, \hat{p}_k \rbrace$ with $\mu(\hat{q}) =  -1$, $h_{\mathcal{H}}(\hat{q}) \in \Z_2 \langle \hat{x} \rangle_{\hat{X}_{(-1)}} \oplus Pos^*(\hat{X})$.

\item For all $\hat{q} \in \Per{f}$ with $\mu(\hat{q}) = 0$, $h_{\mathcal{H}}(\hat{q}) \in \Z_2 \langle \hat{x} \rangle_{\hat{x} \in \hat{X}} \oplus Pos^*(\hat{X})$.
\end{enumerate}
\end{proposition}
\begin{proof}
We write $\hat{X}=\lbrace \hat{x}_1, \ldots, \hat{x}_k \rbrace$. To prove $(1)$, note that since $(\mathcal{H},\mathbb{J})$ is $\hat{X}$-dominating, for each $\hat{x}_i \in \hat{X}_{(1)} \cup \hat{X}_{(-1)}$ we have
\begin{align*}
\mathcal{M}(\hat{p}_i,\hat{x}_i;\mathcal{H},\mathbb{J}) &\neq \emptyset.
\end{align*}
By Corollary \ref{Cor: Cont moduli space is 0 or 1}, this implies that $\vert\mathcal{M}(\hat{p}_i,\hat{x}_i;\mathcal{H},\mathbb{J})\vert=1$.
\par
Next, consider $\hat{x}_j \in \hat{X}$, $i \neq j$ having $\mu(\hat{x}_j)= \mu(\hat{x}_i)$ and suppose that 
\begin{align*}
\mathcal{M}(\hat{p}_i,\hat{x}_j; \mathcal{H},\mathbb{J}) &\neq \emptyset.
\end{align*}
If $\mu(\hat{x}_i)=\mu(\hat{x}_j)=1$, then we consider $u \in \mathcal{M}(\hat{p}_j,\hat{x}_j;\mathcal{H},\mathbb{J})$ and $v \in \mathcal{M}(\hat{p}_i,\hat{x}_j; \mathcal{H},\mathbb{J})$. We note that 
\begin{align*}
0= \ell_{-\infty}(u,v) \leq &\ell_{\infty}(u,v) \leq b(\hat{x}_j)=-1,
\end{align*}
which gives a contradiction. If $\mu(\hat{x}_i)=\mu(\hat{x}_j)=-1$, then we consider $u \in \mathcal{M}(\hat{p}_i,\hat{x}_i;\mathcal{H},\mathbb{J})$ and $v \in \mathcal{M}(\hat{p}_i,\hat{x}_j; \mathcal{H},\mathbb{J})$. We then obtain
\begin{align*}
1=a(\hat{p}_i)\leq \ell_{-\infty}(u,v) \leq \ell_{\infty}(u,v)=\ell(\hat{x}_i,\hat{x}_j)=0,
\end{align*} 
another contradiction. So all such moduli spaces are empty.
\par
Finally, let $\hat{y} \in \Per{H} \setminus \hat{X}$ have $\mu(\hat{y})=\mu(\hat{x}_i)$ and suppose that 
\begin{align*}
\mathcal{M}(\hat{p}_i,\hat{y}; \mathcal{H},\mathbb{J}) &\neq \emptyset.
\end{align*}
In this case, Lemma \ref{lem:SingularCobordGivesStrictPositivity} implies that $\hat{y} \in Pos^*(\hat{X})$. This proves $(1)$.
\par
To prove item $(2)$,let $\hat{x}_i \in \hat{X}$ have $\mu(\hat{x}_i)=\mu(\hat{q})=1$ and suppose that 
\begin{align*}
\mathcal{M}(\hat{q},\hat{x}_i;\mathcal{H},\mathbb{J}) &\neq \emptyset.
\end{align*}
Let $u \in \mathcal{M}(\hat{p}_i,\hat{x}_i;\mathcal{H},\mathbb{J})$ and $v \in \mathcal{M}(\hat{q},\hat{x}_i;\mathcal{H},\mathbb{J})$ arguing as we have above together with the facts that $\ell(\hat{q},\hat{p}_i)=0$ while $b(\hat{x}_i)=-1$ gives a contradiction, so all such moduli spaces are empty. If $\mathcal{M}(\hat{q},\hat{y}; \mathcal{H},\mathbb{J}) \neq \emptyset$, then taking $v \in \mathcal{M}(\hat{q},\hat{y}; \mathcal{H},\mathbb{J})$ and for each $\hat{x}_i \in \hat{X}$, letting $u_i \in \mathcal{M}(\hat{p}_i,\hat{x}_i;\mathcal{H},\mathbb{J})$, we see that $\lbrace u_i \rbrace_{i=1}^k \cup \lbrace v \rbrace$ is a regular model for a cobordism from $(f,J^-)$ to $(H,J)$, and so in particular induces a positive cobordism. That $\ell(\hat{x}_i,\hat{y}) >0$ for some $i=1, \ldots,k$ follows from the fact that since $\hat{X} \in murm(H)$ and $\mu(\hat{y})=1$, $\hat{X} \cup \lbrace \hat{y} \rbrace$ is linked, and so the above cobordism must have some intersections.
\par
The proof of $(3)$ is word-for-word the same as the proof of the proof of $(2)$, only in this case, it is possible that moduli spaces of the form $\mathcal{M}(\hat{q},\hat{x}_i;\mathcal{H},\mathbb{J})$ are non-empty for $\hat{x}_i \in \hat{X}_{(-1)}$.
\par
To prove item $(4)$, let $\hat{\gamma} \in \Per{H} \setminus \hat{X}$ and suppose that $\mathcal{M}(\hat{q},\hat{\gamma};\mathcal{H},\mathbb{J}) \neq \emptyset$. We will show that $\hat{\gamma} \in Pos^*(\hat{X})$. There are two cases: either $\hat{q}=\hat{p}_i$ for some $\hat{x}_i \in \hat{X}_{(0)}$, or else $\hat{q} \neq \hat{p}_i$ for any such $\hat{x}_i$. Let us first suppose that $\hat{q}=\hat{p}_i$ for some $\hat{x}_i \in \hat{X}_{(0)}$. In this case, Lemma \ref{lem:SingularCobordGivesStrictPositivity} immediately gives that $\hat{\gamma} \in Pos^*(\hat{X})$. If $\hat{q} \neq \hat{p}_i$ for any $\hat{x}_i \in \hat{X}_{(0)}$ then consideration of the braid cobordism between the unlinked capped braid $\hat{X} \cup \lbrace \hat{q}  \rbrace$ and the linked capped braid $\hat{X} \cup \lbrace \hat{\gamma} \rbrace$ given $\lbrace u_i \rbrace_{i=1}^k \cup \lbrace v \rbrace$ for $u_i \in \mathcal{M}(\hat{p}_i,\hat{x}_i;\mathcal{H},\mathbb{J})$, $i=1,\ldots,k$ and $v \in \mathcal{M}(\hat{q},\hat{\gamma};\mathcal{H},J)$, yields (since $\hat{X}$ is maximally unlinked relative the Morse range and so $\hat{X} \cup \lbrace \hat{\gamma} \rbrace$ is linked) that the graph of $v$ in $\R \times S^1 \times \Sigma$ must intersect the graph of some $u_i$ at least once, and any intersections may only contribute positively to the change in linking number as all the maps solve Equation \ref{sFE} for $(\mathcal{H},\mathbb{J})$. Consequently, $\hat{\gamma} \in Pos^*(\hat{X})$.
\end{proof}

\begin{proposition}\label{Prop: Continuation to dominating Morse Control}
If $(\mathcal{H}',\mathbb{J}') \in \HJ^{\hat{X}}(H,J;f,J^-)$, then $\sigma \in \ker h_{\mathcal{H}'}$ for all $\sigma \in Pos^*(\hat{X})$ such that 
\begin{align*}
-N +1 \leq \mu(\hat{\gamma}) \leq N-1, \; \forall \hat{\gamma} \in \supp \sigma.
\end{align*} 
(Here $N= \inf \lbrace k > 0: \exists A \in \pi_2(M), c_1(A)=k \rbrace$ denotes the minimal Chern number).
\end{proposition}
\begin{proof}
Remark that if $\mu(\hat{\gamma})$ satisfies the stated bounds, then the only moduli spaces that may contribute to $h_{\mathcal{H}'}$ are of the form $\mathcal{M}(\hat{\gamma},\hat{q};\mathcal{H}',\mathbb{J}')$ for $\hat{q}$ a trivially capped critical point of $f$. We claim that when $\hat{\gamma} \in Pos^*(\hat{X})$, then $\mathcal{M}(\hat{\gamma},\hat{q};\mathcal{H}',\mathbb{J}') = \emptyset$ for all such $q \in Crit(f)$. Indeed, suppose for a contradiction that there exists $u \in \mathcal{M}(\hat{\gamma},\hat{q};\mathcal{H}',\mathbb{J}')$ for some $q \in Crit(f)$. Since $\hat{\gamma} \in Pos^*(\hat{X})$, there exists some $\hat{x}_i \in \hat{X}$ such that $\ell(\hat{x}_i, \hat{\gamma}) >0$. Moreover, by hypothesis, $(\mathcal{H}',\mathbb{J}')$ is $\hat{X}$-dominating, and so there exists $v \in \mathcal{M}(\hat{x}_i,\hat{p}_i;\mathcal{H}',\mathbb{J}')$. Consequently
\begin{align*}
0< \ell(\hat{x}_i,\hat{\gamma}) &= \ell_{-\infty}(v,u) \leq \ell_{\infty}(v,u),
\end{align*}
but either $\hat{p}_i \neq \hat{q}$, in which case $\ell_{\infty}(v,u)=\ell(\hat{p}_i,\hat{q})=0$, or else $\hat{p}_i= \hat{q}$, in which case $\ell_{\infty}(v,u) \leq a(\hat{p}_i)$ and $a(\hat{p}_i) \leq 0$, since $\mu(\hat{p}_i) \in \lbrace -1,0,1 \rbrace$. In either case, $\ell_{\infty}(v,u) \leq 0$, which gives a contradiction.
\end{proof}
\begin{proposition}
\begin{enumerate}
\item Let $(\mathcal{H},\mathbb{J}) \in \HJ^{\hat{X}}(f,J^-;H,J)$, then the map 
\begin{align*}
\pi^{\hat{X}} \circ h_{\mathcal{H}} : CF_*(f,J^-) &\rightarrow CF_*(\hat{X};H,J)
\end{align*}
is a morphism of chain complexes.
\item Let $(\mathcal{H}',\mathbb{J}') \in \HJ^{\hat{X}}(H,J;f,J^+)$, then the map 
\begin{align*}
h_{\mathcal{H}'}\vert_{CF_*(\hat{X};H,J)}: CF_*(\hat{X};H,J) &\rightarrow CF_*(f,J^+)
\end{align*} 
is a morphism of chain complexes.
\end{enumerate}
\end{proposition}
\begin{proof}
As in the proof of Theorem \ref{thm:RestrictedDifferentialGivesComplex}, it suffices to prove that the maps are chain maps in the Morse range. Attending first the map $\pi^{\hat{X}} \circ h_{\mathcal{H}}$, we therefore consider $\hat{p} \in \Per{f}_{(k)}$ for $k \in \lbrace -1, 0 ,1 \rbrace$. We note that $h_{\mathcal{H}}(\hat{p}) \in \Z_2 \langle \hat{x} \rangle_{\hat{x} \in \hat{X}} \oplus Pos^*(\hat{X})$ by Proposition \ref{ContinuationControlProp} and $\partial_{H,J} Pos^*(\hat{X}) \subseteq Pos^*(\hat{X}) \subseteq \ker \pi^{\hat{X}}$ by Lemmas \ref{lem:PosNegInKerOfProj} and \ref{lem:PosInvarUnderDiff}. Consequently, we see that 
\begin{align*}
(\partial^{\hat{X}} \circ h_{\mathcal{H}})(\hat{p})&= (\partial^{\hat{X}} \circ \pi^{\hat{X}} \circ h_{\mathcal{H}})(\hat{p}) + (\partial^{\hat{X}} \circ \pi^{Pos^*(\hat{X})} \circ h_{\mathcal{H}}) (\hat{p}) \\
&=(\partial^{\hat{X}} \circ \pi^{\hat{X}} \circ h_{\mathcal{H}})(\hat{p}) + (\pi^{\hat{X}} \circ \partial_{H,J^+} \circ \pi^{Pos^*(\hat{X})} \circ h_{\mathcal{H}}) (\hat{p}) \\
&=(\partial^{\hat{X}} \circ \pi^{\hat{X}} \circ h_{\mathcal{H}})(\hat{p})
\end{align*}
Thus, since $h_{\mathcal{H}}$ is a chain map with respect to the full Floer differential, we compute
\begin{align*}
(\pi^{\hat{X}} \circ h_{\mathcal{H}})(\partial_{f,J^-} \hat{p}) = (\pi^{\hat{X}} \circ \partial_{H,J^+})(h_{\mathcal{H}}(\hat{p}))=& (\partial^{\hat{X}} \circ h_{\mathcal{H}})(\hat{p}) 
=\partial^{\hat{X}}((\pi^{\hat{X}} \circ h_{\mathcal{H}})(\hat{p})),
\end{align*}
which shows that $\pi^{\hat{X}} \circ h_{\mathcal{H}}$ is a chain map.
\par 
To see that $h_{\mathcal{H}'}\vert_{CF_*(\hat{X};H,J)}$ is a chain map, consider $\hat{x} \in \hat{X}$. It will suffice by the $\Lambda_{\omega}$-equivariance of Floer continuation maps to show that

\begin{align*}
h_{\mathcal{H}'}(\partial^{\hat{X}} \hat{x})&= (\partial_{f,J^-}\circ h_{\mathcal{H}'})(\hat{x}).
\end{align*}
Using the fact that $h_{\mathcal{H}'}$ is a chain map with respect to the usual Floer differential and that 
\begin{align*}
\partial_{H,J} \hat{x}&= \partial^{\hat{X}} \hat{x} + \sigma,
\end{align*}
for $\sigma \in Pos^*(\hat{X})$, we compute
\begin{align*}
(\partial_{f,J^-}\circ h_{\mathcal{H}'})(\hat{x})&=h_{\mathcal{H}'}(\partial_{H,J} \hat{x}) \\
&=h_{\mathcal{H}'}(\partial^{\hat{X}} \hat{x}) +h_{\mathcal{H}'}(\sigma).
\end{align*}
Note that $\sigma \in CF_k(H,J)$ for $k \in \lbrace 0,-1,-2 \rbrace$. If $k=-2$, then $h_{\mathcal{H}'}(\sigma)=0$, as $CF_{-2}(f,J^-)=0$, while if $k \in \lbrace 0,-1 \rbrace$, then since $\sigma \in Pos^*(\hat{X})$, Proposition \ref{Prop: Continuation to dominating Morse Control} implies that $h_{\mathcal{H}'}(\sigma)=0$. In either case, we see that
\begin{align*}
(\partial_{f,J^-}\circ h_{\mathcal{H}'})(\hat{x}) &= h_{\mathcal{H}'}(\partial^{\hat{X}} \hat{x}),
\end{align*}
as desired.
\end{proof}

\begin{lemma}\label{Lem:Composition of Contmaps is identity on homology}
For any $(\mathcal{H},\mathbb{J}) \in \HJ^{\hat{X}}(f,J^-;H,J)$ and any $(\mathcal{H}',\mathbb{J}') \in \HJ^{\hat{X}}(H,J;f,J^+)$,
the map $h_{\mathcal{H}'} \circ (\pi^{\hat{X}} \circ h_{\mathcal{H}})$ induces the identity map on homology. 
\end{lemma}
\begin{proof}
Once more, it suffices to show that $h_{\mathcal{H}'}  \circ (\pi^{\hat{X}} \circ h_{\mathcal{H}})$ induces an isomorphism on homology in degrees lying in the Morse range. Since, for any $\hat{p} \in CF_k(f,J^-)$, with $k \in \lbrace -1,0,1 \rbrace$, we have that $h_{\mathcal{H}}(\hat{p}) \in \Z_2 \langle \hat{x} \rangle_{\hat{x} \in \hat{X}} \oplus Pos^*(\hat{X})$, and $Pos^*(\hat{X}) \subseteq \ker \pi^{\hat{X}} \cap \ker h_{\mathcal{H}'}$, so we compute 
\begin{align*}
(h_{\mathcal{H}'} \circ h_{\mathcal{H}})(\hat{p})&= h_{\mathcal{H}'}((\pi^{\hat{X}} \circ h_{\mathcal{H}})(\hat{p}) + (\pi^{Pos^*(\hat{X})} \circ h_{\mathcal{H}})(\hat{p})) \\
&=(h_{\mathcal{H}'} \circ \pi^{\hat{X}} \circ h_{\mathcal{H}})(\hat{p}) \\ 
&= h_{\mathcal{H}'} \vert_{CF_*(\hat{X};H,J^+)} \circ (\pi^{\hat{X}} \circ h_{\mathcal{H}})(\hat{p}).
\end{align*}
But it is a standard fact in Floer theory that $h_{\mathcal{H}'} \circ h_{\mathcal{H}}$ induces the identity map on homology, and so it must be that case that the composition
\begin{align*}
HF_k(f) \xrightarrow{(\pi^{\hat{X}} \circ h_{\mathcal{H}})_*} & HF_k(\hat{X};H,J) \xrightarrow{ (h_{\mathcal{H}'})_*} HF_k(f)
\end{align*}
is the identity map for $k \in \lbrace -1,0,1 \rbrace$.
\end{proof}
We will use the above fact to bootstrap ourselves into a much finer-grained understanding of the structure of $CF_*(\hat{X};H,J)$ in the coming section.

\subsection{Construction and properties of $\mathcal{F}^{\hat{X}}$}\label{ConstructingFoliationsSection}
The purpose of this section is to prove the existence of the foliation in the following theorem. 
\begin{theorem}\label{Thm: Foliation thm}
Let $H \in C^\infty(S^1 \times \Sigma)$ be a non-degenerate Hamiltonian, and let $J \in C^\infty(S^1; \mathcal{J}_\omega(\Sigma))$ be such that $(H,J)$ is Floer regular. For any capped braid $\hat{X} \in murm(H)$, we may construct an oriented singular foliation $\mathcal{F}^{\hat{X}}$ of $S^1 \times \Sigma$ with the following properties
\begin{enumerate}
\item The singular leaves of $\mathcal{F}^{\hat{X}}$ are precisely the graphs of the orbits in $\hat{X}$.
\item The regular leaves are precisely the annuli parametrized by maps
\begin{align*}
\check{u}: \R \times S^1 &\rightarrow S^1 \times \Sigma \\
(s,t) &\mapsto (t,u(s,t)).
\end{align*}
for some $u \in \widetilde{\mathcal{M}}(\hat{x},\hat{y};H,J)$, and some $\hat{x}, \hat{y} \in \hat{X}$.
\item The vector field $\check{X}^H(t,z)=\partial_t \oplus X^H_t(z)$ is \textit{positively transverse} to every regular leaf of $\mathcal{F}^{\hat{X}}$.
\end{enumerate}
\end{theorem} 
The positive transversality property will be established in Section \ref{PositivelyTransverseSubsection}. Our construction of the singular foliation in Theorem \ref{Thm: Foliation thm} proceeds by establishing that a generic point  in $S^1 \times \Sigma$ lies inside the foliated sector $W(\hat{x},\hat{y})$ for some $\hat{x} \in \hat{X}_{(1)},\hat{y} \in \hat{X}_{(-1)}$. The remaining points lie in the closure of these sectors, and so lie either on leaves parametrized by broken cylinders, or the graphs of orbits in $X$. To establish existence of the requisite leaves, we make use of the cap action of a point on the Floer complex, first introduced in detail in \cite{LeOh96} (see also \cite{PSS96}). For simplicity, we restrict our discussion of the cap action to the case where $(M^{2n},\omega)$ is a spherically monotone symplectic manifold with minimal Chern number $N \geq n +1$, that is
\begin{align*}
 c_1\vert_{\pi_2(M)} &= \lambda \cdot \omega\vert_{\pi_2(M)},\; \text{for some } \lambda >0, \; \text{and} 
\\
N&= \inf \lbrace k > 0: \exists A \in \pi_2(M), c_1(A)=k \rbrace \geq n +1.
\end{align*} 
This restriction allows us to avoid having to concern ourselves with bubbling phenomena which can arise in the compactness arguments involving the moduli spaces which occur when defining the cap product with a cycle (more precisely, it ensures that condition $(iii)$ in Proposition 3.2 of \cite{LeOh96} is automatically satisfied). Since every surface is spherically monotone, this restriction is harmless for our purposes.
\\ \par 
Given a singular homology class $\alpha \in H_k(M;\Z_2)$, we represent $\alpha$ by a smooth cycle $\alpha^{\#}: \cup \Delta^k \rightarrow M$, then for any Floer regular $(H,J)$ and any $t \in S^1$, we may consider, for any $\hat{x}, \hat{y} \in \Per{H}$, the moduli space
\begin{align*}
\mathcal{M}^{\alpha^{\#},t}(\hat{x},\hat{y};H,J) &:= \lbrace u \in \widetilde{\mathcal{M}}(\hat{x},\hat{y};H,J): u(0,t) \in \im \alpha^{\#} \rbrace
\end{align*}
with expected dimension $\mu(\hat{x}) - \mu(\hat{y}) - (2n-k)$. For $t \in S^1$, we will say that the smooth chain $\alpha^\#$ is \textbf{$(H,J;t)$-generic} if the evaluation map $ev_t(u,q):=(u(0,t),\alpha^{\#}(q)) \in M \times M$, $(u,q) \in \widetilde{\mathcal{M}}(\hat{x},\hat{y};H,J)  \times \cup \Delta^k$, is transversal to the diagonal whenever $\mu(\hat{x})- \mu(\hat{y}) \leq (2n-k)+1$. Note that if $\alpha^\#$ is any smooth cycle, then by a standard argument using the Sard-Smale theorem (cf. Lemma 6.5.5 of \cite{McSa12}) there is a residual set of $f \in Diff_0(M)$ such that $f\circ \alpha^\#$ is $(H,J;t)$-generic. Whenever $\alpha^\#$ is $(H,J;t)$-generic we may define the \textbf{cap product} of $\alpha$ on $HF_*(H)$ (at time $t$) at the chain level by defining, for $\hat{x} \in \Per{H}$,
\begin{align*}
\alpha^\# \cap_t \hat{x} &:= \sum_{\substack{\hat{y} \in \Per{H}: \\
\mu(\hat{x})-\mu(\hat{y})=2n -k}} n^{\alpha^{\#},t}(\hat{x},\hat{y}) \hat{y},
\end{align*}
where $n^{\alpha^{\#},t}(\hat{x},\hat{y})$ is the mod $2$ count of the number of elements in $\mathcal{M}^{\alpha^{\#},t}(\hat{x},\hat{y};H,J)$. The cap action descends to homology, and is independent at the homology level of all choices. Moreover, for generic adapted homotopies of Floer data, the cap action commutes with continuation maps \textit{at the chain level}. That is, for generic $(\mathcal{H},J)$ we have
\begin{align}\label{reln:CapCommutesWithCont}
h_{\mathcal{H}}(\alpha^{\#} \cap_t \hat{x}) &= \alpha^{\#} \cap_t  h_{\mathcal{H}}(\hat{x}),
\end{align}
whenever the Floer pairs $(H^\pm,J^\pm)$ at the ends of the homotopy are such that the relevant moduli spaces are transversal. It follows from the above that, under the identification of $HF_*(H)$ with $QH_{*+n}(M)$, the cap action on $HF_*(H)$ is identified with the obvious action of the homology of $M$ on its quantum homology via the quantum product (see \cite{PSS96} for a proof of this identification).
\par 
A rather important point for us will be that the cap action interacts nicely with respect to the chain maps $\pi^{\hat{X}} \circ h_{\mathcal{H}}$ and $h_{\mathcal{H}'}\vert_{CF_*(\hat{X};H,J)}$ introduced in the previous section.
\begin{proposition}\label{prop:CapInteractionWithContinuation}
Let $\hat{X} \in murm(H)$, $(\mathcal{H}, \mathbb{J}) \in \HJ^{\hat{X}}(f,J^-;H,J)$, $(\mathcal{H}',\mathbb{J} \in \HJ^{\hat{X}}(H,J;f,J^+)$ and suppose that $\alpha^{\#}$ represents $\alpha$ as above and $\alpha^{\#}$ is both $(\mathcal{H},\mathbb{J};t)$-generic and $(\mathcal{H}',\mathbb{J}';t)$-generic, then 
\begin{align}
(\pi^{\hat{X}} \circ h_{\mathcal{H}})( \alpha^{\#} \cap_t \hat{p}) &= \pi^{\hat{X}}(\alpha^{\#} \cap_t (\pi^{\hat{X}} \circ h_{\mathcal{H}})(\hat{p})), \; \forall \hat{p} \in \Per{f} \; and \label{Eqn: CapIntersactionWithContEqn1} \\
 (h_{\mathcal{H}'}\circ \pi^{\hat{X}}) ( \alpha^{\#} \cap_t \pi^{\hat{X}}(\hat{y})) &= \alpha^{\#} \cap_t (h_{\mathcal{H}'}\circ \pi^{\hat{X}})(\hat{y}), \; \forall \hat{y} \in \Per{H}. \label{Eqn: CapIntersactionWithContEqn2}
\end{align}
\end{proposition}
\begin{proof}
A straightforward computation shows that the map induced by capping with a smooth cycle $\alpha^{\#}$ at time $t$ is $\Lambda_{\omega}$-equivariant at the chain level, in the sense that for any $A \in \Gamma_{\omega}$
\begin{align*}
\alpha^{\#} \cap_t (e^A \cdot \hat{x})&= e^A \cdot( \alpha^{\#} \cap_t \hat{x}).
\end{align*}
Since the maps $\pi^{\hat{X}}$ and $h_{\mathcal{H}}$ are also $\Lambda_{\omega}$-equivariant and $f$ is by hypothesis $C^2$-small, so that
\begin{align*}
CF_*(f,J^-) &= Crit(f) \otimes \Lambda_{\omega},
\end{align*} 
as $\Lambda_{\omega}$-modules, we may reason similarly as in the proof of Theorem \ref{thm:RestrictedDifferentialGivesComplex}, and reduce to the case where $\mu(\hat{p}) \in \lbrace -1,0,1 \rbrace$ (since any $\hat{q} \in \Per{f}$ may be written as $e^A \cdot \hat{p}$ for some $\hat{p}$ with Conley-Zehnder index in the Morse range, and so if the desired relations hold in the Morse range, then by $\Lambda_{\omega}$-equivariance of the involved maps, the desired relations hold for all $\hat{q} \in \Per{f}$). Thus, we  may, and do, suppose going forward that $\mu(\hat{p}) \in \lbrace -1, 0 ,1 \rbrace$.
\par
Note that Proposition \ref{ContinuationControlProp} implies that we may write $h_{\mathcal{H}}(\hat{p})= \sigma + \beta$ for $\sigma \in CF_*(\hat{X};H,J)$ and $\beta \in Pos^*(\hat{X})$, so that the right-hand side of $(\ref{Eqn: CapIntersactionWithContEqn1})$ may be computed as
\begin{align*}
\pi^{\hat{X}}(\alpha^{\#} \cap_t (\pi^{\hat{X}} \circ h_{\mathcal{H}})(\hat{p})) &= \pi^{\hat{X}}(\alpha^{\#} \cap_t \pi^{\hat{X}}(\sigma + \beta))  \\
&=\pi^{\hat{X}}(\alpha^{\#} \cap_t \sigma).
\end{align*}
In order to establish that this agrees with the quantity on the left-hand side of $(\ref{Eqn: CapIntersactionWithContEqn1})$, the central point is to remark that capping with $\alpha^\#$ preserves $Pos^*(\hat{X})$, because linking in the Floer complex is non-decreasing along Floer cylinders. Indeed, if $\hat{y} \in Pos^*(\hat{X})$ and there exists $u \in \mathcal{M}^{\alpha^\#,t}(\hat{y},\hat{y}';H,J)$ for some $\hat{y}' \in \Per{H}$, then for each $\hat{x} \in \hat{X}$, Lemma \ref{EndsLinkLemma} with $v(s,t)=x(t)$ implies that $\ell(\hat{y},\hat{x}) \leq \ell(\hat{y}',\hat{x})$. Thus $\alpha^{\#} \cap_t \beta \in Pos^*(\hat{X})$ and the left-hand side of $(\ref{Eqn: CapIntersactionWithContEqn1})$ may be computed as
\begin{align*}
(\pi^{\hat{X}} \circ h_{\mathcal{H}})(\alpha^{\#} \cap_t \hat{p})= &\pi^{\hat{X}}(\alpha^{\#} \cap_t \sigma + \alpha^{\#} \cap_t \beta) = \pi^{\hat{X}}(\alpha^{\#} \cap_t \sigma),
\end{align*}
where we use Lemma \ref{lem:PosNegInKerOfProj} in the last equality. This establishes Equation $(\ref{Eqn: CapIntersactionWithContEqn1})$. Equation $(\ref{Eqn: CapIntersactionWithContEqn2})$ is proved similarly, needing only the additional remark that for $\hat{x} \in \hat{X}$, $\alpha^{\#} \cap_t \hat{x} \in \Z_2 \langle \hat{x} \rangle_{\hat{x} \in \hat{X}} \oplus Pos^*(\hat{X})$, which follows by the same reasoning as above, using that $b(\hat{x}) \geq 0$ when $\mu(\hat{x})$ lies in the Morse range.
\end{proof}

The following proposition is essentially tautological.
\begin{proposition}
Let $(H,J)$ be Floer regular, $t \in S^1$, and suppose that $p \in \Sigma$ is $(H,J;t)$-generic (for the point class in homology). Then $p \in W(\hat{x},\hat{y})$ implies that $\mu(\hat{x}) - \mu(\hat{y}) \geq 2$, and $\hat{y} \in \supp (p \cap_t \hat{x})$ if and only if $(t,p) \in W(\hat{x},\hat{y})$.
\end{proposition}
Combining the above with Corollary \ref{cor:FoliatedSectorCorollary}, allows us to conclude:
\begin{cor}
Suppose that $\mu(\hat{x})=2k+1$ for $k \in \Z$ and $p$ is $(H,J;t)$-generic, then $\hat{y} \in \supp (p \cap_t \hat{x})$ if and only if there exists an open neighbourhood of $(t,p) \in S^1 \times \Sigma$ which is foliated by leaves of $\mathcal{F}^{\hat{x},\hat{y}}$. 
\end{cor}
Recall from Section \ref{FoliatedSectorsSubsection} that $W(\hat{x},\hat{y})= \lbrace (t,u(s,t)) \in S^1 \times \Sigma: u \in \widetilde{\mathcal{M}}(\hat{x},\hat{y};H,J) \rbrace$, and write $\mathcal{W}(\hat{X})$ for the union of all $W(\hat{x},\hat{y})$ where $\hat{x} \in \hat{X}_{(1)}$, $\hat{y} \in \hat{X}_{(-1)}$.
\begin{lemma}\label{lem:DenseLeaves}
Let $(H,J)$ be Floer regular and $\hat{X} \in murm(H)$, then $\mathcal{W}(\hat{X})$ is open and dense in $S^1 \times \Sigma$.
\end{lemma}
\begin{proof}
Let $(\mathcal{H},\mathbb{J}) \in \HJ^{\hat{X}}(f,J^-;H,J)$. We fix $t \in S^1$ arbitrarily and let $p \in \Sigma$ be $(\mathcal{H},\mathbb{J};t)$-generic. We let $\sigma \in CF_1(f,J^-)$ represent the fundamental class $[\Sigma] \in QH_2(\Sigma) \simeq HF_1(f)$, and we note that we must have $[p \cap_t \sigma]= [pt] \in QH_0(\Sigma)$,
which is in particular not $0$. But Lemma \ref{Lem:Composition of Contmaps is identity on homology} implies that $(\pi^{\hat{X}} \circ h_{\mathcal{H}})_*$ is injective on homology, and so 
\begin{align*}
0 \neq (\pi^{\hat{X}} \circ h_{\mathcal{H}})(p \cap_t \sigma)&= \pi^{\hat{X}}(p \cap_t (\pi^{\hat{X}} \circ h_{\mathcal{H}})(\sigma)),
\end{align*} 
and this implies that for every such generic $p$, there must exist some $\hat{x} \in \supp (\pi^{\hat{X}} \circ h_{\mathcal{H}})(\sigma) \subseteq \hat{X}_{(1)}$ and some $\hat{y} \in \hat{X}_{(-1)}$ such that $\hat{y} \in \supp p \cap_t \hat{x}$. Hence for every $t \in S^1$, there is a generic set of $p \in \Sigma$ such that $p$ is $(\mathcal{H},\mathbb{J};t)$-generic and so $(t,p)$ lies inside the open $3$-dimensional connecting submanifold $W(\hat{x},\hat{y})$ for some such $\hat{x}, \hat{y} \in \hat{X}$, which proves the lemma.
\end{proof}

\begin{lemma}\label{lem:UnlinkedPlusConnectingStillUnlinked}
Let $(H,J)$ be Floer regular, and $\hat{X} \subseteq \Per{H}$ any unlinked capped braid. Then for any $\hat{\gamma} \in \Per{H}$ such that $\mathcal{M}(\hat{x}^+,\hat{\gamma};H,J) \times \mathcal{M}(\hat{\gamma},\hat{x}^-;H,J) \neq \emptyset$ for some $\hat{x}^+ \in \hat{X}_{(1)}$, $\hat{x}^- \in \hat{X}_{(-1)}$, $\hat{X}$ and $\hat{\gamma}$ are unlinked.
\end{lemma}
\begin{proof}
Suppose that $\hat{X} \cup \lbrace \hat{\gamma} \rbrace$ is linked, then Lemma \ref{lem:SingularCobordGivesStrictPositivity} implies that $\hat{\gamma} \in Pos^*(\hat{X})$, and so $\ell(\hat{\gamma},\hat{x})>0$ for some $\hat{x} \in \hat{X}$, but then Proposition \ref{prop:LinkingMonotoneAlongFloerDifferential} implies that $\ell(\hat{x}^-,\hat{x}) >0$, which contradicts the assumption that $\hat{X}$ is unlinked.
\end{proof}
Inductively applying Lemma \ref{lem:UnlinkedPlusConnectingStillUnlinked} yields the following corollary.

\begin{cor}\label{cor:OrbitsOnBrokenTrajectoriesAreUnlinked}
Suppose that $(H,J)$ is Floer regular, and let $\hat{X} \subseteq \Per{H}$ be such that $\hat{X} $ is unlinked and $\hat{X}=\hat{X}_{(1)} \cup \hat{X}_{(-1)}$, then $\hat{X}$ and $\hat{\Upsilon}$
are unlinked, where $\hat{\Upsilon}$ is the capped braid consisting of all $\hat{\gamma} \in \Per{H}$ with $\mu(\hat{\gamma})=0$ such that $\mathcal{M}(\hat{x}^+,\hat{\gamma};H,J) \times \mathcal{M}(\hat{\gamma},\hat{x}^-;H,J) \neq \emptyset$, $\hat{x}^\pm \in \hat{X}$.
\end{cor}

Finally, we are ready to prove the existence of the advertised foliation. We write
\begin{align*}
\mathcal{M}^{H,J}(\hat{X})&:= \bigcup_{\hat{x},\hat{y} \in \hat{X}} \mathcal{M}(\hat{x},\hat{y};H,J).
\end{align*}
Note that we include the case where $\hat{x}=\hat{y}$ in the above union. In such a case $[u] \in \mathcal{M}(\hat{x},\hat{x};H,J)$ may simply be identified with the loop $x$.
\begin{theorem}[Existence part of Theorem \ref{Thm: Foliation thm}]\label{FoliationTheorem}
Let $(H,J)$ be a non-degenerate Floer pair, and $\hat{X} \in murm(H)$, then the collection of submanifolds $\mathcal{F}^{\hat{X}} := \cup_{[u] \in \mathcal{M}^{H,J}(\hat{X})} \lbrace \im \check{u} \rbrace$ forms a Stefan-Sussmann foliation of $S^1 \times \Sigma$.
\end{theorem}
\begin{proof}
We adapt a strategy used in \cite{HWZ03} that shows that the foliation $\widetilde{\mathcal{F}}^{\hat{X}}$ with leaves given by the graphs $\tilde{u}$ of all the $u \in \widetilde{\mathcal{M}}(\hat{x},\hat{y};H,J)$, $\hat{x},\hat{y} \in \hat{X}$, is a smooth $2$-dimensional foliation of $\R \times S^1 \times \Sigma$, from which it follows immediately that $\mathcal{F}^{\hat{X}}$ is a Steffan-Sussmann foliation. Indeed, in this event, $\mathcal{F}^{\hat{X}}$ integrates the distribution $\mathcal{D}^{\hat{X}}=\check{\pi}_* \widetilde{\mathcal{D}}^{\hat{X}}$, where $\widetilde{\mathcal{D}}^{\hat{X}}$ is the distribution integrated by $\widetilde{\mathcal{F}}^{\hat{X}}$, and this realizes $\mathcal{D}^{\hat{X}}$ in a way that is manifestly smooth in the sense of generalized distributions (see Definition \ref{def:SmoothGenDistrib}).
\par
To see that $\widetilde{\mathcal{F}}^{\hat{X}}$ is a smooth foliation, note first that none of the leaves of $\widetilde{\mathcal{F}}^{\hat{X}}$ may intersect. Indeed, the leaves of $\widetilde{\mathcal{F}}^{\hat{X}}$ are graphs of $(H,J)$ Floer cylinders between elements of $\hat{X}$, so if $u_i \in \widetilde{\mathcal{M}}(\hat{x}^-_i,\hat{x}^+_i;H,J)$, $i=0,1$ are distinct such Floer cylinders with $\hat{x}^\pm_i \in \hat{X}$, then since $\hat{X} \in murm(H)$, we have that $a(\hat{x}) = 0 \leq \ell_{-\infty}(u_0,u_1)$ if $\hat{x}=\hat{x}_0^- = \hat{x}_1^+$, and $\ell(\hat{x}_0^-,\hat{x}_1^-)= 0 \leq \ell_{-\infty}(u_0,u_1)$ if $\hat{x}_0^- \neq \hat{x}^-_1$. In either case, $0 \leq \ell_{-\infty}(u_0,u_1)$. Similarly, $\ell_{\infty}(u_0,u_1) \leq 0$, which implies that $\ell_{\hat{u}_0,\hat{u}_1}(s) =0$ for all $s \in \R$ by Lemma \ref{EndsLinkLemma}, and we thereby conclude that $\tilde{u}_0$ and $\tilde{u}_1$ are disjoint. Next, we note that by Lemma \ref{lem:DenseLeaves}, the set $\mathcal{W}(\hat{X})$ is open and dense in $S^1 \times \Sigma$. This implies, by the $\R$-invariance of solutions to Equation \ref{FE}, that the set of points $\widetilde{W}(\hat{X})$ lying on the graph $\tilde{u}$ of some $u \in \widetilde{\mathcal{M}}(\hat{x},\hat{y};H,J)$, $\hat{x} \in \hat{X}_{(1)}, \hat{y} \in \hat{X}_{(-1)}$ is open and dense in $\R \times S^1 \times \Sigma$. Consequently, the partition $\bigcup \widetilde{\mathcal{F}}^{\hat{x},\hat{y}}$, where the union runs over all $\hat{x} \in \hat{X}_{(1)}$, $\hat{y} \in \hat{X}_{(-1)}$, gives a smooth foliation of an open, dense set of $\R \times S^1 \times \Sigma$. Consequently we may argue just as in \cite{HWZ03} in the paragraphs following the proof of lemma $6.10$ (p. 231-232); all of the remaining leaves in $\widetilde{\mathcal{F}}^{\hat{X}}$ are graphs of constant orbits or of cylinders $u$ which connect orbits of index difference equal to $1$. In either case, by standard compactness theorems of Floer theory those graphs which form the leaves of the foliation of $\widetilde{\mathcal{W}}(\hat{X})$ converge modulo reparametrization in the $C^\infty_{loc}$-topology either to the graphs $(s,t) \mapsto (s,t,x(t))$ of the orbits $x$ for $\hat{x} \in \hat{X}$, or to graphs of cylinders connecting orbits of index difference $1$, which come in pairs $(u,v) \in \widetilde{\mathcal{M}}(\hat{x},\hat{\gamma};H,J) \times \widetilde{\mathcal{M}}(\hat{\gamma},\hat{y};H,J)$, $\hat{x} \in \hat{X}_{(1)}, \hat{\gamma} \in \hat{X}_{(0)}, \hat{y} \in \hat{X}_{(-1)}$. By Corollary \ref{cor:OrbitsOnBrokenTrajectoriesAreUnlinked} and Lemma \ref{lem:UnlinkedPlusConnectingStillUnlinked}, the capped braid formed by the collection of all the $\hat{\gamma} \in \Per{H}$ on which such pairs break are unlinked with $\hat{X}$, and so lie in $\hat{X}$ by maximality. Consequently, the graphs of such broken trajectories cannot intersect, nor can they intersect any leaf of $\widetilde{\mathcal{F}}^{\hat{X}}$ in the dense set $\widetilde{\mathcal{W}}(\hat{X})$. Since every point in $\R \times S^1 \times \Sigma$ lies in the closure of $\widetilde{\mathcal{W}}(\hat{X})$, every such point much lie on the graph of an orbit in $X$ or the graph of such a broken cylinder. It follows that the union of all the leaves in $\widetilde{\mathcal{F}}^{\hat{X}}$ thus fits together into a smooth foliation on all of $\R \times S^1 \times \Sigma$, and so the theorem follows.
\end{proof}
\subsubsection{$\mathcal{F}^{\hat{X}}$ as negative gradient flow-lines of the restricted action functional}\label{NegFlowLinesSection}
For a Floer regular pair $(H,J)$ and $\hat{X} \in murm(H)$, denote 
\begin{align*}
\mathfrak{M}_{\hat{X}}&:= \lbrace \hat{\alpha} \in \cL{\Sigma}: \exists \hat{x}, \hat{y} \in \hat{X},  \exists u \in \widetilde{\mathcal{M}}(\hat{x},\hat{y};H,J), \; \text{such that } \hat{u}_s=\hat{\alpha}, \text{for some } s \in \R \rbrace
\end{align*}
equipped with the subspace topology from $\cL{\Sigma}$. Remark that in the definition of $\mathfrak{M}_{\hat{X}}$ we do not insist that $\hat{x}$ and $\hat{y}$ are distinct, so that $\hat{X} \subset \mathfrak{M}_{\hat{X}}$, since the trivial cylinders $u_{\hat{x}}(s,t)=x(t)$ always lie in $\widetilde{\mathcal{M}}(\hat{x},\hat{x};H,J)$ for every $\hat{x} \in \widetilde{Per}_0(H)$.

\begin{proposition}\label{SmoothDiffeoLoopSpaceProp}
The map $Ev: S^1 \times \mathfrak{M}_{\hat{X}} \rightarrow S^1 \times \Sigma$, given by $Ev(t,\hat{\alpha})=(t,\alpha(t))$ is a diffeomorphism.
\end{proposition}
\begin{proof}
The generalized distribution $\mathcal{D}^{\hat{X}}_{(t,u(s,t))} = \langle \partial_s u , \partial_t \oplus \partial_t u \rangle$
which is integrated by $\mathcal{F}^{\hat{X}}$ contains the one-dimensional distribution
$\mathcal{D}^{\mathfrak{M}}_{(t,u(s,t))}= \langle \partial_t \oplus \partial_t u \rangle$,
which is easily seen to be smooth near the singular fibers by employing the local model for leaves of a Stefan-Sussmann foliation of Section \ref{SingFolSection}. Consequently, $\mathcal{D}^{\mathfrak{M}}$ is a smooth foliation which integrates precisely to the graphs of the maps $\alpha: S^1 \rightarrow \Sigma$ for $\hat{\alpha} \in \mathfrak{M}_{\hat{X}}$. This is obviously equivalent to the proposition.
\end{proof}
\begin{definition}
For $\hat{X} \in murm(H)$, define \textbf{the ($\hat{X}$-)restricted action functional} $A^{\hat{X}} \in C^\infty(S^1 \times \Sigma)$ by $A^{\hat{X}}:= \mathcal{A}_H \circ Ev^{-1}$. Additionally, for each $t \in S^1$, we define $A^{\hat{X}}_t := \iota_t^* A^{\hat{X}}$, where $\iota_t: \Sigma \hookrightarrow S^1 \times \Sigma$ is the inclusion of the fiber over $t \in S^1$.
\end{definition}
Note that each $A^{\hat{X}}_t$ is automatically Morse, since the Hessian of $A^{\hat{X}}_t$ at $x(t)$ for $\hat{x} \in \hat{X}$ obviously inherits the non-degeneracy of the Hessian of $\mathcal{A}_H$ at $\hat{x}$. In fact, our construction clearly identifies Floer trajectories connecting orbits in $\hat{X}$ with negative gradient flow lines of the $A^{\hat{X}}_t$, giving us Morse models for the foliation $\mathcal{F}^{\hat{X}}$.
\begin{proposition}\label{IdentificationOfModuliSpaces}
If $(H,J)$ is Floer regular, $\hat{X} \in murm(H)$ and $\epsilon >0$ is sufficiently small, then for every $t \in S^1$, and every $\hat{x},\hat{y} \in \hat{X}$, there is a natural identification $\widetilde{\mathcal{M}}(\hat{x},\hat{y};H,J) \cong \widetilde{\mathcal{M}}(\hat{x},\hat{y};\epsilon A^{\hat{X}}_t,J_t)$ given by $u(s,t) \mapsto u(\epsilon s, t)$.
\end{proposition}

\begin{cor}\label{cor:MorseModels}
Let $(H,J)$ be Floer regular and $\hat{X} \in murm(H)$. Then for every $t \in S^1$, and any $\epsilon >0$ sufficiently small, \begin{align*}
CF_*(\hat{X};H,J)\cong (C^{Morse}(A^{\hat{X}}_t,g_{J_t}) &\otimes \Lambda_\omega)_{*+1} \cong CF_*(\epsilon A^{\hat{X}},J_t).
\end{align*}
\end{cor}

\subsubsection{$\mathcal{F}^{\hat{X}}$ as a positively transverse foliation}\label{PositivelyTransverseSubsection}
Each regular leaf of the foliation $\mathcal{F}^{\hat{X}}$ arises naturally as the image of an embedding $\check{u}: \R \times S^1 \hookrightarrow S^1 \times \Sigma$ for $u$, the standard orientation on the cylinder induces the orientation $\partial_s \check{u} \wedge \partial_t \check{u}$ on each regular leaf, so we may view $\mathcal{F}^{\hat{X}}$ in a natural way as an oriented singular foliation.
\begin{definition}
Let $\mathcal{F}$ be an oriented codimension $1$ Steffan-Sussmann foliation of an oriented $d$-dimensional manifold $(M^d,o_M)$. We will say that a smooth path $\alpha: [0,1] \rightarrow M$ is \textbf{positively transverse} to $\mathcal{F}$ if the following dichotomy holds, either
\begin{enumerate}
\item $\alpha$ is contained in a singular leaf of $\mathcal{F}$, or
\item for every $t \in [0,1]$, $\lbrace (\partial_t \alpha)_t,v_1, \ldots, v_{d-1} \rbrace$ is an oriented basis for $(T_{\alpha(t)} M,o_M)$, where $\lbrace v_1, \ldots, v_{d-1} \rbrace$ is an oriented basis for the tangent space of the regular leaf of $\mathcal{F}$ passing through $\alpha(t)$. 
\end{enumerate}
\end{definition}

\begin{definition}
Let $\mathcal{F}$ be an oriented codimension $1$ Steffan-Sussmann foliation on an oriented $d$-dimensional manifold $(M^d,o_M)$ and let $X \in \mathcal{X}(M)$ be a vector field generating an isotopy $(\phi_t^X)_{t \in \R}$. We say that $X$ (or $(\phi_t^X)_{t \in \R}$) is \textbf{positively transverse to $\mathcal{F}$} if every integral curve of $X$ is positively transverse to $\mathcal{F}$.
\end{definition}

\begin{proposition}
Let $(H,J)$ be Floer regular, $\hat{X} \in murm(H)$, and $\check{X}_H:= \partial_t \oplus X_H \in \mathcal{X}(S^1 \times \Sigma)$, then $\check{X}_H$ is positively transverse to $\mathcal{F}^{\hat{X}}$.
\end{proposition}
\begin{proof}
As the singular leaves of $\mathcal{F}^{\hat{X}}$ are orbits of $\check{X}_H$, it suffices to consider points $(t,p) \in S^1 \times \Sigma$ lying on regular leaves. In such a case, since $u$ solves Equation \ref{FE}, the basis formed by $\lbrace \check{X}_H, \partial_s \check{u},  \partial_t \check{u} \rbrace$ is easily seen to be orientation-equivalent to the basis $\lbrace \partial_t, \partial_s u, J_t \partial_s u \rbrace$, which is a positively oriented basis, as $J_t \in \mathcal{J}(\Sigma,\omega)$ for all $t \in S^1$.
\end{proof}
The previous proposition tells us that to any non-degenerate Hamiltonian $H$ and each $\hat{X} \in murm(H)$, we may associate a foliation on $S^1 \times \Sigma$ with respect to which the graph of the isotopy is well-behaved in a certain sense. However, if we're willing to modify the isotopy by a contractible loop, then we can in fact do better and obtain a positively transverse singular foliation on $\Sigma$ itself. 
\par 
To see this, consider the distribution $\mathcal{D}^{\mathfrak{M}}_{(t,u(s,t))} = \langle \partial_t \oplus \partial_t u \rangle$ introduced in the proof of Proposition \ref{SmoothDiffeoLoopSpaceProp}. As noted therein, $\mathcal{D}^{\mathfrak{M}}$ integrates to a smooth $1$-dimensional foliation by the graphs of the loops $t \mapsto u_s(t)$ for $\hat{u}_s \in \mathfrak{M}_{\hat{X}}$. This induces a natural loop of diffeomorphisms $(\psi^{\hat{X}}_t)_{t \in S^1}$ given by sliding the fiber $\lbrace 0 \rbrace \times \Sigma$ along the foliation which integrates $\mathcal{D}^{\mathfrak{M}}$. In other words, we have the isotopy $\psi^{\hat{X}}_t(p)=u_p(s,t)$, $t \in S^1$, where $u_p \in \widetilde{\mathcal{M}}(\hat{x},\hat{y};H,J)$, $\hat{x},\hat{y} \in \hat{X}$, is any Floer cylinder such that $u_p(s,0)=p$. It follows from Corollary \ref{cor:FoliatedSectorCorollary} and the fact that if $\hat{x}=\hat{y}$ then $u(s,t)=x(t)$ that $\psi^{\hat{X}}$ is well-defined.

\begin{proposition}\label{prop:PropsOfPsi}
$\psi:=(\psi^{\hat{X}}_t)_{t \in S^1}$ is a contractible loop.
\end{proposition}
\begin{proof}
$\psi$ defines a loop of diffeomorphisms based at the identity by construction. If the genus of $\Sigma$ is strictly greater than $1$, then the proposition follows from the fact that $Diff_0(\Sigma)$ is contractible (\cite{EE69} p. 21). 
\par
If $\Sigma= \T^2$, $Diff_0(\Sigma)$ has as a strong deformation retract the group of diffeomorphisms given by the action of the torus on itself by translation (\cite{EE69} p. 38, see also \cite{Pol12} p.52-53). In particular, contractible loops in $Diff_0(\T^2)$ are precisely the loops having contractible orbits. Note that for any $\hat{x}=[x,w] \in \hat{X}$, $\psi^{\hat{X}}_t(x(0))=x(t)$, and so $\psi^{\hat{X}}$ has contractible periodic orbits, and hence is contractible if $\Sigma=\T^2$. 
\par
In the case that $\Sigma=S^2$ we have that $Diff(S^2) \simeq SO(3)$ (\cite{EE69} p. 21), and so $\pi_1(Diff(S^2))=\pi_1(SO(3)) = \Z/2$. It is straightforward to check that if $(f_t)_{t \in S^1}$ is a loop based at the identity in $Diff(S^2)$, then the homotopy class of $(f_t)_{t \in S^1}$ in $\pi_1(Diff(S^2))$ is classified by the parity of the winding number of the loop
\begin{align*}
v(t)&:= ((\Phi_{\hat{z}}^t)^{-1} \circ Df_t \circ \Phi_{\hat{z}}^t)(v_0), \; t \in S^1
\end{align*}
for any $v_0 \in \R^2 \setminus \lbrace 0 \rbrace$, where $\hat{z} \in \cL{S^2}$ is any capped $1$-periodic orbit of $(f_t)_{t \in S^1}$ and 
\begin{align*}
\Phi_{\hat{z}}: S^1 \times \R^2 &\rightarrow z^*TS^2
\end{align*} 
is any trivialization of the tangent bundle along $z$ which extends over the chosen capping disk (here is a sketch of an argument: the statement is immediate if we are considering a loop in $SO(3)$, viewed as the subgroup of rotations of $S^2$. In order to reduce to this case, note that $Diff(S^2)$ deformation retracts onto $SO(3)$, via a deformation retraction $H: [0,1] \times Diff(S^2) \rightarrow Diff(S^2)$. If $(f_t)$ is a loop in $Diff(S^2)$, it is contractible if and only if $(H(1,f_t))$ is contractible in $SO(3)$. Note that the cylinder $Ev_z(H(s,f_t))$, $s,t \in [0,1]$, provides a correspondence between capped periodic orbits of $(f_t)$ and $(H(1,f_t))$, where $Ev_z:Diff(S^2) \rightarrow S^2$, and that the winding numbers of the corresponding capped orbits have the same parity). Since the capped loops $\hat{x} \in \hat{X}$ are all capped $1$-periodic orbits of $\psi^{\hat{X}}$ by construction, it suffices to fix some $\hat{x}=[x,w] \in \hat{X}$ and to compute the winding number of some vector in $T_{x(0)} S^2$ under the linearized flow of $\psi^{\hat{X}}$ along $x$. Suppose without loss of generality that $\mu(\hat{x};H)=1$ and let $u \in \widetilde{\mathcal{M}}(\hat{x},\hat{y};H,J)$, for $\hat{y} \in \hat{X}$. We have by definition that for any $\theta \in S^1$, $\psi^{\hat{X}}_{\theta}(u(s,t))= u(s,t+\theta)$, and so if we let $\xi_u \in \Gamma(x^*T \Sigma)$ be the negative asymptotic eigenvector of $u$ then the asymptotic estimates of Theorem \ref{PosAsymp} imply that $(D \psi^{\hat{X}}_{\theta})(\xi(0)) = \xi(\theta)$ for all $\theta  \in S^1$. Since $\xi_u \in \Gamma(x^*T \Sigma)$ is eigenvector of $A_{x,J}$ with winding number $0$ relative the capping $w$, it follows that the linearized winding of $\psi^{\hat{X}}$ along $x$ relative the capping $w$ is even, and so $\psi^{\hat{X}}$ is a contractible loop.
\end{proof}

Note that $\mathcal{F}^{\hat{X}}$ is everywhere transverse to the fibers $\lbrace t \rbrace \times \Sigma$ of $S^1 \times \Sigma$, and so may be viewed as an $S^1$-family of (singular) foliations on $\Sigma$. Let us write $\mathcal{F}^{\hat{X}}_t$ for the foliation obtained on $\Sigma$ by intersecting $\mathcal{F}^{\hat{X}}$ with $\lbrace t \rbrace \times \Sigma$.
\begin{theorem}\label{thm:LeCalvezFoliationThm}
Let $(H,J)$ be a Floer regular pair, $\hat{X} \in murm(H)$, then the orbits of the isotopy $(\psi^{\hat{X}})^{-1} \circ \phi^H$ are positively transverse to the foliation $\mathcal{F}^{\hat{X}}_0$.
\end{theorem}
\begin{proof}
Writing $\psi=\psi^{\hat{X}}$, observe that the vector field $(Z_t)_{t \in [0,1]}$ which generates the isotopy $\psi^{-1} \circ \phi^H$ is easily computed via the chain rule as $(Z_t)_{u(s,0)}=(\psi_t^{-1})_* (X^H_t- \partial_t u)_{u(s,t)}$. Note moreover that the definition of $\psi_t$ implies that
\begin{align*}
(\psi_t)_*(\partial_s u)_{u(s,0)})&=(\partial_s u)_{u(s,t)} 
\end{align*}
for all $u \in \mathfrak{M}^{\hat{X}}$ and all $(s,t) \in \R \times S^1$. Consequently, because for all $t \in S^1$ $\psi_t$ is an orientation-preserving diffeomorphism, we see that
\begin{align*}
sgn \; \omega_{u(s,0)}(Z_t,\partial_s u) = sgn \;(\psi_t^* \omega)(Z_t,\partial_s u)&= sgn \; \omega_{u(s,t)}( X^H_t- \partial_t u,\partial_s u)
\\ &= sgn \; \omega_{u(s,t)}(-J_t \partial_s u,\partial_s u)
\end{align*}
from which the claim follows.
\end{proof}
Theorem \ref{MainThm2} is an immediate consequence of Proposition \ref{IdentificationOfModuliSpaces}, Proposition \ref{prop:PropsOfPsi}, and the preceding theorem.

\subsection{Consequences for the structure of Hamiltonian isotopies}\label{Sec: Consequences2}
For any Hamiltonian $H$, we write $H^{\sharp m}:= H \# \ldots \# H$ for the $m$-fold concatenated Hamiltonian which generates $(\phi^H_t)^{\sharp m}$. For $\hat{x}=[x(t),w(se^{2 \pi \i t})] \in \cL{\Sigma}$, we write $\hat{x}^{\sharp m}:= [x(mt),w(se^{2 \pi \i mt})]$, and for $\hat{X}=\lbrace \hat{x}_i \rbrace_{i=1}^k$ we write $\hat{X}^{\sharp m}:= \lbrace \hat{x}_i^{\sharp m} \rbrace_{i=1}^k$. For a cylinder $u(s,t)$, write $u^{\sharp m}(s,t):=u(s,mt)$. 
\\ \\
For $\hat{\gamma} \in \widetilde{Per}_0(H^{\sharp m}) \setminus \hat{X}$ and $\hat{x} \in \hat{X}$, define $\ell_{\hat{\gamma}}(\hat{x}):= \frac{\ell(\hat{x}^{\sharp m},\hat{\gamma})}{m}$. For any distinct  $\hat{x}, \hat{y} \in \hat{X}$ and  $u \in C^\infty(\R \times S^1;\Sigma)_{\hat{x},\hat{y}}$ we may define $I_{\gamma}(u)$ to be the signed intersection number of (some transverse perturbation of) the submanifolds parametrized by $\check{u}(s,t)=(t,u(s,t)) \in S^1 \times \Sigma$ and the multi-section $t \mapsto \Gamma(t):=\lbrace (t,\gamma(\frac{t+k}{m})): k=0,1,\ldots,m-1 \rbrace \subset S^1 \times \Sigma$ (with the obvious induced orientations). By noting that such intersections are in bijective (and sign-preserving) correspondence with the intersections of $\check{\gamma}$ and $\check{u}^{\sharp m}$, it is easy to deduce from the definition of the homological linking number that $I_{\gamma}(u) = m(\ell_{\hat{\gamma}}(\hat{y}) - \ell_{\hat{\gamma}}(\hat{x}))$ for any choice of capping $\hat{\gamma}$ of $\gamma$. The fact that $\check{X}^H=\partial_t \oplus X^H$ is positively transverse to the regular leaves of $\mathcal{F}^{\check{X}}$ implies that for every $\gamma \in Per_0(H^{\sharp m})$, $I_{\gamma}(u) \geq 0$ for every $u \in \mathcal{M}(\hat{x},\hat{y};H,J)$, and distinct $\hat{x}, \hat{y} \in \hat{X}$, with equality only if the multi-section $\Gamma$ is disjoint from $\check{u}$.
We obtain as an immediate corollary:

\begin{cor}\label{cor:LinkingIncreasesIffOrbitPassesThroughFoliatedSector}
Let $(H,J)$ be non-degenerate, $\hat{X} \in murm(H)$ and let $\hat{\gamma} \in \Per{H^{\sharp m}}$ for some $m \in \Z_{>0}$, then $I_{\gamma}(u)$ is non-negative for every $u \in \widetilde{\mathcal{M}}(\hat{x},\hat{y};H,J)$ with $\hat{x}, \hat{y} \in \hat{X}$ distinct. Moreover, for every distinct $\hat{x}, \hat{z} \in \hat{X}$, $\ell_{\hat{\gamma}}(\hat{x}) < \ell_{\hat{\gamma}}(\hat{z})$ if and only if $\check{\gamma}(t) \in W(\hat{x},\hat{z})$ for some $t \in S^1$. 
\end{cor}

\begin{definition}
For a Hamiltonian $H$, we will say that a capped braid $\hat{X} \subseteq \Per{H}$ is \textbf{strongly linking} if for any $\hat{\gamma} \in \Per{H}$, $\ell(\hat{\gamma},\hat{x}) =0$ for all $\hat{x} \in \hat{X}$ implies that $\hat{\gamma} \in \hat{X}$.
Denote by $usl(H)$ the collection of all $\hat{X} \subseteq \Per{H}$ such that $\hat{X}$ is both unlinked and strongly linking.
\end{definition}
Clearly, $usl(H) \subseteq mu(H)$. The following theorem says that if $\hat{X} \in murm(H)$, then any (capped) periodic orbit of arbitrarily large period of $H$ which is not itself an iterate of an orbit in $\hat{X}$ must link with some iterate of some $\hat{x} \in \hat{X} \in murm(H)$.

\begin{theorem}\label{Thm: murm implies mu}
Let $H$ be non-degenerate. If $\hat{X} \in murm(H)$, then $\hat{X}^{\sharp m} \in usl(H^{\sharp m})$ for all $m \in \Z_{m >0}$. In particular, every $\hat{X} \in murm(H)$ is maximally unlinked as a subset of $\Per{H}$.
\end{theorem}

\begin{proof}
Let $\hat{X} \in murm(H)$, fix some $J$ such that $(H,J)$ is Floer regular, and suppose for a contradiction that $\hat{\gamma} \in \Per{H^{\sharp k}} \setminus \hat{X}^{\sharp m}$ but $\ell(\hat{x}^{\sharp m},\hat{\gamma})=0$ for all $\hat{x}^{\sharp m} \in \hat{X}^{\sharp m}$. We may in fact suppose that $\hat{\gamma} \not \in \pi_2(\Sigma) \cdot \hat{X}^{\sharp m}$, since if $\hat{\gamma}=A \cdot \hat{x}$ for some $A \in \pi_2(\Sigma)$, then for any $\hat{y}^{\sharp m} \in \hat{X}^{\sharp m}$, $\hat{x}^{\sharp m} \neq \hat{y}^{\sharp m}$, Proposition \ref{LinkingProp} implies $\ell(\hat{\gamma},\hat{y}^{\sharp m})=\ell(\hat{\gamma},\hat{y}^{\sharp m}) + \frac{c_1(A)}{2}=\frac{c_1(A)}{2}$, since $\hat{X}$ is unlinked, which easily implies that $\hat{X}^{\sharp m}$ is unlinked as well.
\par
So we may as well assume that $\gamma \neq x$ for any $\hat{x}^{\sharp m} \in \hat{X}^{\sharp m}$. Since $\mathcal{F}^{\hat{X}}$ foliates $S^1 \times \Sigma$, it is necessary that $\check{\gamma}(0) \in W(\hat{x},\hat{y})$ for some $\hat{x}, \hat{y} \in \hat{X}$ and so by Corollary  \ref{cor:LinkingIncreasesIffOrbitPassesThroughFoliatedSector} implies that $\ell(\hat{x}^{\sharp m},\hat{\gamma}) < \ell(\hat{y}^{\sharp m}, \hat{\gamma})$, a contradiction.
\end{proof}

\subsubsection{Comparison to Le Calvez's theory of transverse foliations}\label{Sec: Calvez}
\begin{definition}
Let $I=(\phi_t)_{t \in [0,1]}$ be an isotopy of homeomorphisms based at the identity. We define
\begin{align*}
Fix(I)&:= \bigcap_{t \in [0,1]} Fix(\phi_t),
\end{align*}
and we say that $I$ is a \textbf{maximal isotopy} if for every $x \in Fix(\phi_1)$, the loop $t \mapsto \phi_t(x)$ is not contractible in $\Sigma \setminus Fix(I)$.
\end{definition}
\begin{definition}[cf. \cite{LT18}]
Given a continuous oriented singular $1$-dimensional foliation $\mathcal{F}$ of $\Sigma$, a continuous path $\gamma: [0,1] \rightarrow \Sigma$ is said to be \textbf{positively transverse} to $\mathcal{F}$ if its image is disjoint from $Sing(\mathcal{F})$ and for each $t_0 \in [0,1]$, there exists an orientation preserving homeomorphism of a neighbourhood of $\gamma(t_0)$ to a neighbourhood of $0 \in \R^2$ which sends $\mathcal{F}$ to the standard vertical foliation of $\R^2$, oriented downward, and sends $\gamma$ to a map whose $x$-coordinate is increasing in a neighbourhood of $t_0$. 
\end{definition} 

In \cite{LeC05}, Le Calvez developed a theory which associates to any maximal isotopy $I=(\phi_t)_{t \in [0,1]}$, an oriented singular continuous $1$-dimensional foliation $\mathcal{G}^{I}$ on $\Sigma$ having singular points on precisely the points $x \in Fix(I)$, and moreover having the property that the dynamics of $I$ are \textit{homotopically positively transverse} to the leaves of $\mathcal{G}^I$, in the following sense.

\begin{definition}
We say that an isotopy $I$ \textbf{homotopically positively transverse} to an oriented singular continuous foliation $\mathcal{G}$ if, for all $x \in \Sigma \setminus Fix(I)$, the path $x \mapsto \phi_t(x)$ is homotopic relative endpoints inside of $\Sigma \setminus Fix(I)$ to a path which is positively transverse to $\mathcal{G}$.
\end{definition}
\begin{remark}
Le Calvez and those working with his theory use the term `positively transverse' rather than `homotopically positively transverse' as we have used here. We introduce this term here merely to disambiguate between the other notion of positive transversality (for smooth isotopies) which we have already employed in this paper.
\end{remark}
The reader will surely notice the similarity of Le Calvez's result to the results expounded in the previous section. Let us compare them more closely in the case where $\phi_1$ is a smooth, non-degenerate Hamiltonian diffeomorphism.
\par
As we have shown, to any non-degenerate pair $(H,J)$ and any $\hat{X} \in murm(H)$, we may associate a new isotopy $I^{\hat{X}}:=((\psi^{\hat{X}}_t)^{-1}\circ \phi^H_t)_{t \in [0,1]}$ having
\begin{align*}
Fix(I) &= \bigcup_{\hat{x} \in \hat{X}} \lbrace x(0) \rbrace.
\end{align*}
The fact that $murm(H) \subseteq mu(H)$ (by Theorem \ref{Thm: murm implies mu}) implies that this isotopy is maximal, and we know that it is positively transverse (not simply homotopically positively transverse) to the foliation $\mathcal{F}^{\hat{X}}_0$ of $\Sigma$. This implies positive transversality in the sense of Le Calvez, but is --- at least \textit{prima facie} ---  a rather stronger condition, in that positive transversality in Le Calvez's sense allows us to homotope the orbits of the maximal isotopy, individually, in order to achieve positive transversality in the usual sense, whereas here, every non-trivial orbit of $((\psi^{\hat{X}}_t)^{-1} \circ \phi^H_t)$ is already positively transverse in the usual sense. Said another way, if $\hat{X}$ is a trivial capped braid, so that the isotopy $(\phi^H_t)_{t \in [0,1]}$ is already a maximal isotopy, then it is already homotopically positively transverse to  $\mathcal{F}^{\hat{X}}_0$ because $(\psi^{\hat{X}})^{-1}$ provides us with a choice of homotopy for each orbit $t \mapsto \phi^H_t(x)$ in such a way that these homotopies fit together in a coherent way so as to be induced by composition with a contractible loop. Thus, the transverse singular foliations constructed by Floer theory provide instances of Le Calvez-type foliations which enjoy a certain sort of `uniform' homotopic transversality which makes them particularly nice. We also remark that the foliations constructed by Le Calvez in \cite{LeC99} for arbitrary diffeomorphisms of $T^2$ via generating function methods have this property, but it is unclear at the time of this writing if this can be achieved for the foliations constructed via the brick decomposition methods of \cite{LeC05}, which permit the development of the theory for $\phi \in Homeo_0(\Sigma)$ and $\Sigma$ an arbitrary surface. This raises the following potentially interesting question for the theory of Le Calvez-type foliations.
\begin{quest}\label{Quest: Calvez}
If $\phi \in Homeo_0(\Sigma)$ and $I=(\phi_t)_{t \in [0,1]}$ is a maximal isotopy ending at $\phi$, does there always exist a singular foliation $\mathcal{G}$ of Le Calvez-type together with a contractible loop $(\psi_t)_{t \in S^1} \subseteq Homeo(\Sigma)$ (based at the identity), such that for all $x \in \Sigma \setminus Sing(\mathcal{G})$, the paths $t \mapsto (\psi_t^{-1} \circ \phi_t)(x)$ are positively transverse to $\mathcal{G}$?
\end{quest}
In \cite{LeC05}, Le Calvez also observed that in the case of Hamiltonian homeomorphisms satisfying some minor niceness requirements of their fixed point sets which always applies in our setting, the foliations $\mathcal{F}^I$ constructed by his methods are always \textit{gradient-like}, in the sense that they admit discrete Lyapunov functions which are decreasing along the leaves of the foliation. These Lyapunov functions are, essentially, winding numbers with the maximally unlinked set of orbits used to construct the foliation. Thus, we can see the results of our previous section as recovering Le Calvez's theory in the smooth case, on the proviso that $\hat{X} \in murm(H)$, and offering certain refinements on the structure of these fixed points along with a certain strengthening of the transversality condition. 
\par
We should note, however, that in general $murm(H) \neq mu(H)$, and so even in the smooth case, Le Calvez's method provides foliations which do not arise via our methods. In fact, it is not difficult to show that for every $\hat{X} \in murm(H)$, the isotopy $t \mapsto I=((\psi^{\hat{X}}_t)^{-1}\circ \phi^H_t)_{t \in [0,1]}, t \in [0,1]$ is \textbf{torsion-low} in the sense of Yan in \cite{Ya18}. Conversely, it follows essentially directly from Proposition $1.5$ in \cite{Ya18} (together with remarking on the relationship between the winding number of the linearization at a fixed point and its Conley-Zehnder index --- see Remark $17$ in \cite{HRS16}, for instance) that if $I'=(\phi^H_t)_{t \in [0,1]}$ is a smooth torsion-low Hamiltonian isotopy with $\phi^H_1$ non-degenerate, then 
\begin{align*}
\hat{X}&:= \bigcup_{x \in Fix(I')} \lbrace [x,x] \rbrace \subseteq \Per{H}
\end{align*}
is maximally unlinked relative the Morse range. Consequently, the theory developed in this section recovers precisely the `torsion-low' portion of Le Calvez's theory from a Floer theoretic perspective.

\section{Appendices}

\subsection{Appendix to Section \ref{Ch: Construct chain maps}}\label{App: Construct Chain Maps App}
\subsubsection{A short fibration argument}
\begin{proposition}
Let $(M,\omega)$ be a symplectic manifold. The space
\begin{align*}
\mathcal{E} &:= \lbrace (m,v;H) \in TM \times C^{\infty}(M): X^H(m)=v  \rbrace
\end{align*} 
equipped with the canonical projection $\pi: \mathcal{E} \rightarrow TM$ is a locally trivial fibration.
\end{proposition}
\begin{proof}
Let $(m,v) \in TM$ be arbitrary, and let $\phi: (U,\omega) \rightarrow (\R^{2n},\omega_0)$ be a Darboux neighbourhood of $m \in M$ with $\phi(m)=0 $. Letting $\kappa >0$ be such that $\phi(U)=B^{2n}(\kappa)$ and fixing some $\epsilon >0$, we will construct a local section for $\mathcal{E}$ over the set $T \phi^{-1}(B^{2n}(\kappa - \epsilon))$ which implies the claim.
\par 
Consider the map
\begin{align*}
SG_0: C^\infty(\R^{2n}) &\rightarrow T_0 \R^{2n}= \R^{2n} \\
H &\mapsto X^H(0).
\end{align*}
$SG_0$ is a linear surjection, and so we may select a right inverse
\begin{align*}
C: T_0 \R^{2n} &\rightarrow C^{\infty}(\R^{2n}).
\end{align*}
We may suppose without loss of generality that in fact $C$ takes values in the set $C_0^{\infty}(B^{2n}(\epsilon))$ of functions with support compactly contained in $B^{2n}(\epsilon)$. We may define
\begin{align*}
\mathcal{C}_{loc}: B^{2n}(\kappa - \epsilon) \times \R^{2n} &\rightarrow C_0^{\infty}(B^{2n}(\kappa)) \\
(x,v) &\mapsto A_x(C(v)),
\end{align*} 
where
\begin{align*}
A_x: C^\infty(\R^{2n}) \rightarrow C^{\infty}(\R^{2n}) \\
H(p) &\mapsto H(p-x)
\end{align*}
is the pullback of affine translation by $x \in \R^{2n}$. It's clear that $\mathcal{C}_{loc}$ is smooth by construction. Writing $V = \phi^{-1}(B^{2n}(\kappa - \epsilon))$,
\begin{align*}
\mathcal{C}:=\phi^*(\mathcal{C}_{loc} \circ D \phi\vert_{TV}): TV &\rightarrow C^{\infty}(M)
\end{align*}
defines a smooth map such that if $G=\mathcal{C}(m',v')$ for $(m',v') \in TV$, then $X^{G}(m')=v'$. That is, $\mathcal{C}$ is a smooth right inverse for $\pi: \mathcal{E} \rightarrow TM$ on $TV$, as desired.
\end{proof}
Let $N$ be a smooth (possibly open) manifold, $f: N \rightarrow M$ a smooth map and $Y: N \rightarrow f^*TM$ a vector field along $f$. We define
\begin{align*}
\mathcal{H}(Y)&:= \lbrace (n,H) \in N \times C^{\infty}(M): X^H(f(n))=Y(n) \rbrace.
\end{align*}
\begin{cor}\label{Cor: loc triv fib}
$\pi: \mathcal{H}(Y) \rightarrow N$ is a locally trivial fibration.
\end{cor}
\begin{proof}
It is straightforward to check that $\mathcal{H}(Y)$ is nothing but the pullback bundle $Y^*\mathcal{E}$. Since $(\mathcal{E},\pi_{TM})$ is a locally trivial fibration, so is $\mathcal{H}(Y)$.
\end{proof}

\subsubsection{A proof of Lemma \ref{Lem: Local Pert Non-Reg to Reg} }\label{Subsec: Proof of Local pert to regular lemma}
Let $(H^\pm,J^\pm)$ be Floer regular and $\hat{x}^\pm \in \Per{H^\pm}$. We fix $(\mathcal{H},\mathbb{J}) \in \HJ(H^-,J^-;H^+,J^+)$. If $\mathcal{F}_{\mathcal{H},\mathbb{J}}(u)=0$, $\text{ind } (D \mathcal{F}_{\mathcal{H},\mathbb{J}})_u)=0$ then Lemma \ref{Lem: Local Pert Non-Reg to Reg} is immediately true if $\ker (D \mathcal{F}_{\mathcal{H},\mathbb{J}})_u =0$, by taking $u'=u$ and $\mathcal{H}'=\mathcal{H}$. By Corollaries \ref{Cor: Automatic Regularity for Odd orbits} and \ref{Automatic Codim 1 for Even orbits}, it suffices to consider the case where $u$ is such that 
\begin{align*}
\dim \ker (D \mathcal{F}_{\mathcal{H},\mathbb{J}})_{u} &= \text{corank \;} (D \mathcal{F}_{\mathcal{H},\mathbb{J}})_{u}=1.
\end{align*}
Let $\lambda \mapsto \mathcal{H}^\lambda$, $\lambda \in [-1,1]$ be a smooth path in $\mathcal{H}(H^-,H^+)$ such that $(\mathcal{H}^{\lambda},\mathbb{J})$ is a regular homotopy of Floer data for $\lambda= \pm 1$ and such that $\mathcal{H}^0=\mathcal{H}$. Our goal is to show that if the path $\mathcal{H}^\lambda$ is chosen appropriately generically, then, using only perturbations in a neighbourhood $U$ of $\im \tilde{u}$ which may be taken to avoid the graphs of finitely many other Floer-type cylinders, the space $\bigcup_{\lambda \in [-1,1]} \mathcal{M}(\hat{x}^-,\hat{x}^+;\mathcal{H}^\lambda,\mathbb{J})$ may be taken to be a non-empty $1$-manifold which contains $u$. This clearly implies that for any neighbourhood $\mathcal{O} \subseteq C^\infty(\R \times S^1;M)_{\hat{x}^-,\hat{x}^+}$, there exists $u' \in \mathcal{O}$ and $\mathcal{H}' \in \mathcal{H}(H^-,H^+)$ such that $(\mathcal{H}',\mathbb{J}) \in \HJ_{reg}$, $\supp \mathcal{H}'-\mathcal{H} \subseteq U$ and $\mathcal{F}_{\mathcal{H}',\mathbb{J}}(u')=0$.
\par 
Fix a neighbourhood $N \subseteq \R \times S^1 \times M$ of $\im \tilde{u}$, and let $U \subseteq N$ be an open set which is dense in $N$ and such that $U \cap \im \tilde{u}$ is dense in $\tilde{u}$.    
\par
Let us define $C^\infty_{\epsilon}(U;0)$ to be the set of compactly supported smooth functions $h \in C^\infty_0( (-1,1) \times U)$ such that $h^0 \equiv 0$, $h^\lambda(s,t,x)=h(\lambda,s,t,x)$ for $\lambda \in (-1,1)$ and which satisfy
\begin{align*}
\|h \|_{\epsilon} &:= \sum_{k=0}^\infty \epsilon_k \| h \|_{C^k} <\infty,
\end{align*}
where $\epsilon=(\epsilon_k)_{k=0}^\infty$ is some sequence of positive numbers which decreases sufficiently quickly at infinity so that $(C^\infty_{\epsilon}(U;0), \| \cdot \|_{\epsilon})$ is dense in the space of smooth functions on $U$ which satisfy $h^0 \equiv 0$. For $\hat{x}^\pm \in \Per{H^\pm}$ and $p >2$, let $W^{1,p}(\hat{x}^-,\hat{x}^+)$ be the usual Banach space completion of $C^{\infty}(\R \times S^1;M)_{\hat{x}^-,\hat{x}^+}$ and consider the infinite-dimensional vector bundle $\mathcal{E} \rightarrow W^{1,p}(\hat{x}^-,\hat{x}^+) \times C^\infty_\epsilon(U;0)$ where
\begin{align*}
\mathcal{E}&:= \lbrace (u,h,Y): (u,h) \in W^{1,p}(\hat{x}^-,\hat{x}^+) \times C^\infty_\epsilon(U;0), \; Y \in L^p(u^*TM) \rbrace, 
\end{align*}
and consider the universal moduli space
\begin{align*}
Z(\hat{x}^-,\hat{x}^+) &:= \lbrace (\lambda, u, h): \lambda \in (-1,1), h \in C^\infty_\epsilon(U;0), u \in \mathcal{M}(\hat{x}^-,\hat{x}^+;\mathcal{H} + h^\lambda, \mathbb{J}) \rbrace
\end{align*}
which is cut out by the intersection of the map
\begin{align*}
\sigma: (-1,1) \times  W^{1,p}(\hat{x}^-,\hat{x}^+) \times C^\infty_\epsilon(U;0) &\rightarrow \mathcal{E} \\
(\lambda,u,h) &\mapsto \sigma_{\lambda}(u,h) = \partial_s u +\mathbb{J}( \partial_t u - X^{\mathcal{H}+h^\lambda})
\end{align*}
with the zero section. By our hypothesis on the corank of $D \mathcal{F}_{\mathcal{H},\mathbb{J}}$ at $u$, and the density of compactly supported smooth functions in $L^p$-space, there exists $\xi \in L^p(u^*TM)\cap C_0^\infty(u^*TM)$ such that
\begin{align}\label{Eqn: Transversal path}
L^p(u^*TM) &= \im (D \mathcal{F}_{\mathcal{H},\mathbb{J}})_u \oplus \langle \xi \rangle.
\end{align}
It is then easy to see that we may suppose without loss of generality (up to perturbing the path $\lambda \mapsto \mathcal{H}^\lambda$ near $\lambda=0$) that 
\begin{align*}
(\frac{\partial \mathcal{H}^\lambda}{\partial \lambda})_{\lambda=0}(u)&=\mathbb{J} \xi.
\end{align*}
\begin{definition}
If $\lambda \mapsto \mathcal{H}^\lambda \in \mathcal{H}(H^-,H^+)$ is such that $(\mathcal{H}^{\lambda},\mathbb{J})$ is a regular homotopy of Floer data for $\lambda= \pm 1$, $\mathcal{H}^0=\mathcal{H}$, and 
\begin{align*}
L^p(u^*TM) &= \im (D \mathcal{F}_{\mathcal{H},\mathbb{J}})_u \oplus \langle -\mathbb{J}(\frac{\partial \mathcal{H}^\lambda}{\partial \lambda})_{\lambda=0}(u) \rangle,
\end{align*}
then we will say that the path $(\mathcal{H}^{\lambda})_{\lambda \in [-1,1]}$ is \textbf{transverse to $u$ at $0$}.
\end{definition}
The main point is the following proposition.
\begin{proposition}
If $(\mathcal{H}^\lambda)_{\lambda \in [-1,1]}$ is transverse to $u$ at $0$, then there are neighbourhoods $\Lambda \subseteq [-1,1]$ of $0$, $\mathcal{O} \subseteq W^{1,p}(\hat{x}^-,\hat{x}^+)$ of $u$ and $\mathcal{V} \subseteq C^\infty_\epsilon(U;0)$ of $0$ such that the restriction
\begin{align*}
\sigma\vert_{\Lambda \times \mathcal{O} \times \mathcal{V}}: \Lambda \times \mathcal{O} \times \mathcal{V} &\rightarrow \mathcal{E}
\end{align*}
is transverse to the zero section of $\mathcal{E}$.
\end{proposition}
\begin{proof}
Since the set of surjective Fredholm maps is open and the Floer operator is continuously differentiable, it will suffice to show that the vertical differential of $\sigma$ is surjective at $(0,u,0) \in [-1,1] \times W^{1,p}(\hat{x}^-,\hat{x}^+) \times C^\infty_0(U;0)$. To see this, compute the vertical differential at a point $(\lambda_0,v,h_0) \in Z(\hat{x}^-,\hat{x}^+)$
\begin{align*}
(D \sigma)_{(\lambda_0,v,h_0)}(a,Y,h) &=D (\mathcal{F}_{\mathcal{H}+h_0,\mathbb{J}})_v(Y) -\mathbb{J}(aX^{\partial_{\lambda}(\mathcal{H} + h_0)^{\lambda_0}} +X^{h^{\lambda_0}}).
\end{align*}
Note then that if $\lambda_0=0$ we have that $X^{h^{\lambda_0}}=0$, since $h \in C^\infty_{\epsilon}(U;0)$ and so $h^0=0$ by hypothesis. If we further impose that $v=u$ and $h=0$, we see that 
\begin{align*}
(D \sigma)_{(0,u,0)}(a,Y,h)&=D (\mathcal{F}_{\mathcal{H},\mathbb{J}})_u(Y)+ a \xi
\end{align*}
where
\begin{align*}
L^p(u^*TM) &= \im  (D\mathcal{F}_{\mathcal{H},\mathbb{J}})_u \oplus \langle \xi \rangle,
\end{align*}
since $\lambda \mapsto \mathcal{H}^\lambda$ is transversal to $u$ at $0$ by hypothesis. This proves the claim.
\end{proof}
Finally, we claim
\begin{proposition}
If $(\mathcal{H}^\lambda)_{\lambda \in [-1,1]}$ is transverse to $u$ at $0$, then there are neighbourhoods $\mathcal{O} \subseteq W^{1,p}(\hat{x}^-,\hat{x}^+)$ of $u$ and $\mathcal{V} \subseteq C^\infty_\epsilon(U;0)$ of $0$ such that the restriction
\begin{align*}
\sigma\vert_{(-1,1) \times \mathcal{O} \times \mathcal{V}}: (-1,1) \times \mathcal{O} \times \mathcal{V} &\rightarrow \mathcal{E}
\end{align*}
is transverse to the zero section of $\mathcal{E}$.
\end{proposition}
\begin{proof}
This basically follows from the previous proposition in combination with the usual arguments used to establish transversality in this setting, with the only caveat being that we must show that away from $\lambda=0$, transversality can be obtained via perturbations with support in $U$. This is essentially a consequence of the usual proof in any case, but we will sketch out the argument for the convenience of the reader. 
\par
The previous proposition gives us open neighbourhoods $\Lambda \subseteq [-1,1]$ of $0$, $\mathcal{O} \subseteq W^{1,p}(\hat{x}^-,\hat{x}^+)$ of $u$ and $\mathcal{V} \subseteq C^\infty_\epsilon(U;0)$ of $0$ such that the restriction $\sigma\vert_{\Lambda \times \mathcal{O} \times \mathcal{V}}$ is transverse to the zero section of $\mathcal{E}$. Up to shrinking $\mathcal{O}$, we may suppose that if $v \in \mathcal{O}$, then $\im \tilde{v} \subseteq \text{int } N$. We will suppose that this is the case from now on. In particular, any $v \in \mathcal{O}$ has the property that $U \cap \im \tilde{v}$ is dense in $\im \tilde{v}$. Clearly, it will suffice to show that $\sigma$ is transverse to the zero section of $\mathcal{E}$ at any point $(\lambda_0,v,h_0)$ with $\lambda_0 \in (-1,1) \setminus \lbrace 0 \rbrace$. Indeed as the vertical differential of $\sigma$ is
\begin{align*}
(D \sigma)_{(\lambda_0,v,h_0)}(a,Y,h) &=D (\mathcal{F}_{\mathcal{H}+h_0,\mathbb{J}})_v(Y) -\mathbb{J}(aX^{\partial_{\lambda}(\mathcal{H} + h_0)^{\lambda_0}} +X^{h^{\lambda_0}}),
\end{align*}
we may consider its restriction to the subspace $W^{1,p}(v^*TM) \times C^\infty_\epsilon(U;0)$, which is the tangent space of the fiber of $(-1,1) \times W^{1,p}(\hat{x}^-,\hat{x}^+) \times C^\infty_\epsilon(U;0)$ over $\lambda_0$. We claim that the map $F:= (D \sigma)_{(\lambda_0,v,h_0)}\vert_{W^{1,p}(v^*TM) \times C^\infty_\epsilon(U;0)}$,
\begin{align*}
F: W^{1,p}(v^*TM) \times C^\infty_\epsilon(U;0) &\rightarrow L^p(v^*TM) \\
(Y,h) &\mapsto D (\mathcal{F}_{\mathcal{H}+h_0,\mathbb{J}})_v(Y) -\mathbb{J} X^{h^{\lambda_0}}
\end{align*}
is surjective. That this is so follows along standard lines. If we suppose that the above map is not surjective, then for $q$ H\"{o}lder conjugate to $p$, there will exist some non-zero $Z \in L^q(v^*TM)$ which annihilates the image of $F$ under the usual duality pairing of $L^q(v^*TM)$ with the dual space to $L^p(v^*TM)$. A standard argument implies that $Z$ lies in the kernel of the adjoint of $(D \mathcal{F}_{\mathcal{H}+h_0,\mathbb{J}})_v$, and thus $Z$ is smooth by elliptic regularity. Since $Z$ is non-zero, we may view $Z$ as a vector field along $\tilde{v}$ by setting 
\begin{align*}
\tilde{Z}(s,t)= (0,0,Z(s,t)) \in T_{\tilde{v}(s,t)} \R \times S^1 \times M.
\end{align*}
If $Z \neq 0$, then there is some open set $W \subseteq \R \times S^1$  on which $Z$ is non-vanishing, and so $\tilde{Z}$ is non-vanishing along $\tilde{v}(W) \subseteq \R \times S^1 \times M$. But $U \cap \im \tilde{v}$ is open and dense in $\im \tilde{v}$ and so we may clearly take some $h \in C^\infty_\epsilon(U;0)$ such that $\supp h \cap \im \tilde{v} \subseteq \tilde{v}(W)$, which is not annihilated by $Z$. This is a contradiction, so it follows that $F$ is surjective.
\end{proof}
As a consequence of the above, the Sard-Smale theorem implies that there is set $\mathcal{R} \subseteq C^\infty_\epsilon(U;0)$ which is open and dense in a neighbourhood of $0$ such that 
\begin{align*}
\bigcup_{\lambda \in [-1,1]} \mathcal{M}(\hat{x}^-,\hat{x}^+; \mathcal{H}^\lambda + h^\lambda; \mathbb{J})
\end{align*}
is a manifold in a neighbourhood of $(0,u) \in  \mathcal{M}(\hat{x}^-,\hat{x}^+; \mathcal{H}^0; \mathbb{J})$, and the standard arguments in Floer theory then show that the dimension of this manifold is $1$.

\subsection{Proof of Proposition \ref{Prop: Basic properties of PSS-image invars Main}}\label{App: Proof of Basic Props of PSS-image invars}\label{App: New Spec Invar}
\begin{proposition}\label{Prop: Basic properties of PSS-image invars Appendix}
For any $H,K \in C^\infty(S^1 \times M)$ and any $\alpha \in QH_*(M,\omega) \setminus \lbrace 0 \rbrace$
\begin{enumerate}
\item if $r: [0,1] \rightarrow \R$ is smooth, then 
\begin{align*}
c_{im}(\alpha;H+r) &= c_{im}(\alpha;H) + \int_0^1 r(t) \;dt.
\end{align*}
\item $c_{im}(\psi_* \alpha; \psi_*H) = c_{im}(\alpha;H)$ for any symplectic diffeomorphism $\psi$.
\item $\vert c_{im}(\alpha;H) - c_{im}(\alpha;K)\vert \leq \| H-K \|_{L^{1,\infty}}$.
\item (Weak triangle inequality) $c_{im}(\alpha; H \#K) \leq c_{im}(\alpha;H) + c_{im}([M];K)$.
\end{enumerate}
\end{proposition}
\begin{proof}
As one might expect, the proofs of these properties follow largely the same lines as the proofs of the corresponding properties for the Oh-Schwarz spectral invariants with only very minor modifications to the details in those places where we must exhibit `infimizing sequences of PSS maps' which transfer some Morse cycles representing $\alpha$ to a sequence of cycles in $CF_*(H)$ whose action levels tend to the infimal possible value, in order to prove that the corresponding property holds for $c_{im}(\alpha;-)$.
\begin{enumerate}
\item  $H$ and $H+r$ induce the same Hamiltonian vector fields, and $(H,J)$ and $(H+r,J)$ induce the same Floer equations. Consequently, we have the canonical identification of Floer complexes
\begin{align}\label{Label: canonical ident from action shift}
CF_*(H,J)&=CF_*(H+r,J).
\end{align} 
Moreover, if $\mathcal{D}=(f,g;\mathcal{H},\mathbb{J}) \in PSS_{reg}(H,J)$, then we may view $\mathcal{H}$ as a $T$-adapted homotopy of Hamiltonians from $0$ to $H$. Hence, if 
\begin{align*}
\beta: \R &\rightarrow [0,1]
\end{align*} 
is a smooth non-decreasing function such that $\beta(s)=0$ for $s \in (-\infty,-T-1]$ and $\beta(s)=1$ for $s \in [T,\infty)$, then 
\begin{align*}
\mathcal{H}_\beta(s,t,x):= \mathcal{H}(s,t,x)+\beta(s)r(t)
\end{align*}
is a $(T+1)$-adapted homotopy from $0$ to $H+r$ such that the pair $(\mathcal{H}_{\beta},\mathbb{J})$ defines the same Floer equation as $(\mathcal{H},\mathbb{J})$, and therefore the map
\begin{align*}
\Phi^{PSS}_{\mathcal{D}_{\beta}}: QC_{*+n}(f,g) &\rightarrow CF_*(H+r,J)
\end{align*} 
induced by $\mathcal{D}_\beta:=(f,g;\mathcal{H},\mathbb{J}) \in PSS_{reg}(H +r,J)$ agrees with $\Phi^{PSS}_{\mathcal{D}}$ under the canonical identification of \ref{Label: canonical ident from action shift}. The desired statement follows.

\item If $\mathcal{D}=(f,g;\mathcal{H},\mathbb{J})$ and $\psi \in Symp(M,\omega)$, then by the naturality of the negative gradient flow equation with respect to push-forwards by diffeomorphisms, and the naturality of Floer's equation with respect to symplectomorphisms, we have that $\mathcal{D} \in PSS_{reg}(H,J)$ if and only if $\psi_*\mathcal{D}=(\psi_*f,\psi_*g;\psi_*H,\psi_*\mathbb{J}) \in PSS_{reg}(\psi_* H, \psi_*J)$, and the relevant PSS moduli spaces are identified via the map
\begin{align*}
\mathcal{M}(p,[x,w];\mathcal{D}) &\rightarrow \mathcal{M}(\psi(p),[\psi \circ x, \psi \circ w];\psi_*\mathcal{D}) \\
u &\mapsto \psi \circ u.
\end{align*}
Consequently, we see that for $\sigma_f \in QC_{*+n}(f,g)$
\begin{align*}
\Phi^{PSS}_{\mathcal{D}}(\sigma_f)&=\sigma \in CF_*(H,J)
\end{align*}
if and only if
\begin{align*}
\Phi^{PSS}_{\psi_* \mathcal{D}}(\psi_* \sigma_f)&= \psi_* \sigma \in CF_*(\psi_*H,\psi_*J),
\end{align*}
where $\psi_*(p \otimes e^A)= \psi(p) \otimes e^{\psi_*A}$ defines the $\Lambda_\omega$-linear map
\begin{align*}
\psi_*: C^{Morse}(f,g) \otimes \Lambda_{\omega} \rightarrow C^{Morse}(\psi_*f,\psi_*g) \otimes \Lambda_{\omega}
\end{align*}
on generators, and $\psi_*([x,w])=[\psi \circ x, \psi \circ w]$ defines the linear map 
\begin{align*}
\psi_*: CF_*(H,J) &\rightarrow CF_*(\psi_*H,\psi_*J)
\end{align*}
on generators. Finally, it's clear that $\sigma_f \in C^{Morse}(f,g)\otimes \Lambda_{\omega}$ represents $\alpha \in QH_*(M,\omega)$ if and only if $\psi_* \sigma_f$ represents $\psi_* \alpha \in QH_*(M,\omega)$. 
\item This is a direct consequence of Proposition \ref{Prop: Energy Bounds on Diff Of imSpec Invar}.
\item We fix Floer regular pairs $(H,J_1)$ and $(K,J_2)$. The fundamental point, established in Section 4 of \cite{Sc00} (see also Section $6.2$ of \cite{Oh05b}) is that if $(S_{2,1},j)$ denotes an model Riemannian surface with two negative cylindrical ends 
\begin{align*}
\phi^-_i: ((-\infty,0] \times S^1,j_0) &\rightarrow (Z_i^-,j) \subseteq S_{2,1}, \; i=1,2,
\end{align*}
and one positive cylindrical end 
\begin{align*}
\phi^+: ((-\infty,0] \times S^1,j_0) &\rightarrow (Z^+,j) \subseteq S_{2,1},
\end{align*}
then for any $\delta >0$, there exist smooth maps   
\begin{align*}
\mathcal{H}^S: S_{2,1} &\rightarrow C^\infty(M), \; \text{and} \\
\mathbb{J}^S: S_{2,1} &\rightarrow \mathcal{J}(M,\omega)
\end{align*}
such that
\begin{enumerate}
\item $(\mathcal{H}^S \circ \phi^-, \mathbb{J}^S \circ \phi_i^-)(s,t)$ equals $(H,J_1)$ for all $s$ sufficiently small if $i=1$ and equals $(K,J_2)$ for all such $s$ if $i=2$
\item $(\mathcal{H}^S \circ \phi^+, \mathbb{J}^S \circ \phi^+)(s,t)$ equals $(H \#K, J_3)$ for all $s$ sufficiently large, where $J_3 \in \mathcal{J}_\omega(M)$ is such that $(H \#K, J_3)$ is Floer regular.
\item For $[x_1,w_1] \in \Per{H}$, $[x_2,w_2] \in \Per{K}$ and $[x_3,w_3] \in \Per{H\#K}$, if $u: S_{2,1} \rightarrow M$ is any smooth map with $\lim_{s \rightarrow -\infty} (u \circ \phi^-_i)(s,t)=x_i$, $i=1,2$, $\lim_{s \rightarrow \infty} (u \circ \phi^+)(s,t)=x_3(t)$ and the map $S^2 \rightarrow M$ formed by gluing $w_1$ to $u$ along $Z_1^-$, $w_2$ to $u$ along $Z_2^-$ and $\bar{w}_3$ along $Z^+$ is a torsion element of $H_2(M;\Z)$, then
\begin{align*}
\mathcal{A}_{H \# K}([x_3,w_3]) &\leq \mathcal{A}_H([x_1,w_1]) + \mathcal{A}_K([x_2,w_2]) + \delta.
\end{align*}
\item If we denote by $n(\hat{x}_1,\hat{x}_2;\hat{x}_3)$ the count of elements in the zero-dimensional moduli spaces $\mathcal{M}(\hat{x}_1,\hat{x}_2;\hat{x}_3;\mathcal{H}^S,\mathbb{J}^S)$ which consist of smooth maps $u: S_{2,1} \rightarrow M$ satisfying the conditions of the previous point and which, moreover satisfy, in any conformal coordinates $(s,t)$ on $(S_{2,1},j)$, the $(\mathcal{H}^S,\mathbb{J}^S)$-Floer equations
\begin{align}\label{eq: Pair-of-pants Floer equation}
\partial_s u + \mathbb{J}^S(\partial_t u - X^{\mathcal{H}^S})&=0,
\end{align}
then the map
\begin{align*}
\Phi^S:CF_*(H,J_1) \otimes CF_*(K,J_2) &\rightarrow CF_*(H \#K,J_3) \\
\hat{x} \otimes \hat{y} &\mapsto \sum_{\mu(\hat{z})=\mu(\hat{x})+\mu(\hat{y})} n(\hat{x},\hat{y};\hat{z})\hat{z}  
\end{align*} 
defines a map which descends to the Pair-of-Pants product on Floer homology. 
\item If $\alpha^{\#}=\Phi^{PSS}(\alpha) \in HF_*(H)$ and $\beta^{\#}= \Phi^{PSS}(\beta) \in HF_*(K)$, then $(\Phi^S)_*(\alpha \otimes \beta)= \Phi^{PSS}(\alpha \ast \beta) \in HF_*(H \# K)$.
\end{enumerate}
The statement is then proven by fixing some $\epsilon >0$ and letting $\sigma_H \in CF_*(H,J_1)$ and $\sigma_K \in CF_*(K,J_2)$ be such that
\begin{align*}
\lambda_H(\sigma_H) &< c_{im}(\alpha;H) + \frac{\epsilon}{2}, \\
\lambda_K(\sigma_K)&< c_{im}(\beta;K) + \frac{\epsilon}{2},
\end{align*} 
and there exists PSS data $\mathcal{D}_H=(f_1,g_1;\mathcal{H},\mathbb{J}_1) \in PSS_{reg}(H,J_1)$ and $\mathcal{D}_K=(f_2,g_2;\mathcal{K},\mathbb{J}_2) \in PSS_{reg}(K,J_2)$ such that $\sigma_H \in \im \Phi^{PSS}_{\mathcal{D}_H}$ and $\sigma_K \in \im \Phi^{PSS}_{\mathcal{D}_K}$ (of course, the existence of such cycles is guaranteed by the definition of $c_{im}(\alpha;-)$). The previous points imply that $\sigma_{H \# K}:=\Psi^S(\sigma_H \otimes \sigma_K)$ represents the class $\alpha \ast \beta$ under the natural isomorphism $QM_{* +n}(M,\omega) \simeq HF_*(H \# K)$ and $\supp \Psi^S(\sigma_H \otimes \sigma_K)$ is composed of capped orbits with action bounded above by
\begin{align*}
\lambda_H(\sigma_H)+\lambda_K (\sigma_K)+\delta + \epsilon &= c_{im}(\alpha;H) + c_{im}([M];K) + \delta + \epsilon.
\end{align*}
If we were dealing with the Oh-Schwarz spectral invariants, then we could remark that $\epsilon$ and $\delta$ are arbitrary positive numbers and we would be done (and we would then have no need to use the additional hypothesis that $\beta = [M]$), but to establish the statement for $c_{im}$, we must exhibit some regular PSS data $\mathcal{D} \in PSS_{reg}(H \# K,J_3)$ such that $\sigma_{H \# K} \in \im \Phi^{PSS}_\mathcal{D}$. Note that the standard gluing theorems from Floer theory imply that we may glue the PSS data $\mathcal{D}_H$ and $\mathcal{D}_K$ along the negative cylindrical ends $Z^-_1$ and $Z^-_2$ respectively to obtain regular PSS data
\begin{align*}
\mathcal{D}_{H \#K}:=\mathcal{D}_H \#_{Z^-_1} \mathcal{D}_K \#_{Z^-_2} (\mathcal{H}^S,\mathbb{J}^S) \in PSS_{reg}(H \# K;J_3)
\end{align*}
having the property that we have the bijection of moduli spaces
\begin{align*}
\mathcal{M}(p, \hat{x};\mathcal{D}_{H}) \times \mathcal{M}(q,\hat{y}; \mathcal{D}_K) \times \mathcal{M}(\hat{x},\hat{y};\hat{z};\mathcal{H}^S,\mathbb{J}^S) &\rightarrow \mathcal{M}(p,\hat{z};\mathcal{D}_{H\#K}) \\
(u,v,w) &\mapsto \bar{u} \#_{Z^-_1} \bar{v} \#_{Z^-_2} w.
\end{align*}
\begin{remark}\label{Rmk: In proof rmk}
It is precisely here that we use the fact that $[\sigma_K]=\Phi^{PSS}_*[M]$. In general (that is, when $[\sigma_K]=\Phi^{PSS}_*\beta$ for $\beta \in QH_{*+n}(M,\omega) \setminus \lbrace 0 \rbrace$ not necessarily equal to $[M]$), the three moduli spaces above glue together to give 
\begin{align*}
 \mathcal{M}^{\beta}(p,\hat{z};\mathcal{D}_{H\#K}) &\subseteq  \mathcal{M}(p,\hat{z};\mathcal{D}_{H\#K}),
\end{align*}
which is the subspace of maps $u \in \mathcal{M}(p,\hat{z};\mathcal{D}_{H\#K})$ which additionally satisfy $u(s_0,t_0) \in \im \beta^\#$ for an appropriately chosen marked point $(s_0,t_0) \in \R \times S^1$ and for $\beta^\#: \cup \Delta^k \rightarrow M$ a smooth cycle composed of unstable manifolds of the Morse-Smale pair $(f_2,g_2)$ which represent $\beta$ in the quantum homology of $M$. When $\beta = [M]$, this condition is vacuous and we generically obtain
\begin{align*}
 \mathcal{M}^{[M]}(p,\hat{z};\mathcal{D}_{H\#K}) &= \mathcal{M}(p,\hat{z};\mathcal{D}_{H\#K}).
\end{align*}
\end{remark}
Returning to the proof, the above bijection of moduli spaces implies that 
\begin{align*}
\Phi^{PSS}_{H \# K}(\sigma_H)= &\Phi^S(\sigma_H \otimes \sigma_K)=\sigma_{H \# K},
\end{align*}
so that $\Psi^S(\sigma_H \otimes \sigma_K)$ lies in the image of some chain-level PSS map, and this proves the claim.
\end{enumerate}
\end{proof}
\begin{remark}
We note that it is precisely the point discussed in Remark \ref{Rmk: In proof rmk} that keeps us from being able to establish the full triangle inequality
\begin{align*}
c_{im}(\alpha \ast \beta; H \# K) &\leq c_{im}(\alpha;H) + c_{im}(\beta; K)
\end{align*}
for the PSS-image spectral invariant. Indeed, for any $\beta \neq [M]$, the maps $(u,v,w) \in \mathcal{M}(p, \hat{x};\mathcal{D}_{H}) \times \mathcal{M}(q,\hat{y}; \mathcal{D}_K) \times \mathcal{M}(\hat{x},\hat{y};\hat{z};\mathcal{H}^S,\mathbb{J}^S)$ will glue together to give a disk with origin lying in a cycle $\alpha^\#$ representing $\alpha$, an additional marked point lying in a cycle $\beta^\#$ as before and with boundary on the orbit $z \in Per_0(H \# K)$. In general, when $\beta \neq [M]$, we will have $\deg(\alpha) +n \neq \mu(\hat{z})$ and so the maps in the `glued moduli spaces'
\begin{align*}
\mathcal{M}^{glued}(p, \hat{x};q, \hat{y};\hat{x},\hat{y};\hat{z}) &\simeq \mathcal{M}(p, \hat{x};\mathcal{D}_{H}) \times \mathcal{M}(q,\hat{y}; \mathcal{D}_K) \times \mathcal{M}(\hat{x},\hat{y};\hat{z};\mathcal{H}^S,\mathbb{J}^S)
\end{align*}
which contribute to $\sigma_{H \# K}$ cannot be identified with a subset of the maps in the moduli space
\begin{align*}
&\mathcal{M}(p,\hat{z}; \mathcal{D}_{\mathcal{H} \#K})
\end{align*}
which contribute to $\Phi^{PSS}_{\mathcal{D}_{H \#K}}(\alpha^\# \ast \beta^\#)$, since the latter consist of maps $u: \R \times S^1 \rightarrow M$ which, roughly speaking, satisfy $\lim_{s \rightarrow - \infty} u(s,t) \in \im (\alpha \# \beta)^\#$. It is, however, plausible that a judicious choice of a different Morse-Smale pair for the PSS data $\mathcal{D}_{H \# K}$ (instead of the pair $(f_1,g_1)$) could be made so that $\mathcal{M}(p,\hat{z}; \mathcal{D}_{\mathcal{H} \#K})$ agrees with $\mathcal{M}^{glued}(p, \hat{x};q, \hat{y};\hat{x},\hat{y};\hat{z})$. Since the weaker form of the triangle inequality suffices for our purposes, however, we do not pursue this matter further here.
\end{remark}

\subsection{Stefan-Sussmann Foliations }\label{SingFolSection}
In this article, it is convenient to have some explicit language with which to speak about singular foliations. To that end, this appendix will present some elementary notions from the theory of \textit{Stefan-Sussmann foliations} (cf. \cite{De01} and the references therein) which will be suitable to our purposes.
\par
As in the non-singular theory, one may think of (singular) foliations as being partitions of the ambient space into integral submanifolds of some (generalized) distributions. With this in mind, let
\begin{align*}
\mathcal{G}^k(M) &\rightarrow M
\end{align*}
denote the \textit{$k$-Grassmannian} of $M$, having fiber $Gr(k,T_xM)$ over $x \in M$, and let 
\begin{align*}
\mathcal{G}^*(M) &:= \sqcup_{k=0}^n \mathcal{G}^k(M),
\end{align*}
where $n= \dim M$, denote the \textit{total Grassmannian} of $M$.
\begin{definition}
A \textbf{(generalized) distribution} on a manifold $M$ is a section
\begin{align*}
D: M &\rightarrow \mathcal{G}^*(M).
\end{align*}
A local section $X: M \rightarrow TM$ is said to \textbf{belong to $D$}  if $X(x) \in D(x)$ for all $x \in dom(X)$. The set of all smooth local sections $X \in \mathcal{X}_{loc}(M)$ belonging to $D$ is denoted by $\Delta_D$.
\end{definition}

\begin{definition}\label{def:SmoothGenDistrib}
A generalized distribution $D$ is said to be \textbf{smooth} if for every $x \in M$,
\begin{align*}
D(x) &= \text{span} \; \langle \; X(x) \; \rangle_{X \in \Delta_D}.
\end{align*}
\end{definition}
\begin{exmp*}
\begin{enumerate}
\item Let $Y \in \mathcal{X}(M)$ be a vector field. The assignment $D_Y(x):= \langle Y(x) \rangle$ defines a generalized distribution. Note that $D_Y$ is zero-dimensional wherever $Y$ vanishes and is one-dimensional otherwise, so $D_Y$ will not define a distribution in the usual sense unless $Y$ is non-vanishing. In this case, $\Delta_{D_Y}$ consists of smooth local vector fields which vanish at zeroes of $Y$ and which are otherwise parallel to $Y$ wherever $Y$ does not vanish.
\item Generalizing the preceding example, if $\lbrace Y_1, \ldots, Y_k \rbrace$ is a collection of smooth vector fields on $M$, then the distribution $D(x):= \langle Y_i(x) \rangle_{i=1}^k$ is a smooth distribution.
\end{enumerate}
\end{exmp*}
Naturally, we will want to define objects which integrate these generalized smooth distributions when they are appropriately integrable. To that end, following \cite{St74}, we shall define:
\begin{definition}
A \textbf{(smooth) $k$-leaf} of $M$ is a subset $L \subseteq M$ equipped with a differentiable structure $\sigma$ such that
\begin{enumerate} 
\item $(L,\sigma)$ is a connected $k$-dimensional immersed submanifold of $M$ and
\item for any continuous map $f: N \rightarrow M$ such that $f(N) \subseteq L$ and $N$ a locally connected topological space, we have that
\begin{align*}
f: N &\rightarrow (L, \sigma)
\end{align*}
is continuous.
\end{enumerate}
\end{definition}

\begin{definition}\label{Def: SS Foliation}
A \textbf{$(C^\infty)$-singular (Stefan-Sussmann) foliation} of $M$ is a partition $\mathcal{F}$ of $M$ into smooth leaves of $M$ such that for every $x \in M$, there exists a local smooth chart
\begin{align*}
\varphi: U  &\xrightarrow{\cong} \mathcal{O}(x) \subseteq M
\end{align*}
from $U \subseteq \R^n$ an open neighbourhood of $0 \in \R^n$, such that
\begin{enumerate}
\item $U = V \times W$ for $V$ an open neighbourhood of $0$ in $\R^k$ and $W$ an open neighbourhood of $0$ in $\R^{n-k}$, where $k$ is the dimension of the smooth leaf $L_x \in \mathcal{F}$ containing $x$.
\item $\varphi(0,0)=x$.
\item For any leaf $L \in \mathcal{F}$,
\begin{align*}
L \cap \varphi(U \times W) &= \varphi(U \times l),
\end{align*}
where $l:= \lbrace w \in W: \; \varphi(0,w) \in L \rbrace$.
\end{enumerate}
\end{definition}
\begin{exmp*}
\begin{enumerate}
\item If $\mathcal{F}$ is a foliation in the usual sense (see \cite{CC00}), then at any $x \in M$, any foliated chart for $\mathcal{F}$ based at $x$ provides a chart which satisfies the preceding properties.
\item If $M^n$ is a manifold of dimension $n$ and $Y \in \mathcal{X}(M)$ is a smooth vector field with isolated zeroes, then the partition of $M$ given by the orbits of the flow of $Y$ form a Stefan-Sussman foliation. Indeed, if $x \in M$ is a zero of $Y$, then the leaf through $0$ is simply a point and so any smooth chart from the atlas of $M$ centered at $x$ provides a smooth chart satisfying the above properties. If $Y(x) \neq 0$, then the leaf through $x$ will be one-dimensional, and there exists some small neighbourhood $\mathcal{O}(x)$ about $x$ containing only points where $Y$ is non-vanishing. Standard results from the theory of ODEs (see, for instance, Section 21 of \cite{AR67}) imply the existence of a `flow box', ie. a smooth diffeomorphism $\psi: U \times V \rightarrow \mathcal{O}(x)$, with $U \subset \R$, $V \subset \R^{n-1}$ both neighbourhoods of $0$ such that $\psi(0,0)=x$ and $\psi_* \partial_{x_1} = Y$ (here $x_1$ is the coordinate on $U$). Such a chart satisfies the properties enumerated above.
\end{enumerate}
\end{exmp*}

\begin{definition}
A smooth generalized distribution $D$ is said to be \textbf{integrable} if for every $x \in M$ there exists an immersed submanifold $L \subseteq M$, such that
\begin{enumerate}
\item $x \in L$, and
\item $T_y L \subseteq D(y)$ for all $y \in L$.
\end{enumerate}
Such an immersed submanifold is called an \textbf{integral submanifold of $D$}.
\end{definition}
\begin{exmp*}
\begin{enumerate}
\item The distribution $D(x):= \langle Y(x) \rangle$ generated by any smooth vector field $Y \in \mathcal{X}(M)$ is integrable with maximal integral submanifolds given by the orbits of $Y$.
\item The distribution $D(x):= \langle Y_1(x), \ldots, Y_k(x) \rangle$ generated by any collection of smooth point-wise linearly independent vector fields $Y_1, \ldots, Y_k \in \mathcal{X}(M)$ is integrable if and only if $D(x)$ is involutive (ie. $[Z_1,Z_2] \in \Delta_D$ for all $Z_1, Z_2 \in \Delta_D$) by Frobenius' integrability theorem.
\item (Constructing integrable generalized distributions via submersions) Let $M$ be a smooth manifold equipped with an integrable generalized distribution, and $f: M \rightarrow N$ a smooth submersion such that $D'(f(x)):= Df(D(x))$ is constant along fibers of $f$, ie. $Df(D(x))=Df(D(y))$ for all $y \in f^{-1}(x)$. Suppose that every maximal integral submanifold $L$ of $D$ admits an immersed submanifold $S=S_L \subseteq L$ on which which $f$ restricts to an immersion, and which satisfies $f(S)=f(L)$, then the distribution $D'$ on $N$ is integrable.
\end{enumerate}
\end{exmp*}

The main point for us is the following theorem (due to \cite{St74}).
\begin{theorem}
If $D$ is a smooth integrable generalized distribution and $\mathcal{F}_D$ is the partition of $M$ formed by taking the collection of maximal connected integral submanifolds of $D$, then $\mathcal{F}_D$ is a smooth singular Stefan-Sussmann foliation.
\end{theorem}
For a singular foliation $\mathcal{F}$, we let
\begin{align*}
d(-,\mathcal{F}): M &\rightarrow \Z_{\geq 0} \\
x &\mapsto \dim L_x
\end{align*}
denote the function which keeps track of the dimension of the leaf of $\mathcal{F}$ passing through $x \in M$. It's not hard to see that $d(-,\mathcal{F})$ is lower semi-continuous.
\begin{definition}
A smooth singular foliation $\mathcal{F}$ is said to be of \textbf{dimension $k$}, where
\begin{align*}
k &= \max_{x \in M} d(x,\mathcal{F}).
\end{align*}
A smooth singular foliation of an $n$-dimensional manifold will be said to be of \textbf{codimension $n-k$}.
\end{definition}
\begin{definition}
For a dimension $k$ smooth singular foliation, we define the \textbf{domain} of $\mathcal{F}$ to be
\begin{align*}
dom(\mathcal{F})&:= \lbrace x \in M: d(x, \mathcal{F}) = k \rbrace,
\end{align*}
while we define the \textbf{singular set} of $\mathcal{F}$ to be
\begin{align*}
sing(\mathcal{F})&:= M \setminus dom(\mathcal{F}).
\end{align*}
A leaf of $\mathcal{F}$ is said to be \textbf{regular} if it is of maximal dimension, otherwise it is said to be \textbf{singular}.
If $M$ is an oriented manifold, then $\mathcal{F}$ is said to be \textbf{oriented} if every regular leaf of $\mathcal{F}$ is in addition equipped with an orientation, and the local charts about points on the regular leaves may be taken to come from an atlas of oriented charts for $M$.
\end{definition}
\begin{exmp*}
\begin{enumerate}
\item If $\mathcal{F}$ is a $k$-dimensional foliation (in the usual sense) of $M$, then $\mathcal{F}$ is a smooth singular foliation of dimension $k$, with $dom(\mathcal{F})=M$ and $sing(\mathcal{F})=\emptyset$. If $M$ is oriented, then $\mathcal{F}$ may be equipped with the structure of an oriented singular foliation if and only if it can be equipped with the structure of an oriented foliation in the usual sense.
\item $\mathcal{F}$ is the singular foliation induced by the orbits of a smooth vector field $Y \in \mathcal{X}(M)$, then $\mathcal{F}$ is of dimension $1$. $\dom(\mathcal{F})$ consists of the points $x \in M$ where $Y(x) \neq 0$, while $sing(\mathcal{F})$ consists of the zero locus of $Y$. The regular leaves of $\mathcal{F}$ are the non-stationary orbits of $Y$. If $M$ is oriented, then $\mathcal{F}$ may obviously be taken to be oriented if we equip then non-stationary orbits with the orientation induced by then vector field $Y$.
\end{enumerate}
\end{exmp*}


\bibliographystyle{abbrv}
\bibliography{HamFloerTheoryBib}


\end{document}